\documentclass{amsart}
\usepackage[margin=1in]{geometry}

\usepackage{tikz}
\usepackage{physics}
\usepackage{mathdots}
\usepackage{yhmath}
\usepackage{cancel}
\usepackage{siunitx}
\usepackage{array}
\usepackage{multirow}
\usepackage{wrapfig}
\usepackage[font=small]{caption}
\usepackage[all]{xy}
\usepackage{amscd}
\usepackage{gensymb}
\usepackage{tabularx}
\usepackage{extarrows}
\usepackage{booktabs}
\usetikzlibrary{fadings}
\usetikzlibrary{patterns}
\usetikzlibrary{shadows.blur}
\usetikzlibrary{shapes}
\usepackage{amsmath}
\usepackage{amsthm}
\usepackage{amssymb}
\usepackage{mathtools}
\usepackage{soul}
\usepackage{color}

\usepackage{ytableau}
\usepackage{shuffle}

\newtheorem{thm}{Theorem}[section]

\newtheorem{prop}[thm]{Proposition}
\newtheorem{cor}[thm]{Corollary}
\newtheorem{lem}[thm]{Lemma}

\theoremstyle{definition}
\newtheorem{defn}[thm]{Definition}
\newtheorem{ex}[thm]{Example}
\newtheorem{rem}[thm]{Remark}

\newcommand{\rperp}{ \mathrel{
\begin{tikzpicture}[x=1pt,y=1pt,yscale=-.25,xscale=.25]
\draw[line width=.40pt]    (10,10) -- (20,10) ;
\draw[line width = .40pt]    (20,20) -- (0,20) ;
\draw[line width =.40pt]    (10,0) -- (10,20) ;
\end{tikzpicture}}}

\newlength\cellsize 
\savebox2{%
\begin{picture}(12,12)
\put(0,0){\line(1,0){12}}
\put(0,0){\line(0,1){12}}
\put(12,0){\line(0,1){12}}
\put(0,12){\line(1,0){12}}
\end{picture}}
\newcommand\cellify[1]{\def\thearg{#1}\def\nothing{}%
\ifx\thearg\nothing
\vrule width0pt height\cellsize depth0pt\else
\hbox to 0pt{\usebox2\hss}\fi%
\vbox to 12\unitlength{
\vss
\hbox to 12\unitlength{\hss$#1$\hss}
\vss}}
\newcommand\tableau[1]{\vtop{\let\\=\cr
\setlength\baselineskip{-16000pt}
\setlength\lineskiplimit{16000pt}
\setlength\lineskip{0pt}
\halign{&\cellify{##}\cr#1\crcr}}}
\savebox3{%
\begin{picture}(12,12)
\put(0,0){\line(1,0){12}}
\put(0,0){\line(0,1){12}}
\put(12,0){\line(0,1){12}}
\put(0,12){\line(1,0){12}}
\end{picture}}
\newcommand\expath[1]{%
\hbox to 0pt{\usebox3\hss}%
\vbox to 12\unitlength{
\vss
\hbox to 12\unitlength{\hss$#1$\hss}
\vss}}

\newlength\ccellsize 
\savebox4{%
\begin{picture}(18,18)
\put(0,0){\line(1,0){18}}
\put(0,0){\line(0,1){18}}
\put(18,0){\line(0,1){18}}
\put(0,18){\line(1,0){18}}
\end{picture}}
\newcommand\ccellify[1]{\def\thearg{#1}\def\nothing{}%
\ifx\thearg\nothing
\vrule width0pt height\ccellsize depth0pt\else
\hbox to 0pt{\usebox4\hss}\fi%
\vbox to 18\unitlength{
\vss
\hbox to 18\unitlength{\hss$#1$\hss}
\vss}}
\newcommand\bigtableau[1]{\vtop{\let\\=\cr
\setlength\baselineskip{-16000pt}
\setlength\lineskiplimit{16000pt}
\setlength\lineskip{0pt}
\halign{&\ccellify{##}\cr#1\crcr}}}
\savebox5{%
\begin{picture}(18,18)
\put(0,0){\line(1,0){18}}
\put(0,0){\line(0,1){18}}
\put(12,0){\line(0,1){18}}
\put(0,12){\line(1,0){18}}
\end{picture}}
\newcommand\cexpath[1]{%
\hbox to 0pt{\usebox5\hss}%
\vbox to 18\unitlength{
\vss
\hbox to 18\unitlength{\hss$#1$\hss}
\vss}}

\def\shuffle{\begin{picture}(16,5)(-2,0) 
\put(0,0){\line(0,1){5}}
\put(5,0){\line(0,1){5}}
\put(10,0){\line(0,1){5}}
\put(0,0){\line(1,0){10}}
\end{picture}}
\def\Qshuffle{\stackrel{Q}{\shuffle}}
\title[Colored Dual Immaculates]{A Generalization of the Dual Immaculate Quasisymmetric Functions in Partially Commutative Variables}
\author[S. Daugherty]{Spencer Daugherty} \thanks{sdaughe@ncsu.edu. Department of Mathematics, North Carolina State University, USA}

\usepackage[backend=bibtex, maxbibnames=10, doi=false, url=false, isbn=false]{biblatex}
\begin{document}

\maketitle

%%%%%%%%%%%% abstract %%%%%%%%%%%%%%%%%%%
\begin{abstract}

We define a new pair of dual bases that generalize the immaculate and dual immaculate bases to the colored algebras $QSym_A$ and $NSym_A$. The colored dual immaculate functions are defined combinatorially via tableaux, and we present results on their Hopf algebra structure, expansions to and from other bases, and skew functions. For the colored immaculate functions, defined using creation operators, we study expansions to and from other bases and provide a right Pieri rule. This includes a combinatorial method for expanding colored immaculate functions into the colored ribbon basis that specializes to a new analogous result in the uncolored case. We use the same methods to define colored generalizations of the row-strict immaculate and row-strict dual immaculate functions with similar results. 
\end{abstract}

\section{Introduction}

The quasisymmetric functions, introduced by Gessel \cite{gessel}, and the noncommutative symmetric functions, introduced by Gelfand, Krob, Lascoux, Leclerc, Retakh, and Thibon \cite{gelfand}, are generalizations of the symmetric functions with rich theory and importance in algebraic combinatorics.  Their algebras, $QSym$ and $NSym$, are dual Hopf algebras that also appear in representation theory, algebraic geometry, and category theory. A significant amount of work has been done to find quasisymmetric or noncommutative analogues of symmetric function objects, specifically of the Schur basis.  This includes the development of the quasisymmetric Schur basis, extended Schur basis, dual immaculate basis, and row-strict dual immaculate basis in $QSym$ and the dual quasisymmetric Schur basis, the row-strict dual quasisymmetric Schur basis, the shin basis, the immaculate basis, and the row-strict immaculate basis in $NSym$ \cite{Assa, berg18, Camp, haglund2011, rowstrict}.

The immaculate functions, for example, are Schur-like in that they map to the Schur functions under the forgetful map from $NSym$ to $Sym$, and they have a Jacobi-Trudi rule, a right Pieri rule, and a creation operator construction.  The dual immaculate functions, on the other hand, resemble the combinatorial definition of the Schur functions using tableaux. The primary goal of this paper is to define and study generalizations of the immaculate and dual immaculate functions in the colored algebras $QSym_A$ and $NSym_A$ introduced by Doliwa in \cite{doliwa21}.

The isomorphism between $NSym$ and a subalgebra of rooted trees led to a colored generalization of $NSym$, called $NSym_A$, that is isomorphic to a larger subalgebra of colored rooted trees studied in \cite{Foissy}.  The colored algebra $QSym_A$ is defined dually using partially commutative colored variables.  The initial goal of these generalizations was to extend the study of the relationship between symmetric functions and integrable systems to a noncommutative setting which is of growing interest in mathematical physics \cite{other_doliwa, kupershmidt2000kp}.  Additionally, the Hopf algebra of rooted trees has various applications in the field of symbolic computation \cite{Grossman_symb}.  Doliwa defines the algebraic structure of $QSym_A$ and $NSym_A$ and analogues to some classic bases. We bring the study of Schur-like bases to this space to continue developing its theory.

Studying the lift of the dual immaculate and immaculate functions to the colored quasisymmetric functions and colored noncommutative symmetric functions also allows us to obtain results on the original bases.  The introduction of colored variables reduces a significant amount of cancellation and allows for the study of various patterns in more detail.  Any results on the colored dual immaculate functions or the colored immaculate functions have immediate implications for their original counterparts. 

Section \ref{background} of this paper provides background on the symmetric functions, Hopf algebras, the quasisymmetric functions, and the noncommutative symmetric functions. We then review the immaculate and dual immaculate functions, the skew dual immaculate functions, and the immaculate poset. Section \ref{colorsection} introduces Adam Doliwa's colored generalizations of $QSym$ and $NSym$. We review their Hopf algebra structure as well as the previously defined bases.  In Section \ref{colordualsection}, we define the colored dual immaculate functions by introducing a colored generalization of immaculate tableaux.  We then give expansions of the colored dual immaculate functions into the colored monomial and colored fundamental bases using the combinatorics of colored immaculate tableaux.  Furthermore, we provide an expansion of the colored fundamental functions into the colored dual immaculate basis defined combinatorially by counting paths in a graph related to standard colored immaculate tableaux.  This result specializes to a new analogous result on the fundamental and dual immaculate bases in $QSym$. Section \ref{colorimmsection} defines the colored immaculate functions using a colored generalization of Bernstein creation operators. We prove a right Pieri rule for the colored immaculate basis and give expansions of the colored complete homogeneous and colored ribbon bases into the colored immaculate basis. Using duality, we obtain an expansion of the colored immaculate functions into the colored ribbon basis using the colored immaculate descent graph. Applying the uncoloring map yields a new combinatorial model for expanding the immaculate functions into the ribbon noncommutative symmetric functions. Additionally, applying the forgetful map yields a new expression for the decomposition of Schur functions into ribbon Schur functions.
In Section \ref{posetsection}, we introduce a partially ordered set on sentences in the style of the immaculate poset and skew colored immaculate tableaux. We use this poset to define skew colored dual immaculate functions and find results related to the structure constants of the colored immaculate basis. In Section \ref{rowstrictsection}, we first review the row-strict immaculate and row-strict dual immaculate functions, then we define the colored row-strict immaculate and colored row-strict dual immaculate functions.  These two bases are related to the immaculate and dual immaculate bases by an involution on sentences, which we use to translate our results from previous sections to the row-strict case.

\subsection{Acknowledgements} The author is grateful to Laura Colmenarejo and Sarah Mason for their support and generative feedback, and to Nantel Bergeron, Sheila Sundaram, and Aaron Lauve for helpful conversations. The author would also like to thank the referees for their helpful comments and thoughtful review.

\section{Background}\label{background} 

A \textit{partition} of  a positive integer $n$, written $\lambda \vdash n$, is a sequence of positive integers $\lambda = (\lambda_1, \ldots ,\lambda_k)$ such that $\lambda_1 \geq \ldots \geq \lambda_k$ and $\sum_i \lambda_i = n$.  The \textit{length} of a partition $\lambda = (\lambda_1,\ldots,\lambda_k)$ is the number of parts, $\ell(\lambda) = k$, and the \textit{size} of a partition is the sum of its parts, $|\lambda|=\sum_i \lambda_i$.  A \textit{composition} of a positive integer $n$, written $\alpha \vDash n$, is a sequence of positive integers $\alpha = (\alpha_1, \ldots ,\alpha_k)$ such that $\sum \alpha_i = n$. The \textit{length} of a composition $\alpha = (\alpha_1,\ldots,\alpha_k)$ is the number of parts, $\ell(\alpha) = k$, and the \textit{size} of a composition is the sum of its parts, $|\alpha|=\sum_i \alpha_i$.  A \textit{weak composition} is a composition that allows zeroes as entries. If $\beta$ is a weak composition then $\tilde{\beta}$, called the \emph{flattening} \cite{Assaf17} of $\beta$, is the composition that results from removing all 0's from $\beta$. The length of a weak composition is also its number of parts, although it is often implicitly assumed that there are infinitely many zeroes at the end of any weak composition.

\begin{ex} The partition $\lambda = (3,2,1,1)$ has size $|\lambda| = 7$ and length $\ell(\lambda) = 4$. The composition $ \alpha = (2,1,3)$ has size $|\alpha| = 6$ and length $\ell(\alpha)=3$.  The flattening of the weak composition $\beta = (0,1,1,0,2)$ is $\tilde{\beta}=(1,1,2)$.
\end{ex}

The \emph{Young diagram} of a partition $\lambda = (\lambda_1, \ldots ,\lambda_k)$ is a left-justified array of boxes such that row $i$ has $\lambda_i$ boxes. Following the English convention, the top row is considered to be row 1. The \emph{composition diagram} of a composition $\alpha = (\alpha_1, \ldots , \alpha_k)$ is a left-justified array of boxes such that row $i$ has $\alpha_i$ boxes. This only differs from a Young diagram in that the number of boxes in each row of a Young diagram must weakly decrease from top to bottom, but there is no such restriction for composition diagrams. Let $\alpha = (\alpha_1, \ldots, \alpha_k)$ and $\beta = (\beta_1, \ldots, \beta_j)$ be compositions such that $j\leq k$ and $\beta_i \leq \alpha_i$ for $1 \leq i \leq j$.  The \emph{skew shape} $\alpha/\beta$ is a composition diagram of shape $\alpha$ where the first $\beta_i$ boxes in the $i^{\text{th}}$ row are removed for $1 \leq i \leq j$. We represent this removal by shading in the removed boxes.

\begin{ex} Let $\lambda = (3,2,1,1)$, $\alpha = (2,1,3)$, and $\beta=(1,1,2)$.  Then the Young diagram of $\lambda$, the composition diagram of $\alpha$, and the skew shape $\alpha/\beta$ respectively are: \vspace{2mm}
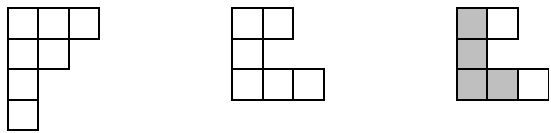
$$
\scalebox{0.75}{
\begin{ytableau}
    *(white) & *(white) & *(white)\\
    *(white) & *(white)\\
    *(white)\\
    *(white)\\
\end{ytableau}
\text{\qquad \qquad \qquad}

\begin{ytableau}
    *(white) & *(white)\\
    *(white)\\
    *(white) & *(white) & *(white)
\end{ytableau}
\text{\qquad \qquad \qquad}
\begin{ytableau}
    *(lightgray) & *(white)\\
    *(lightgray)\\
    *(lightgray) & *(lightgray) & *(white)
\end{ytableau}
}
$$\vspace{-2mm}
\end{ex}

Let $\alpha = (\alpha_1, \ldots, \alpha_k)$ and $\beta = (\beta_1, \ldots, \beta_j)$ be two compositions. Under the \textit{refinement order} $\preceq$ on compositions of size $n$, we say $\alpha \preceq \beta$ if and only if $\{\beta_1, \beta_1+\beta_2, \ldots, \beta_1 + \cdots + \beta_j = n\} \subseteq \{\alpha_1, \alpha_1 + \alpha_2, \ldots, \alpha_1 + \cdots + \alpha_k = n\}$. Under the \textit{lexicographic order} $\leq_{\ell}$ on compositions, $\alpha \leq_{\ell} \beta$ if and only if $\alpha_i < \beta_i$ where $i$ is the first positive integer such that $\alpha_i \not= \beta_i$ where $\alpha_i =0$  if $i>k$ and $\beta_i = 0$ if $i>j$. Under the \textit{reverse lexicographic order} $\leq_{r\ell}$ on compositions, $\alpha \leq_{r\ell} \beta$ if and only if $\alpha_i>\beta_i$ where $i$ is the smallest positive integer such that $\alpha_i \not= \beta_i$. Note that, in the last two orders, if such an $i$ does not exist then $\alpha = \beta$. Under the \emph{dominance order} 
$\subseteq$ on compositions, we say $\alpha \subseteq \beta$ if and only if $k \leq j$ and $\alpha_i \leq \beta_i$ for $1 \leq i \leq k$.
 
\begin{ex} We have the following chains in the corresponding orders:
\begin{enumerate}
    \item Under the refinement order, $(1,1,1,1) \preceq (1,2,1) \preceq (1,3) \preceq (4)$.
    \item  Under the lexicographic order, $(1,2,3) \leq_{\ell} (1,3,2) \leq_{\ell} (2,1,3) \leq_{\ell} (2,3,1) \leq_{\ell} (3,1,2) \leq_{\ell} (3,2,1)$.
    \item Under the reverse lexicographic order, 
$(3,2,1) \leq_{r\ell} (3,1,2) \leq_{r\ell} (2,3,1) \leq_{r\ell} (2,1,3) \leq_{r\ell} (1,3,2)$.
    \item Under the dominance order, $(1,1,1) \subseteq (2,1,1,1) \subseteq (2,3,1,2)$.
\end{enumerate}
\end{ex}

There is a natural bijection between ordered subsets of $[n-1] = \{1,2, \ldots, n-1 \}$ and compositions of $n$. For an ordered set $S = \{s_1,\ldots,s_k\} \subseteq \{[n-1]\}$, $comp(S) = (s_1, s_2-s_1, \ldots, s_k-s_{k-1}, n-s_k)$ and for a composition $\alpha = (\alpha_1,\ldots,\alpha_j)$, $set(\alpha) = \{\alpha_1, \alpha_1+\alpha_2, \ldots, \alpha_1 + \alpha_2 + \ldots + \alpha_{j-1}\}$. 

\begin{ex}
For $n=8$, let $S = \{2,3,6,7\}$ and $\alpha = (1,2,1,4).$  Then $comp(S)=(2,1,3,1,1)$ and $set(\alpha) = \{1,3,4\}$. 
\end{ex}

For a positive integer $s$ and compositions $\alpha \models n$ and $\beta \models n+s$, we write $\alpha \subset_s \beta$ if $\alpha_j \leq \beta_j$ for all $1 \leq j \leq \ell(\alpha)$, and $\ell(\beta) \leq \ell(\alpha) +1$. This notation comes from \cite{berg18} and $\subset_1$ constitutes a partial order on compositions.

\begin{ex}
The compositions $\beta$ for which $(1,2) \subset_2 \beta$ are  $$(1,2) \subset_2 (2,3) \text{ and }(1,2) \subset_2 (2,2,1) \text{ and }(1,2) \subset_2 (1,3,1).$$
\end{ex}

A \textit{permutation} $\omega$ of a set is a bijection from the set to itself. The permutation $\omega$ of $[n]$ is written in one-line notation as $\omega(1)\omega(2) \cdots \omega(n)$.

\begin{ex} The permutation $312$ maps $1 \to 3$, $2 \to 1$, and $3 \to 2$.
\end{ex}

For any set $I$, the \emph{Kroenecker delta} is the function defined for $i,j \in I$ as $$ \delta_{i,j} = \begin{cases} 1 & \text{ if } i = j,\\
0 & \text{ if } i \not= j.
\end{cases}$$

\subsection{The symmetric functions}
 Let $x = (x_1,x_2, \ldots)$ and $c_{\alpha} \in \mathbb{Q}$.  A \textit{symmetric function} $f(x)$ with rational coefficients is a formal power series $f(x) = \sum_{\alpha} c_{\alpha}x^{\alpha}$ where $\alpha$ is a weak composition of a positive integer,  $x^{\alpha} = x_1^{\alpha_1}\ldots x_k^{\alpha_k}$, and $f(x_{\omega(1)}, x_{\omega(2),\ldots}) = f(x_1,x_2,\ldots)$ for all permutations $\omega$ of $\mathbb{Z}_{>0}$.

\begin{ex} The following function $f(x)$ is a symmetric function:
$$f(x) = x_1^2x_2^3x_3 + x_1^2x_3^3x_2 + x_2^2x_1^3x_3 + x_2^2x_3^3x_1 + x_3^2x_1^3x_2+ x_3^2x_2^3x_1 + \ldots + x_4^2x_5^3x_7 + \ldots .$$
\end{ex}

The algebra of symmetric functions is denoted $Sym$, and we take $\mathbb{Q}$ as our base field unless otherwise specified. $Sym$ has many bases with various applications and combinatorial importance, but we limit ourselves to defining the Schur basis here. See \cite{EC2} for more background on symmetric functions.

For a partition $\lambda \vdash n$, a \emph{semistandard Young tableau} (SSYT) of shape $\lambda$ is a filling of the Young diagram of $\lambda$ with positive integers such that the numbers are weakly increasing from left to right in the rows and strictly increasing from top to bottom in the columns. The \emph{size} of an SSYT is its number of boxes, $n = \sum_i \lambda_i$, and its \emph{type} is a weak composition encoding the number of boxes filled with each integer. We write $type(T)=\beta = (\beta_1,\ldots,\beta_j)$ if $T$ has $\beta_i$ boxes containing an $i$ for all $i \in [j]$. Note that ``type'' is also referred to as ``content'' in the literature. A \emph{standard Young tableau} (SYT) of size $n$ is a Young tableau in which the numbers in $[n]$ each appear exactly once.  A SSYT $T$ of type $\beta = (\beta_1, \ldots, \beta_k)$ is associated with the monomial $x^T = x_1^{\beta_1} \cdots x_k^{\beta_k}$.

\begin{ex} The semistandard Young tableaux of shape $(2,2)$ with entries in $\{1,2,3\}$ and their associated monomials are:
$$\tableau{1&1\\2&2} \quad \quad
\tableau{1&1\\2&3} \quad \quad
\tableau{1&1\\3&3} \quad \quad
\tableau{1&2\\2&3} \quad \quad
\tableau{1&2\\3&3} \quad \quad
\tableau{2&2\\3&3}
$$
$$\ \ x_1^2x_2^2 \quad \quad x_1^2x_2x_3  \quad \ \ x_1^2x_3^2 \quad \quad x_1x_2^2x_3 \quad \ x_1x_2x_3^2 \quad \ \ \ x_2^2x_3^2 \quad$$
The standard Young tableaux of shape $(2,2)$, both of type $(1,1,1,1)$, are:
$$\tableau{1&2\\ 3&4} \quad \quad \tableau{1&3\\ 2&4}$$
\end{ex}

\begin{defn}
For a partition $\lambda$, the Schur symmetric function is defined as $$s_{\lambda} = \sum_T x^T,$$ where the sum runs over all semistandard Young tableaux $T$ of shape $\lambda$ with entries in $\mathbb{Z}_{> 0}$. These functions form a basis of $Sym$ \cite{EC2}.
\end{defn}

\begin{rem}
The \emph{Schur polynomials} are defined over finitely many variables $s_{\lambda}(x_1,\ldots, x_n)$ and correspond to tableaux filled only with integers in $[n]$. This paper deals only with functions in infinitely many variables, but many results restrict to polynomials.
\end{rem}

The Schur functions can be defined in numerous equivalent ways including via a Jacobi Trudi formula or with Bernstein creation operators. One of their most important properties is that they are the characters of irreducible representations of the general linear group \cite{EC2}. 

\subsection{Hopf algebras}

Hopf algebras are widespread in combinatorics and other fields with notable examples including $Sym$, $QSym$, and $NSym$. We provide a brief overview of the structures needed for our purposes. See \cite{doliwa21, grinberg} for more details.

\begin{defn} 
    For a field $\Bbbk$ of characteristic zero, a \emph{Hopf algebra} $(\mathcal{H}, \mu, \Delta, \eta, \epsilon)$ is a bialgebra, which consists of an associative algebra and a coassociative coalgebra, along with an antipode, as defined below. 
    \begin{enumerate}
        \item An associative algebra $(\mathcal{H}, \mu, \eta)$ is a $\Bbbk$-module $\mathcal{H}$ with  $\Bbbk$-linear multiplication $\mu: \mathcal{H} \otimes \mathcal{H} \rightarrow \mathcal{H}$  and a $\Bbbk$-linear unital algebra morphism $\eta: \Bbbk \rightarrow \mathcal{H}$ that together satisfy the commutative diagrams below.
          $$  \begin{CD}
            \mathcal{H}\otimes \mathcal{H}\otimes \mathcal{H} @ > \mathrm{id} \otimes \mu >>  \mathcal{H} \otimes \mathcal{H} \\
            @V \mu \otimes \mathrm{id} VV     @VV \mu V \\
            \mathcal{H}\otimes \mathcal{H} @ > \mu >>  \mathcal{H}
            \end{CD}
            \qquad \qquad
            \begin{CD}
            \mathcal{H}\otimes \Bbbk @= \mathcal{H}  @= \Bbbk \otimes \mathcal{H}\\
            @V\mathrm{id} \otimes \eta VV @V\mathrm{id} VV @VV \eta \otimes \mathrm{id} V \\
            \mathcal{H}\otimes \mathcal{H} @> \mu >>  \mathcal{H} @< \mu << \mathcal{H}\otimes \mathcal{H}
        \end{CD}$$ \vspace{1mm}
        
        \item A co-associative coalgebra $(\mathcal{H}, \Delta, \epsilon)$ is a $\Bbbk$-module $\mathcal{H}$ with $\Bbbk$-linear comultiplication $\Delta: \mathcal{H} \rightarrow \mathcal{H} \otimes \mathcal{H}$  and a $\Bbbk$-linear unital algebra morphism $\epsilon: \mathcal{H} \rightarrow \Bbbk$ called the conunit that together satisfy the commutative diagrams below.

        $$
        \begin{CD}
        \mathcal{H} @ > \Delta >> \mathcal{H}\otimes \mathcal{H} \\
        @V \Delta  VV     @VV \Delta \otimes \mathrm{id} V \\
        \mathcal{H} \otimes \mathcal{H} @ > \mathrm{id} \otimes \Delta >> \mathcal{H}\otimes \mathcal{H}\otimes \mathcal{H} 
        \end{CD} \qquad \qquad
        \begin{CD}
        \mathcal{H}\otimes \mathcal{H} @ < \Delta <<  \mathcal{H} @ > \Delta >> \mathcal{H}\otimes \mathcal{H} \\
        @V\mathrm{id} \otimes \epsilon VV @V\mathrm{id} VV @VV \epsilon \otimes \mathrm{id} V \\
        \mathcal{H}\otimes \Bbbk @= \mathcal{H}  @= \Bbbk \otimes \mathcal{H}
        \end{CD}
     $$ \vspace{1mm}

    \item The \emph{antipode} $S$ is a $\Bbbk$-linear anti-endomorphism $S$ for which the following diagram 
 below commutes.

    $$
    \label{antipode-diagram}
    \xymatrix{
    &\mathcal{H} \otimes \mathcal{H} \ar[rr]^{S \otimes id_\mathcal{H}}& &\mathcal{H} \otimes \mathcal{H} \ar[dr]^\mu& \\
    \mathcal{H} \ar[ur]^\Delta \ar[rr]^{\epsilon} \ar[dr]_\Delta& &\Bbbk \ar[rr]^{\eta} & & \mathcal{H}\\
    &\mathcal{H} \otimes \mathcal{H} \ar[rr]_{id_\mathcal{H} \otimes S}& &\mathcal{H} \otimes \mathcal{H} \ar[ur]_\mu& \\ }
    $$
\end{enumerate}
\end{defn}

The Hopf algebra $(\mathcal{H}, \mu, \Delta, \eta, \epsilon)$ is often simply denoted by $\mathcal{H}$.

\begin{defn} \label{hopf_dual_mult}
Let $(\mathcal{A},  \mu_{\mathcal{A}}, \Delta_{\mathcal{A}}, \eta_{\mathcal{A}}, \epsilon_{\mathcal{A}})$ and $(\mathcal{B},  \mu_{\mathcal{B}}, \Delta_{\mathcal{B}}, \eta_{\mathcal{B}}, \epsilon_{\mathcal{B}})$ be two Hopf Algebras with antipodes $\mathcal{S}_{\mathcal{A}}$ and $\mathcal{S}_{\mathcal{B}}$ respectively, and consider elements $a, a_1, a_2, 1_{\mathcal{A}} \in \mathcal{A}$ and $b, b_1, b_2, 1_{\mathcal{B}} \in \mathcal{B}$ where $1_{\mathcal{A}}$ and $1_{\mathcal{B}}$ are the multiplicative identity elements. $\mathcal{A}$ and $\mathcal{B}$ are \emph{dually paired} by an inner product $\langle \ ,\ \rangle: \mathcal{B} \otimes \mathcal{A} \rightarrow \mathbb{Q}$, if:
$$\langle \mu_{\mathcal{B}}(b_1,b_2), a \rangle = \langle b_1 \otimes_{\mathcal{B}} b_2, \Delta_{\mathcal{A}}(a)\rangle, \qquad \qquad \langle 1_{\mathcal{B}}, a \rangle = \epsilon_{\mathcal{A}}(a),\qquad \qquad \qquad \quad \qquad \qquad \qquad \quad $$ 
$$\langle b, \mu_{\mathcal{A}}(a_1,a_2) \rangle = \langle \Delta_{\mathcal{B}}(b), a_1 \otimes_{\mathcal{A}} a_2 \rangle, \qquad \qquad \epsilon_{\mathcal{B}}(b) = \langle b, 1_{\mathcal{A}}  \rangle,\qquad \qquad  \langle S_{\mathcal{B}}(b), a \rangle = \langle b, S_{\mathcal{A}}(a) \rangle.$$

\end{defn} 

Two bases $\{a_i\}_{i \in I}$ and $\{b_i\}_{i \in I}$ of $\mathcal{A}$ and $\mathcal{B}$ respectively are \emph{dual bases} if and only if $\langle a_i, b_j \rangle = \delta_{i,j}$. 

\begin{ex} $Sym$ is a self-dual Hopf algebra with the inner product $\langle s_{\lambda}, s_{\mu} \rangle = \delta_{\lambda, \mu}$, meaning that the Schur basis is dual to itself.
\end{ex}

The following result gives a relation for the change of bases between dual bases. 

\begin{prop}\label{hopf_dual} \cite{Hoffman} Let $\mathcal{A}$ and $\mathcal{B}$ be dually paired algebras and let $\{a_i\}_{i \in I}$ be a basis of $\mathcal{A}$. A basis $\{b_i\}_{i \in I}$ of $\mathcal{B}$ is the unique basis that is dual to $\{a_i\}_{i \in I}$ if and only if the following relationship holds for any pair of dual bases $\{c_i\}_{i \in I}$ in $\mathcal{A}$ and $\{d_i\}_{i \in I}$ in $\mathcal{B}$: $$a_i = \sum_{j \in I} k_{i,j} c_j \text{\quad and \quad} d_j = \sum_{i \in I} k_{i,j} b_i.$$
\end{prop}

Next, we have a relation for the coefficient of the product and the coproduct of dual bases.

\begin{prop}\label{hopf_dual_coproduct}\cite{grinberg}
The coproduct of the basis $\{b_i\}_{i \in I}$ in $\mathcal{B}$ is uniquely defined by the product of its dual basis $\{a_i\}_{i \in I}$ in $\mathcal{A}$ in the following way:  $$a_ja_k = \sum_{i \in I}c^{i}_{j,k}a_i \quad \Longleftrightarrow \quad \Delta(b_i) = \sum_{(j,k) \in I \times I}c^{i}_{j,k}b_j \otimes b_k.$$ Further, $\Delta: \mathcal{B} \rightarrow \mathcal{B} \otimes \mathcal{B}$ is an algebra homomorphism.
\end{prop}

\begin{ex}
For partitions $\lambda$ and $\mu$, the product of Schur functions in $Sym$ is given by $s_{\lambda}s_{\mu} = \sum_{\nu} c^{\nu}_{\lambda, \mu} s_{\nu}$ where $c^{\nu}_{\lambda, \mu}$ are the Littlewood Richardson coefficients.  For a partition $\nu$, the coproduct on Schur functions is given by $\Delta(s_{\nu}) = \sum_{\lambda, \mu} c^{\nu}_{\lambda, \mu} s_{\lambda} \otimes s_{\mu}$ by Proposition \ref{hopf_dual_coproduct}.
\end{ex}

\subsection{Quasisymmetric functions}
Let $x = (x_1, x_2, \ldots)$ and $b_{\alpha} \in \mathbb{Q}$. A \emph{quasisymmetric function} $f(x)$ with rational coefficients is a formal power series of the form 
$$f(x) = \sum_{\alpha} b_{\alpha}x^{\alpha},$$ where
\begin{enumerate}
\item $\alpha$ is a weak composition of a positive integer,
\item $x^{\alpha} = x_1^{\alpha_1}\ldots x_k^{\alpha_k}$, and
\item the coefficients of monomials $x_{i_1}^{a_1}\ldots x_{i_k}^{a_k}$ and $x_{j_1}^{a_1}\ldots x_{j_k}^{a_k}$ are equal if $i_1 < \ldots  < i_k$ and $j_1 < \ldots  < j_k$.
\end{enumerate}

We define the two most common bases of $QSym$. Given a composition $\alpha$, the \textit{monomial quasisymmetric function} $M_{\alpha}$ is defined as $$M_{\alpha} = \sum_{i_1<\ldots <i_k}x_{i_1}^{\alpha_1}\ldots x_{i_k}^{\alpha_k},$$ where the sum runs over strictly increasing sequences of $k$ positive integers $i_1, \ldots, i_k \in \mathbb{Z}_{>0}$.  
The \textit{fundamental quasisymmetric function} $F_{\alpha}$ is defined as $$F_{\alpha} = \sum_{\beta \preceq \alpha}M_{\beta}.$$  
The fundamental functions are also denoted $L_{\alpha}$ in the literature \cite{EC2}.

\begin{ex} The monomial quasisymmetric function indexed by $(2,1)$ is
$$M_{(2,1)} = \sum_{i<j}x^2_ix_j = x_1^2x_2 + x_1^2x_3 + \ldots  + x_2^2x_3 + x_2^2x_4 + \ldots  + x_3^2x_4 + x^2_3x_5 + \ldots.$$ The expansion of $F_{(3)}$ into the monomial basis is $$F_{(3)} = M_{(3)}+M_{(2,1)}+M_{(1,2)} + M_{(1,1,1)}.$$
\end{ex}

The algebra of quasisymmetric functions, denoted $QSym$, admits a Hopf algebra structure. The monomial basis inherits its product and coproduct from the quasishuffle and concatenation operations on compositions. The \emph{quasishuffle} $ \Qshuffle$ of compositions is defined as the sum of shuffles of $\alpha = (\alpha_1, \ldots , \alpha_k)$ and $\beta = (\beta_1, \ldots , \beta_l)$ where any two consecutive entries $\alpha_i$ and $\beta_j$ (in that order) may be replaced with $\alpha_i + \beta_j$.  Note that the same composition may appear multiple times in the quasishuffle. Multiplication of monomial functions is given by $$M_{\alpha}M_{\beta}= \sum_{\gamma}M_{\gamma},$$ where $\gamma$ is a summand in $\alpha \Qshuffle \beta$ with multiplicity. Comultiplication of the monomial functions is given by $$\Delta(M_{\alpha}) = \sum_{\beta \cdot \gamma = \alpha}M_{\beta} \otimes M_{\gamma},$$ where the sum runs over all compositions $\beta, \gamma$ such that $\beta \cdot \gamma = \alpha$.

\begin{ex} The following equations show the product and coproduct on monomial quasisymmetric functions expanded in terms of the monomial basis:
$$M_{(2,1)}M_{(1)} = 2M_{(2,1,1)}+M_{(1,2,1)}+M_{(2,2)}+M_{(3,1)},$$
$$\Delta(M_{(1,2,1)}) = 1 \otimes M_{(1,2,1)} + M_{(1)}\otimes M_{(2,1)} + M_{(1,2)} \otimes M_{(1)} + M_{(1,2,1)}\otimes 1.$$
\end{ex} 
For more details on quasisymmetric functions see \cite{mason}. 

\subsection{Noncommutative symmetric functions}

The algebra of \emph{noncommutative symmetric functions}, written $NSym$, is the Hopf algebra dual to $QSym$.  $NSym$ can be defined as the algebra with generators $\{H_1,H_2, \ldots \}$ and no relations, that is
$NSym = \mathbb{Q}  \left\langle H_1, H_2, \ldots \right\rangle.$ 

Given a composition $\alpha = (\alpha_1, \ldots, \alpha_k)$, we define $H_{\alpha} = H_{\alpha_1}H_{\alpha_2} \ldots H_{\alpha_k}$. Then, the set $\{H_{\alpha}\}_{\alpha}$ forms a basis of $NSym$ called the \textit{complete homogeneous basis}. $NSym$ and $QSym$ are dually paired by the inner product defined by $\langle H_{\alpha}, M_{\beta} \rangle = \delta_{\alpha, \beta}$ for all compositions $\alpha, \beta$.    

Multiplication and comultiplication in $NSym$ are defined for the complete homogeneous functions as:
$$H_{\alpha}H_{\beta}=H_{\alpha \cdot \beta } \qquad \text{and } \qquad \Delta(H_{\alpha})=\sum_{(\beta,\gamma)}H_{\beta}\otimes H_{\gamma},$$ \vspace{-4mm}\\
where $\beta$ and $\gamma$ are compositions such that $\alpha$ can be obtained from $\beta \Qshuffle \gamma$. For a composition $\alpha$, the \emph{ribbon noncommutative symmetric function} is defined as $$R_{\alpha}=\sum_{\beta \succeq \alpha}(-1)^{\ell{(a)}-\ell{(\beta)}}H_{\beta}.$$ The ribbon functions are a basis of $NSym$ dual to the fundamental basis of $QSym$, meaning $\langle R_{\alpha}, F_{\beta} \rangle = \delta_{\alpha, \beta}$. For a composition $\alpha$, the \emph{elementary noncommutative symmetric function} is defined as $$E_{\alpha} = \sum_{\beta \preceq \alpha} (-1)^{|\alpha|-\ell(\beta)}H_{\beta}.$$ For more details on the noncommutative symmetric functions see \cite{gelfand}.

\subsection{The dual immaculate quasisymmetric functions}

The dual immaculate basis of $QSym$ was introduced by Berg, Bergeron, Saliola, Serrano, and Zabrocki in \cite{berg18}.  Like the Schur functions, the dual immaculate functions are defined combinatorially as the sum of monomials associated to certain tableaux.   

\begin{defn} Let $\alpha$ and $\beta$ be a composition and weak composition respectively. An \textit{immaculate tableau} of shape $\alpha$ and type $\beta$ is a labelling of the boxes of the diagram of $\alpha$ by positive integers in such a way that:
\begin{enumerate}
    \item the number of boxes labelled by $i$ is $\beta_i$,
    \item the sequence of entries in each row, from left to right, is weakly increasing, and
    \item the sequence of entries in the first column, from top to bottom, is strictly increasing.
\end{enumerate}
\end{defn}

An immaculate tableau $T$ of type $\beta = (\beta_1,\ldots, \beta_h)$ is associated with the monomial $x^T = x_1^{\beta_1}x_2^{\beta} \cdots x_h^{\beta_h}$.

\begin{ex} The immaculate tableaux of shape $\alpha = (2,3)$ and type $\beta = (1,2,2)$ are:

$$\tableau{1 & 2\\2&3 &3}\quad \quad \tableau{1&3\\2&2&3}$$\vspace{1mm}

Both tableaux are associated with the monomial $x_1x_2^2x_3^2$.
\end{ex}

\begin{defn} 
For a composition $\alpha$, the \textit{dual immaculate function} is defined by $$\mathfrak{S}^*_{\alpha} = \sum_Tx^T,$$ where the sum runs over all immaculate tableaux $T$ of shape $\alpha$.  
\end{defn}

\begin{ex}\label{22} The dual immaculate function $\mathfrak{S}^*_{(2,2)}$ corresponds to immaculate tableaux of shape $(2,2)$:
$$\tableau{1&1\\2&2}\ \ \ \tableau{1&1\\2&3}\ \ \ \tableau{1&1\\3&3}\ \ \ \tableau{1&2\\2&2} \  \ \ \tableau{1&2\\2&3}\ \ \ \tableau{1&2\\3&3}\ \ \ \tableau{1&3\\2&2}\ \ \ \tableau{1&3\\2&3}\ \ \ \tableau{1&3\\3&3}\ \ \ \tableau{2&2\\3&3}\ \ \ \tableau{2&3\\3&3}\ \ \  \cdots$$\vspace{-1mm}
Therefore, $$\mathfrak{S}^*_{(2,2)} = x_1^2x_2^2 + x_1^2x_2x_3 + x_1^2x_3^2+ x_1x_2^3 + 2x_1x_2^2x_3 +2x_1x_2x_3^2  + x_1x_3^3+ x_2^2x_3^2 + x_2x_3^3 + \ldots .$$
\end{ex} 

The dual immaculate functions have positive expansions into the monomial and fundamental bases.

\begin{prop}\label{uncolor_mon}\cite{berg18}
The dual immaculate functions are monomial positive.  Moreover, they expand as $$\mathfrak{S}^*_{\alpha} = \sum_{\beta \leq_{\ell} \alpha} K_{\alpha,\beta}M_{\beta},$$ where $K_{\alpha,\beta}$ is the number of immaculate tableaux of shape $\alpha$ with type $\beta$.
\end{prop}

\begin{ex}
Let $\alpha = (2,2)$ as before.  The set of compositions $\beta$ such that $|\alpha|=|\beta|$ and $\beta \leq_{\ell} \alpha$ is $\{(2,2), (2,1,1), (1,3),(1,2,1),(1,1,2), (1,1,1,1)\}$. The two immaculate tableaux of shape $(2,2)$ and type $(1,1,2)$ are $$\tableau{1&2\\3&3}\ \quad \quad \tableau{1&3\\2&3} $$ \vspace{1mm}
Thus, $K_{(2,2),(1,1,2)}=2$, and repeating that calculation for each $\beta$ in the set above yields
$$\mathfrak{S}^*_{(2,2)} = M_{(2,2)} + M_{(2,1,1)} + M_{(1,3)} + 2M_{(1,2,1)}+2M_{(1,1,2)} + 3M_{(1,1,1,1)}.$$
\end{ex}

The expansion of the dual immaculate functions into the fundamental basis relies on the following subset of immaculate tableaux.

\begin{defn}
A \textit{standard immaculate tableau} of shape $\alpha \models n$ is an immaculate tableau in which each integer $1$ through $n$ appears exactly once.
\end{defn}

\begin{ex}
The standard immaculate tableaux of shape $\alpha = (2,3)$ are:\vspace{1mm} 
$$\tableau{1&2\\3&4&5} \quad \quad \tableau{1&3\\2&4&5} \quad \quad \tableau{1&4\\2&3&5} \quad \quad \tableau{1&5\\2&3&4}$$ \vspace{-2mm}
\end{ex}

Every immaculate tableau can be associated with a standard immaculate tableau by standardization.

\begin{defn}\label{standardization} Given an immaculate tableau $T$ of shape $\alpha$, form a standard immaculate tableau $std(T) = U$ of shape $\alpha$ by relabeling the boxes of $T$ with the integers $1$ through $n$ in the following way.  Begin with all boxes filled with 1's then continue on to the boxes filled with 2's, then 3's, and so on, ignoring boxes we have already relabelled. Starting from the lowest row containing each value, move through boxes filled with the same value from left to right and bottom to top, relabelling each with the next integer from $[n]$, starting the very first box with $1$. The resultant tableau $U$ is a standard immaculate tableau which we call the \emph{standardization} of $T$. 
\end{defn}

\begin{ex} The two immaculate tableaux below both have shape $(2,3)$ and type $(1,2,2)$ but different standardizations: $$T_1 = \tableau{1 & 2\\2&3 &3}\quad \quad \quad \ T_2 = \tableau{1&3\\2&2&3}$$\\
$$std(T_1) = \tableau{1 & 3\\2&4 &5}\quad \quad \quad \ std(T_2) = \tableau{1&5\\2&3&4}$$
\end{ex}

\begin{defn}\cite{berg18, EC2} \label{desc}
A standard immaculate tableau $U$ has a \textit{descent} in position $i$ if $(i+1)$ is in a row strictly lower than $i$ in $U$.  We denote the set of all descents in $U$ as $Des(U)$, called the \textit{descent set} of $U$.  If $Des(U)=\{d_1,\ldots ,d_{k-1}\}$ then the \textit{descent composition} of $U$ is defined as $co(U)=comp(Des(U))=(d_1,d_2-d_1,d_3-d_2,\ldots ,n-d_{k-1})$.
\end{defn} 

\begin{prop}\label{uncolor_fund}\cite{berg18}
The dual immaculate functions $\mathfrak{S}^*_{\alpha}$ are fundamental positive.  They expand as $$\mathfrak{S}^*_{\alpha} = \sum_{\beta \leq_{\ell} \alpha}L_{\alpha,\beta} F_{\beta},$$ where $L_{\alpha,\beta}$ is the number of standard immaculate tableaux with shape $\alpha$ and descent composition $\beta$.
\end{prop}

\begin{ex}
Let $\alpha = (2,2)$.  The standard immaculate tableaux of shape $(2,2)$, listed with their descent sets and descent compositions, are 

$$
S_1 = \tableau{1&2\\3&4}\ \qquad \ \ S_2 = \tableau{1&3\\2&4}\ \qquad \ \ S_3 = \tableau{1&4\\2&3}$$

$$ \ \ \ \substack{Des(S_1)=\{2\},\\ co(S_1)=(2,2)} \qquad \ \substack{Des(S_2)=\{1,3\},\\ co(S_2)=(1,2,1)} \qquad \ \
\substack{Des(S_3)=\{1\},\\ co(S_3)=(1,3)}
\quad $$\vspace{0mm}

Therefore, $L_{(2,2)(2,2)}=1$, $L_{(2,2),(1,2,1)} = 1$, and $L_{(2,2),(1,3)}=1$, meaning $\mathfrak{S}^*_{(2,2)} = F_{(1,2,1)} + F_{(1,3)}+ F_{(2,2)}.$
\end{ex}

\subsection{The immaculate noncommutative symmetric functions}

The dual immaculate functions were originally developed as the duals to the immaculate functions in $NSym$ \cite{berg18}.   The immaculate functions are defined constructively by creation operators that generalize the Bernstein operators used to define the Schur functions. 

For $F \in QSym$, the operator $F^{\perp}$ acts on elements $H \in NSym$ based on the relation $\langle H,FG \rangle = \langle F^{\perp}H,G\rangle$.  This expands as $F^{\perp}(H) = \sum_{\alpha}\langle H,FA_{\alpha}\rangle B_{\alpha}$ for dual bases $\{A_{\alpha}\}_{\alpha}$ of $QSym$ and $\{B_{\alpha}\}_{\alpha}$ of $NSym$. Most important for our purposes is the specialization of this operator to the fundamental basis \cite{berg18}, where the expansion of $F^{\perp}_{1^i}$ acting on $H_{\alpha}$ is
$$F^{\perp}_{1^i}(H_{\alpha}) = \sum_{\substack{\beta \in \mathbb{N}^m \\ |\beta|=|\alpha|-i \\ \alpha_j -1 \leq \beta_j \leq \alpha_j}} H_{\tilde{\beta}}.$$ 
We interpret the action of this operator on the indices of $H_{\alpha}$. This operator acts on the composition $\alpha$ by taking a diagram of shape $\alpha$ and returning the sum of all diagrams (as indices of the $H$'s) whose shape is obtained by removing $i$ boxes from the right-hand side with no more than 1 box being removed from each row.

\begin{ex}
For instance, $F^{\perp}_{(1,1)}H_{(2,1,1,2)} = H_{(2,2)} + 2H_{(2,1,1)} + 2H_{(1,1,2)} + H_{(1,1,1,1)}$ can be visualized with the following tableaux, where the gray blocks are removed and all rows are moved up to fill any entirely empty rows.
$$
\scalebox{0.75}{
\begin{ytableau}
\ & \   \\
*(lightgray)    \\
*(lightgray) \\
\ & \ \\
\end{ytableau}
} \quad \quad
\scalebox{0.75}{
\begin{ytableau}
*(white) & *(white)   \\
*(lightgray)    \\
*(white) \\
*(white) & *(lightgray) \\
\end{ytableau}
} \quad \quad
\scalebox{0.75}{
\begin{ytableau}
*(white) & *(white)   \\
*(white)    \\
*(lightgray) \\
*(white) & *(lightgray) \\
\end{ytableau}
} \quad \quad
\scalebox{0.75}{
\begin{ytableau}
*(white) & *(lightgray)   \\
*(lightgray)    \\
*(white) \\
*(white) & *(white) \\
\end{ytableau}
} \quad \quad 
\scalebox{0.75}{
\begin{ytableau}
*(white) & *(lightgray)   \\
*(white)    \\
*(lightgray) \\
*(white) & *(white) \\
\end{ytableau}
} \quad \quad
\scalebox{0.75}{
\begin{ytableau}
*(white) & *(lightgray)   \\
*(white)    \\
*(white) \\
*(white) & *(lightgray) \\
\end{ytableau}
}
$$\vspace{-1mm}
$$\ \ \  H_{2,2} \ \ \quad \quad \  H_{2,1,1}  \quad \quad \  H_{2,1,1} \quad \quad \  H_{1,1,2} \quad \ \ \ \  H_{1,1,2} \quad \ \ \  H_{1,1,1,1}$$
\end{ex}

\begin{defn} For $m \in \mathbb{Z},$ the \textit{noncommutative Bernstein operator} $\mathbb{B}_m$ is defined as $$\mathbb{B}_m = \sum_{i \geq 0}(-1)^iH_{m+i}F_{1^i}^{\perp}.$$ 
\end{defn}

These operators generalize the Bernstein operators used to construct the Schur functions \cite{EC2} and thus allow for the construction of a noncommutative generalization of the Schur functions.

\begin{defn} For $\alpha = [\alpha_1, \ldots , \alpha_m] \in \mathbb{Z}^m,$ the \textit{immaculate noncommutative symmetric function} $\mathfrak{S}_{\alpha}$ is defined as $$\mathfrak{S}_{\alpha} = \mathbb{B}_{\alpha_1}\cdots\mathbb{B}_{\alpha_m}(1).$$ The \emph{immaculate basis} of $QSym$ is the set of immaculate functions $\{\mathfrak{S}_{\alpha}\}_{\alpha}$ where $\alpha \models n$ for $n \in \mathbb{Z}_{>0}$.
\end{defn}

\begin{ex}
If $\alpha = (\alpha_1, \alpha_2)$, then $\mathfrak{S}_{(\alpha_1, \alpha_2)} = \mathbb{B}_{\alpha_1}(H_{\alpha_2}) = H_{\alpha_1}H_{\alpha_2} - H_{\alpha_1+1}H_{\alpha_2-1}$.
\end{ex}

Properties of these Bernstein operators lead to a right Pieri rule for the immaculate functions.
\begin{thm} \cite{berg18} For a composition $\alpha$ and an integer $s$,
 $$\mathfrak{S}_{\alpha}H_s = \sum_{\alpha \subset_s \beta}\mathfrak{S}_{\beta},$$ where the sum runs over all compositions $\beta$ such that $\alpha \subset_s \beta$. 
\end{thm}

\begin{ex} Applying the Pieri rule for $\alpha = (2,1)$ and $s = 2$ yields
$$\mathfrak{S}_{(2,1)}H_{(2)} = \mathfrak{S}_{(2,1,2)} + \mathfrak{S}_{(2,2,1)} + \mathfrak{S}_{(3,1,1)} + \mathfrak{S}_{(2,3)} + \mathfrak{S}_{(3,2)} + \mathfrak{S}_{(4,1)} .$$
\end{ex}

Iteration of this Pieri rule yields the following positive expansions of the complete homogeneous and ribbon bases in terms of the immaculate basis:  
$$H_{\beta} = \sum_{\alpha \geq_{\ell} \beta} K_{\alpha,\beta}\mathfrak{S}_{\alpha} \qquad\text{ and } \qquad R_{\beta} = \sum_{\alpha \geq_{\ell} \beta}L_{\alpha,\beta}\mathfrak{S}_{\alpha}.$$

Notice that these expansions relate to those in Propositions \ref{uncolor_mon} and \ref{uncolor_fund} via Proposition \ref{hopf_dual}.  The expansion of the immaculate functions into the complete homogeneous basis follows a \emph{Jacobi-Trudi rule}.

\begin{thm}\cite{berg18}
For $\alpha = [\alpha_1, \ldots, \alpha_m]\in \mathbb{Z}^m$, $$\mathfrak{S}_{\alpha} = \sum_{\sigma \in S_m} (-1)^{\sigma} H_{\alpha_1 + \sigma_1 - 1, \alpha_2+ \sigma_2-2,\ldots, \alpha_m+\sigma_m-m},$$ with $H_0 = 1$ and $H_{-m} = 0$ for $m>0$. This is equivalent to taking the noncommutative analogue of the determinant of the matrix below obtained by expanding the determinant of the matrix along the first row and multiplying those elements on the left:\\
\[
\left[{\begin{array}{cccc}
H_{\alpha_1} & H_{\alpha_1+1} & \cdots & H_{\alpha_1+ \ell - 1}\\
H_{\alpha_2-1} & H_{\alpha_2} & \cdots & H_{\alpha_2+ \ell - 2}\\
\vdots & \vdots & \ddots & \vdots\\
H_{\alpha_{\ell}- \ell + 1} & H_{\alpha_{\ell}- \ell + 2}& \cdots & H_{\alpha_{\ell}}

\end{array}}\right].
\]
\end{thm}\vspace{5pt}

Certain classes of immaculate functions also have simpler expansions in terms of the complete homogeneous basis \cite{berg18}. For instance, for a positive integer $n$,
$$\mathfrak{S}_{1^n} = \sum_{\alpha \models n}(-1)^{n-\ell(\alpha)}H_{\alpha}.$$

There is another right Pieri rule for multiplication by these immaculate functions. For a composition $\alpha$ and a positive integer $s$,
$$\mathfrak{S}_{\alpha}\mathfrak{S}_{1^s} = \sum_{\substack{\beta \models |\alpha|+s\\ \alpha_i \leq \beta_i \leq \alpha_i+1}} \mathfrak{S}_{\beta}.$$

\subsection{Skew dual immaculate functions and the immaculate poset}\label
{skewposetbackground}The \emph{immaculate poset} $\mathfrak{P}$, also defined in \cite{berg18}, is a labelled poset on compositions where $\alpha$ covers $\beta$ if $\beta \subset_1 \alpha$.  In other words, $\alpha$ covers $\beta$ if $\alpha$ can be obtained by adding $1$ to any part of $\beta$ or to the end of $\beta$ as a new part. In terms of diagrams, this is equivalent to adding a box to the right of any row or adding a box at the bottom of the tableau. 

In the Hasse diagram of $\mathfrak{P}$, label the arrow from $\beta$ to $\alpha$ with $m$, where $m$ is the number of the row where the new box is added. Maximal chains from $\emptyset$ to $\alpha$ are equivalent to standard immaculate tableaux of shape $\alpha$, and maximal chains from $\beta$ to $\alpha$ define skew standard immaculate tableaux of shape $\alpha/\beta$.  A path $\{\beta = \beta^{(0)} \rightarrow^{m_1} \beta^{(1)} \rightarrow^{m_2} \ldots  \rightarrow^{m_k} \beta^{(k)}=\alpha \}$ corresponds to the skew standard immaculate tableaux of shape $\alpha/\beta$ where the boxes are filled with positive integers in the order they were added following the path.

\begin{ex}
Consider two paths $\mathcal{P}_1 = \{\emptyset \xrightarrow{1} (1) \xrightarrow{2} (1,1) \xrightarrow{2} (1,2) \xrightarrow{1} (2,2)\}$ and $\mathcal{P}_2 = \{ \emptyset \xrightarrow{1} (1) \xrightarrow{1} (2) \xrightarrow{2} (2,1) \xrightarrow{2} (2,2)\}$.  These paths correspond to the standard immaculate tableaux $T_1$ and $T_2$ below, respectively.  The path $\mathcal{P}_3 = \{ (1) \xrightarrow{2} (1,1) \xrightarrow{1} (2,1) \xrightarrow{2} (2,2)\}$ corresponds to the skew standard immaculate tableau $T_3$.
$$
\scalebox{0.85}{
$T_1 =$ \begin{ytableau} 1& 4\\2&3 \end{ytableau}\ \quad \ \ $T_2 =$ \begin{ytableau} 1&2\\3&4 \end{ytableau}\ \quad \ \ $T_3$ = \begin{ytableau} *(lightgray) & 2 \\ 1&3 \end{ytableau}
}$$
\end{ex}

\begin{defn} \cite{rowstrict}
Let $\alpha$ and $\beta$ be compositions where $\beta \subseteq \alpha$. A \emph{skew immaculate tableau} of shape $\alpha/\beta$ is a skew shape $\alpha/\beta$ filled with positive integers such that the entries in the first column of $\alpha$ are strictly increasing from top to bottom and the entries in rows are weakly increasing from left to right. We say $T$ is a \textit{skew standard immaculate tableau} if it contains the entries $1, \ldots, |\alpha|-|\beta|$ with each appearing exactly once.
\end{defn}

\begin{defn} \cite{berg18} Given compositions $\alpha, \beta$ with $\beta \subseteq \alpha$, the \emph{skew dual immaculate function} is defined as $$\mathfrak{S}^*_{\alpha/\beta} = \sum_{\gamma} \langle \mathfrak{S}_{\beta}H_{\gamma}, \mathfrak{S}^*_{\alpha} \rangle M_{\gamma},$$ where the sum runs over all $\gamma \models |\alpha|-|\beta|$. 
\end{defn}

The coefficient $\langle \mathfrak{S}_{\beta}H_{\gamma}, \mathfrak{S}^*_{\alpha}\rangle$ is exactly equal to the number of skew standard immaculate tableaux of shape $\alpha/\beta$ with type $\gamma$ \cite{rowstrict}. Thus, the skew dual immaculate functions can also be defined by a sum over skew immaculate tableaux.

\begin{thm}\cite{rowstrict} Let $\alpha$ and $\beta$ be compositions with $\beta \subseteq \alpha$. Then $$\mathfrak{S}^*_{\alpha/\beta} = \sum_T x^T,$$ where the sum runs over all skew immaculate tableaux of shape $\alpha/\beta$.
\end{thm}

Expansions of the skew dual immaculate functions into the fundamental and dual immaculate bases yield coefficients with connections to the multiplicative structure of the immaculate functions.

\begin{prop} \label{uncolor_skew_coeff_bases} \cite{berg18} Given compositions $\alpha$ and $\beta$ with $\beta \subseteq \alpha$,
$$\mathfrak{S}^*_{\alpha/\beta} = \sum_{\gamma} \langle \mathfrak{S}_{\beta}R_{\gamma}, \mathfrak{S}^*_{\alpha} \rangle F_{\gamma} = \sum_{\gamma} \langle \mathfrak{S}_{\beta}\mathfrak{S}_{\gamma}, \mathfrak{S}^*_{\alpha} \rangle \mathfrak{S}^*_{\gamma},$$ where the sums run over all $\gamma \models |\alpha|-|\beta|$. Moreover, the coefficients $c^{\alpha}_{\beta,\gamma} = \langle \mathfrak{S}_{\beta}\mathfrak{S}_{\gamma}, \mathfrak{S}^*_{\alpha} \rangle$ are the immaculate structure constants that appear in the expansion $$\mathfrak{S}_{\beta}\mathfrak{S}_{\gamma}=\sum_{\alpha}c^{\alpha}_{\beta,\gamma}\mathfrak{S}_{\alpha}.$$
\end{prop}

Additionally, the comultiplication of the dual immaculate functions can be described using skew compositions.

\begin{defn} \cite{rowstrict} Given $\alpha \models n$, the comultiplication on $\mathfrak{S}_{\alpha}^*$ is defined as $$\Delta (\mathfrak{S}^*_{\alpha}) = \sum_{\beta} \mathfrak{S}^*_{\beta} \otimes \mathfrak{S}^*_{\alpha/\beta},$$ where the sum runs over all compositions $\beta$ such that $\beta \subseteq \alpha$.
\end{defn}

The multiplication and antipode of the dual immaculate functions do not yet have combinatorial definitions in general. For more on the immaculate and dual immaculate functions see \cite{AllMas, berg_structures, bergerondualpieri, campbell_antipode, campbell2, grinberg_antipode, li, loehr}.

\section{Doliwa's colored $QSym_A$ and $NSym_A$} \label{colorsection}
The algebra of noncommutative symmetric functions, and dually the algebra of quasisymmetric functions, have natural generalizations isomorphic to algebras of sentences.  In \cite{doliwa21}, Doliwa introduces these generalizations which are built using partially commutative colored variables. 

Let $A = \{a_1, a_2, \ldots  , a_m\}$ be an alphabet of letters, which we call \textit{colors}.  \textit{Words} over $A$ are finite sequences of colors written without separating commas. Finite sequences of non-empty words are called \textit{sentences}. The empty word and the empty sentence are both denoted by $\emptyset$. A \emph{weak sentence} may include empty words. The \textit{size} of a word $w$, denoted $|w|$, is the total number of colors it contains. Note that when we refer to ``the number of colors'', we are counting repeated colors unless we say ``the number of unique colors''. The \textit{size} of a sentence $I=(w_1,w_2, \ldots , w_k)$, denoted $|I|$, is also the number of colors it contains.  The \textit{length} of a sentence $I$, denoted $\ell(I)$, is the number of words it contains. The \emph{concatenation} of two words $w = a_1\cdots a_k$ and $v=b_1\cdots b_j$ is $w \cdot v = a_1 \cdots a_kb_1 \cdots b_j$, sometimes just denoted $wv$.  The word obtained by concatenating every word in a sentence $I$ is called the \textit{maximal word} of $I$, denoted $w(I)=w_1w_2\ldots w_k$. For our purposes, we also define the \textit{word lengths} of $I$ as $w \ell(I)=(|w_1|, \ldots , |w_k|)$, which gives the underlying composition of the sentence.

\begin{ex}
Let $a,b,c \in A$ and let $w_1 = ac$, $w_2 = b$, and $w_3 = cab$ be words.  Consider the sentence $I = (w_1,w_2,w_3) = (ac,b,cab)$.  Then, $|w_1|=2$, $|w_2|=1$, $|w_3|=3$, and $|I|=6$. The length of $I$ is $\ell(I)=3$ and the word lengths of $I$ is $w \ell(I)=(2,1,3)$.  The maximal word of $I$ is $w(I)=acbcab$.  
\end{ex}

 A sentence $I$ is a refinement of a sentence $J$, written $ I \preceq J$, if $J$ can be obtained by concatenating some adjacent words of $I$. In other words, $I \preceq J$ if $w(I)=w(J)$ and $w \ell(I) \preceq w \ell(J)$. In this case, $I$ is called a \emph{refinement} of $J$ and $J$ a \emph{coarsening} of $I$. The \emph{Möbius function} on the poset of sentences ordered by refinement is given by \begin{equation}
      \mu(J,I) = (-1)^{\ell(J)-\ell(I)} \text{\quad for \quad} J \preceq I.
 \end{equation}
 Given a total order $\leq$ on $A$, define the following \emph{lexicographic order} $\preceq_{\ell}$ on words. For words $w=a_1 \ldots a_k$ and $v=b_1 \ldots b_j$, we say $w \leq_{\ell} v$ if $a_i < b_i$ for the first positive integer $i$ such that $a_i \not= b_i$. Note that if no such $i$ exists then $w = v$. 

\begin{ex} Let $A = \{a < b < c\}$ and $I = (abc)$.  The refinements of $I$ are $(abc)$, $(a,bc)$, $(ab,c)$, and $(a,b,c)$. Under lexicographic order, $abc \preceq_{\ell} acb \preceq_{\ell} bac \preceq_{\ell} bca \preceq_{\ell} cab \preceq_{\ell} cba$.
\end{ex} 

The \emph{concatenation} of two sentences $I = (w_1,\ldots ,w_k)$ and $J=(v_1,\ldots ,v_h)$ is $I \cdot J = (w_1,\ldots ,w_k,v_1,\ldots ,v_h)$.  Their \emph{near-concatenation} is $I \odot J = (w_1,\ldots ,w_kv_1,\ldots ,v_h)$ where the words $w_k$ and $v_1$ are concatenated into a single word. Given $I = (w_1,\ldots ,w_k)$ where $a_{i}$ is the $i^{\text{th}}$ entry in $I$ and $a_{i+1}$ is the $(i+1)^{\text{th}}$ entry in $I$, we say that $I$ \textit{splits} after the $i^{\text{th}}$ entry if $a_i \in w_j$ and $a_{i+1} \in w_{j+1}$ for $j \in [k]$.

\begin{ex} Let $I = (a,bc)$ and $J = (ca,b)$.  Then, $I \cdot J = (a,bc,ca,b)$ and $I \odot J = (a,bcca,b)$.  The sentence $(a,bcca,b)$ splits after the $1^{\text{st}}$ and $5^{\text{th}}$ entries.
\end{ex}

Given $I = (w_1,\ldots ,w_k)$, the \emph{reversal} of $I$ is $I^r = (w_k, w_{k-1},\ldots ,w_1)$. The \emph{complement} of $I$, denoted $I^c$, is the unique sentence such that $w(I)=w(I^c)$ and $I^c$ splits exactly where $I$ does not. Both maps are involutions on sentences.

\begin{ex}
Let $I = (abc,de)$.  Then $I^r = (de,abc)$ and $I^c = (a,b,cd,e).$
\end{ex}

The \emph{flattening} of a weak sentence $I$, denoted $\tilde{I}$, is the sentence obtained by removing all empty words from $I$. Further, for a weak sentence $J = (v_1,\ldots ,v_k)$ and a sentence $I = (w_1,\ldots ,w_k)$, we say that $J$ is  \emph{right-contained} in $I$, denoted $J \subseteq_R I$, if there exists a weak sentence $I/_R J = (u_1,\ldots ,u_k)$ such that $w_i = u_iv_i$ for every $i \in [k]$. We say that $J$ is \emph{left-contained} in $I$, denoted $J \subseteq_L I$, if there exists a weak sentence $ I/_L J=(q_1, \ldots , q_k)$ such that $w_i = v_iq_i$ for every $i \in [k]$. Note that right-containment is denoted $I/J$ in \cite{doliwa21} but here that notation is used exclusively to denote skew shapes.

\begin{ex}
Let $I = (ab,cdef)$, $J = ( b, ef)$, and $K = (a,cde)$. Then $J \subseteq_R I$ and $I/_R J = (a,cd)$, while $K \subseteq_L I$ and $I/_L K = (b,f)$. Given the weak sentence $I = (\emptyset, a, \emptyset, bc)$, the flattening of $I$ is $\tilde{I}=(a,bc)$.
\end{ex}

\subsection{The Hopf algebra of sentences and colored noncommutative symmetric functions}

The algebra of sentences (colored compositions) is a Hopf algebra with the multiplication being the concatenation of sentences, the comultiplication given by $$\Delta(I)= \sum_{J \subseteq_R I} \widetilde{I/_R J} \otimes \tilde{J},$$ the natural unity map, the counit 
\begin{equation*}
\epsilon(I) = 
\begin{cases}
    1, & \text{if } I = \emptyset,\\
    0, & \text{otherwise,} 
\end{cases}
\end{equation*}
and the antipode
$$S(I) = \sum_{J \preceq I^r} (-1)^{\ell(J)}J.$$

The algebra of sentences taken over an alphabet with only one letter is isomorphic to $NSym$.  Thus, the algebra of sentences taken over any alphabet $A$ is a natural extension of $NSym$ called \emph{the algebra of colored noncommutative symmetric functions}, denoted $NSym_A$. The linear basis of sentences $I$ is the complete homogeneous basis of $NSym_A$, denoted $\{H_I\}_I$. 

$NSym_A$ can also be defined as the algebra freely generated over noncommuting elements $H_w$ for any word in $A$. The Hopf algebra operations extend to $\{H_I\}_I$ as follows:
$$H_I \cdot H_J = H_{I \cdot J}, \quad \quad \quad \Delta(H_I) = \sum_{J \subseteq_R I}H_{\widetilde{I/_R J}} \otimes H_{\tilde{J}}, \quad \quad \quad S(H_I) = \sum_{J \preceq I^r}(-1)^{\ell(J)}H_J.$$ 
The reversal and complement operations extend as $H_I^r = H_{I^r}$ and $H^c_I=H_{I^c}$.

\begin{defn} The \emph{uncoloring} map $\upsilon : NSym_A \rightarrow NSym$ is defined $\upsilon (H_{I}) = H_{w\ell(I)}$ and extended linearly. If the alphabet $A$ only contains one color, then $\upsilon $ is an isomorphism.
\end{defn}

We say that two bases $\{B_I\}_I$ and $\{C_{\alpha}\}_{\alpha}$ in $NSym_A$ and $NSym$ respectively are \emph{analogous} if $\upsilon (B_I) = C_{w\ell(I)}$ for all sentences $I$ when $A$ is an alphabet of one color. For instance, the colored complete homogeneous basis of $NSym_A$ is analogous to the complete homogeneous basis of $NSym$. $NSym_A$ also contains analogues of the elementary and ribbon bases of $NSym$. For a sentence $I$, the \emph{colored elementary function} is defined by $$E_I = \sum_{J \preceq I}(-1)^{|I|-\ell(J)}H_J,$$ and the \emph{colored ribbon function} is defined by \begin{equation}\label{R_H_equation}
   R_I = \sum_{J \succeq I}(-1)^{\ell(J)-\ell(I)}H_J \text{\qquad and so \qquad} H_I = \sum_{J \succeq I} R_J. 
\end{equation}

Note that we use $\upsilon $ to denote the uncoloring maps on both $QSym_A$ and $NSym_A$, and often refer to these together as if they are one map.   

%%%%%%%%%%%%%%%%%%%%%%%%%%%%%%%%%%%%%%%%%%%%%%%%%%%%%%%%%%%%%%%%%%%%%%%%%%%%%%%%%%%
\subsection{The colored quasisymmetric functions and colored duality} \label{QSymA}

The colored quasisymmetric functions, which constitute the algebra dual to $NSym_A$, are constructed using partially commutative colored variables. For a color $a \in A$, define the set of infinite colored variables $x_a = \{x_{a,1}, x_{a,2}, \ldots \}$ and let $x_A = \cup_{a \in A} x_a$. These variables are assumed to be partially commutative in the sense that variables only commute if the second indices are different.  That is, for $a,b \in A$,
$$x_{a,i}x_{b,j}=x_{b,j}x_{a,i} \text{ for } i \not= j \qquad \text{and}\qquad x_{a,i}x_{b,i} \not= x_{b,i}x_{a,i} \text{ if } a \not= b.$$  As a result, every monomial in variables $x_{a,i}$ can be uniquely re-ordered so that the sequence of the second indices of the variables is weakly increasing, at which point any first indices sharing the same color can be combined into a single word. Every monomial has a sentence $(w_1, \ldots, w_m)$ defined by its re-ordered, combined form $x_{w_1,j_1}\cdots x_{w_m, j_m}$ where $j_1 < \ldots < j_m$. Similar notions of coloring with different assumptions of partial commutativity can be found in \cite{other_color, poirier}.

\begin{ex}
The monomial $x_{a,2}x_{b,3}x_{b,1}x_{c,2}$ can be reordered as $ x_{b,1}x_{a,2}x_{c,2}x_{b,3}$ and combined as $x_{b,1}x_{ac,2}x_{b,3}$.  Then, the sentence of this monomial is $(b, ac, b)$. 
\end{ex}

$QSym_A$ is a subset of $\mathbb{Q}[x_A]$ defined as the set of formal power series such that the coefficients of the monomials indexed by the same sentence are equal.

\begin{ex} The following function $f(x_A)$ is in $QSym_A$:
$$ f(x_A) = 3x_{a,1}x_{bc,2} + 3x_{a,1}x_{bc,3} + \ldots + 3x_{a,2}x_{bc,3} + 3x_{a,2}x_{bc,4} + \ldots .$$
\end{ex}

Bases in $QSym$ extend naturally to bases in $QSym_A$.  For a sentence $I = (w_1, w_2, \ldots , w_m)$, the \emph{colored monomial quasisymmetric function} $M_I$ is defined as 
$$M_I = \sum_{1 \leq j_1<j_2<\ldots <j_m}x_{w_1,j_1}x_{w_2,j_2}\ldots x_{w_m, j_m},$$ where the sums runs over strictly increasing sequences of $m$ positive integers $j_1, \ldots, j_m \in \mathbb{Z}_{>0}$.

\begin{ex} The colored monomial quasisymmetric function for the sentence $(a,bc)$ is
 $$M_{(a,bc)} = x_{a,1}x_{bc,2} + x_{a,1}x_{bc,3} + \ldots  + x_{a,2}x_{bc,3} + x_{a,2}x_{bc,4} + \ldots  + x_{a,3}x_{bc,4} + \ldots.$$ 
\end{ex}

\begin{prop} \cite{doliwa21}
The subspace $QSym_A$ of $\mathbb{Q}[x_A]$  spanned by $\{M_I\}_I$ is a subalgebra isomorphic to the graded algebra dual of $NSym_A$ such that $M_I$ is mapped to the dual of $H_I$. That is, $QSym_A$ and $NSym_A$ are Hopf algebras dually paired by the inner product $\langle H_I, M_J \rangle = \delta_{I,J}$.
\end{prop}

$QSym_A$ and $NSym_A$ inherit the product and coproduct from the Hopf algebra of sentences. The quasishuffle $I \Qshuffle J$ is defined as the sum of all shuffles of sentences $I$ and $J$ and shuffles of sentences $I$ and $J$ with any number of pairs $w_iv_j$ of consecutive words $w_i \in I$ and $v_j \in J$  concatenated.

\begin{ex} The usual shuffle operation on $(ab,c)$ and $(d,e)$ is
$$(ab,c) \shuffle (d,e) = (ab,c,d,e) + (ab,d,c,e) + (d,ab,c,e) + (ab,d,e,c) + (d,ab,e,c) + (d,e,ab,c).$$
The quasishuffle of $(ab,c)$ and $(d,e)$ is
\begin{multline*}
    (ab,c) \Qshuffle (d,e) = (ab,c,d,e) + (ab,cd,e)+(ab,d,c,e)+ (abd,c,e)+ (ab,d,ce)+ (abd,ce) +(d,ab,c,e) + \\+ (d,ab,ce) + (ab,d,e,c) + (abd,e,c) + (d,ab,e,c) + (d,abe,c) + (d,e,ab,c).
\end{multline*}
\end{ex} 

\noindent  Multiplication in $QSym_A$ is dual to the coproduct $\Delta$ in $NSym_A$, and given by $$M_IM_J = \sum_K M_K,$$ where  the sum runs over all summands $K$ in $I \Qshuffle J$. Similarly, comultiplication is dual to the concatenation product in $NSym_A$ using the deconcatenation product 
\begin{equation}\label{comult_m}
    \Delta(M_I) = \sum_{I=J\cdot K}M_J \otimes M_K,
\end{equation}
where the sum runs over all sentences $J$ and $K$ such that $I = J \cdot K$. Finally, the antipode $S^*$ in $QSym_A$ is given by $$S^*(M_I) = (-1)^{\ell(I)}\sum_{J^r \succeq I}M_J.$$ 

\begin{defn} The \emph{uncoloring} map $\upsilon : QSym_A \rightarrow QSym$ is defined by $\upsilon (x_{w_1,1}\cdots x_{w_k,k}) = x_1^{|w_1|}\cdots x_k^{|w_k|}$ and extends linearly. If the alphabet $A$ contains only one color, $\upsilon $ is an isomorphism.
\end{defn}

We say two bases $\{B_I\}_I$ and $\{C_{\alpha}\}_{\alpha}$ of $QSym_A$ and $QSym$ are analogous if $\upsilon (B_I)=C_{w\ell(I)}$ for all sentences $I$ when $A$ is an alphabet of one color. By definition, the colored monomial functions are analogues for the monomial quasisymmetric functions.  The fundamental quasisymmetric functions have a colored analogue, called the \emph{colored fundamental quasisymmetric functions}, that are defined as 
\begin{equation}\label{m_f_defs}
    F_I = \sum_{J \preceq I}M_J \text{\quad and \quad} M_I = \sum_{J \preceq I}(-1)^{\ell(J)-\ell(I)}F_J,
\end{equation} where the sums run over all sentences $J$ that are refinements of $I$.  The colored fundamental basis is dual to the colored ribbon basis with $\langle R_I, F_J \rangle = \delta_{I,J}$.

%%%%%%%%%%%%%%%%%%%%%%%%%%%%%%%%%%%%%%%%%%%%%%%%%%%%%%%%%%%%%%%%%%%%%%%%%%%%%%%%%%%%%%%%%
\section{A partially commutative generalization of the dual immaculate functions} \label{colordualsection}

To generalize the dual immaculate functions to $QSym_A$, we first define a colored generalization of tableaux.  These allow for a combinatorial definition of the colored dual immaculate functions, which then expand positively into the colored monomial and colored fundamental bases. Additionally, we define the colored immaculate descent graph and use it to give an expansion of the colored fundamental functions into the colored dual immaculate functions. In \cite{MasSeaLift}, Mason and Searles study a lift of the dual immaculate functions to the full polynomial ring. Our generalization of the dual immaculate functions is more aligned with the Hopf algebra-related aspects of the original functions whereas Mason and Searles' lift relates closely to slide polynomials, key polynomials, and Demazure atoms. The dual immaculate functions are the stable limit of their lifts while they are isomorphic to a special case of our lift.

\subsection{The colored dual immaculate basis of $QSym_A$}

\begin{defn}
For a sentence $J = (w_1, \ldots, w_k)$, the \textit{colored composition diagram} of shape $J$ is a composition diagram of $w\ell(J)$ where the $j^{\text{th}}$ box in row $i$ is colored, or filled, with the $j^{\text{th}}$ color in $w_i$.
\end{defn}

\begin{ex}
The colored composition diagram of shape $J=(aba,cb)$ for $a,b,c \in A$ is
$$\begin{ytableau}
    a & b & a\\
    c & b
\end{ytableau}$$
\end{ex}

\begin{defn}
For a sentence $I$, a \textit{colored immaculate tableau} (CIT) of shape $I$ is a colored composition diagram of $I$ filled with positive integers such that the integer entries in each row are weakly increasing from left to right and the entries in the first column are strictly increasing from top to bottom.
\end{defn}

\begin{defn}
    The \emph{type} of a CIT $T$ is a sentence $B=(u_1,\ldots ,u_j)$ that indicates how many boxes of each color are filled with each integer and in what order those boxes appear. That is, each word $u_i$ in $B$ is defined by starting in the lowest box containing an $i$ and reading the colors of all boxes containing $i$'s going from left to right, bottom to top. If no box is filled with the number $i$, then $u_i=\emptyset$. The flat type of $T$ is given by the flattening of $B$, denoted again by $\tilde{B}$.
\end{defn}

For a CIT $T$ of type $B=(u_1, \ldots, u_j)$, the monomial $x_T$ is defined $x_T = x_{u_1,1}x_{u_2,2} \cdots x_{u_j,j}$, which may also be denoted $x_B$.

\begin{ex}
    The colored immaculate tableaux of shape $J = (aba,cb)$ and type $B = (a,c, \emptyset,b, ba)$ are\vspace{2mm}
    $$
    \scalebox{.75}{
    \ytableausetup{boxsize=8mm}
    \begin{ytableau}
        a, 1 & b, 5 & a, 5\\
        c, 2 & b, 4
    \end{ytableau}
    \qquad \qquad
    \begin{ytableau}
        a, 1 & b, 4 & a, 5\\
        c, 2 & b, 5
    \end{ytableau}}
    $$
    Both tableaux are associated with the monomial $x_{a,1}x_{c,2}x_{b,4}x_{ba,5}$ and have the flat type $\tilde{B} = (a,c,b,ba)$.
\end{ex} \vspace{0mm}

\begin{defn}
For a sentence $J$, the \textit{colored dual immaculate function} is defined as $$\mathfrak{S}^*_J = \sum_T x_T,$$ where the sum is taken over all colored immaculate tableaux $T$ of shape $J$. 
\end{defn}

\begin{ex} For $J = (aba,cb)$, the colored dual immaculate function is $$\mathfrak{S}^*_{aba,cb} = x_{aba,1}x_{cb,2} + x_{ab,1}x_{cba,2} + x_{aba,1}x_{c,2}x_{b,3} + \ldots + 2x_{a,1}x_{c,2}x_{b,3}x_{ba,4} + \ldots.$$
\end{ex}

The colored dual immaculate functions map to the dual immaculate functions in $QSym$ under the uncoloring map $\upsilon $, thus we say the two bases are analogous.  

\begin{prop}
Let $A$ be an alphabet of one color, $A = \{a\}$, and $I$ be a sentence. Then, $$\upsilon (\mathfrak{S}^*_I) = \mathfrak{S}^*_{w\ell(I)}.$$ Moreover, $\{\mathfrak{S}^*_I\}_I$ in $QSym_A$ is analogous to $\{\mathfrak{S}^*_{\alpha}\}_{\alpha}$ in $QSym$.
\end{prop}

\begin{proof}
Observe that $\upsilon $ acts on a monomial $x_T$ where $T$ is a colored immaculate tableau of shape $I$ by mapping it to the monomial $x^{T'}$ where $T'$ is the immaculate tableau of shape $w\ell(I)$ with the same integer entries as $T$. Thus, $\upsilon (\mathfrak{S}^*_I) = \mathfrak{S}^*_{w\ell(I)}$ for all alphabets $A$ and more specifically alphabets $A$ containing only one color. 
\end{proof}

We now introduce further results on colored immaculate tableaux to provide a foundation for the expansions of the colored dual immaculate functions into other bases of $QSym$.

\begin{defn} \label{colorstd}
A \textit{standard colored immaculate tableau} (SCIT) of size $n$ is a colored immaculate tableau in which the integers $1$ through $n$ each appear exactly once.  The \emph{standardization} of a CIT $T$, denoted $std(T)$, is a standard colored immaculate tableau obtained by renumbering the boxes of $T$ in the order they appear in its type.
\end{defn}

\begin{ex} A few colored immaculate tableaux of shape $J = (ab,cb)$ together with their standardizations, which are the only standard colored immaculate tableaux of shape $J = (ab,cb)$, are:

$$
\scalebox{.75}{
\quad \quad
\text{$T_1$ =}
\begin{ytableau}
    a,1 & b, 2 \\
    c, 3 & b, 3
\end{ytableau}
\qquad \qquad 
\text{$T_2 =$}
\begin{ytableau}
    a,1 & b, 2 \\
    c, 2 & b, 3
\end{ytableau}
\qquad \qquad 
\text{$T_3 =$} 
\begin{ytableau}
    a,1 & b, 3 \\
    c, 2 & b, 2
\end{ytableau}
}
$$

\vspace{2mm}

$$\scalebox{.75}{
$std(T_1)=$
\begin{ytableau}
    a,1 & b, 2 \\
    c, 3 & b, 4
\end{ytableau}
\qquad 
$std(T_2)=$
\begin{ytableau}
    a,1 & b, 3 \\
    c, 2 & b, 4
\end{ytableau}
\qquad 
$std(T_3)=$
\begin{ytableau}
    a,1 & b, 4 \\
    c, 2 & b, 3
\end{ytableau}}
$$\vspace{-2mm}
\end{ex}

Standard colored immaculate tableaux share certain statistics and properties with non-colored standard immaculate tableaux. The number of SCIT of shape $J$ is the same as the number of standard immaculate tableaux of shape $w\ell(J)$, meaning both are counted by the same hook length formula in \cite{berg18}. Additionally, the notions of \emph{descent} and \emph{descent composition} for SCIT are the same as those in Definition \ref{desc}, simply disregarding color. However, we define an additional concept of the colored descent composition.

\begin{defn} \label{color_desc_comp}
 Let $T$ be a standard colored immaculate tableau of type $B$ with descent set $Des(T)=\{i_1,\ldots ,i_k\}$ for some $k \in \mathbb{Z}_{>0}$. The \textit{colored descent composition} of $T$, denoted $co_A(T)$, is the unique sentence obtained by splitting $w(B)$ after the $i_j^{\text{th}}$ entry for each $j \in [k]$.
\end{defn}

The colored descent composition can also be defined as the sentence obtained by reading through the colors of the tableau in the order that the boxes are numbered and splitting into a new word each time the next box is in a strictly lower row.  Note that for a SCIT $T$ of type $B$, the colored descent composition is the unique sentence for which $w\ell(co_A(T))=co(T)$ and $w(co_A(T)) = w(B)$.

\begin{ex}\label{color_std_ex} The standard colored immaculate tableaux of shape $(ab,cb)$, along with their descent sets and colored descent compositions, are:

$$
\scalebox{.75}{
$T_1 =$
\begin{ytableau}
    a, 1 & b, 2 \\
    c, 3 & b, 4
\end{ytableau}
\qquad \qquad
$T_2 =$ 
\begin{ytableau}
    a,1 & b, 3 \\
    c, 2 & b, 4
\end{ytableau}
\qquad \qquad
$T_3 =$
\begin{ytableau}
    a, 1 & b, 4 \\
    c, 2 & b, 3
\end{ytableau}}
$$

$$
\scalebox{.75}{
\ \quad \quad $\substack{Des(T_1) = \{2\},\\ co_A(T_1) = (ab,cb)}$ \qquad \ \quad \qquad
$\substack{Des(T_1) = \{1, 3\},\\
co_A(T_2)=(a,cb,b)}$
\qquad \ \ \ \ \qquad
$\substack{Des(T_3) = \{1\},\\ co_A(T_3)=(a,cbb)}$}$$
\end{ex}

\begin{prop}\label{samestandardization}
Let $T_1$ and $T_2$ be colored immaculate tableaux of shape $J$ and type $B$. Then, $T_1 = T_2$ if and only if $std(T_1)=std(T_2)$. 
\end{prop}
\begin{proof} It is trivial that $T_1 = T_2$ implies $std(T_1)=std(T_2)$.  Now, let $std(T_1) = std(T_2) = U$, meaning by definition that the boxes of $T_1$ appear in $B$ in the same order as the boxes of $T_2$.  The box $(i,j)$ in row $i$ and column $j$ in both tableaux is filled with the same integer $k$ and with the $k^{\text{th}}$ color in $w(B)$, thus $T_1 = T_2$.
\end{proof}

\begin{prop}\label{standardrefinement}
Let $U$ be a standard colored immaculate tableau of shape $J$. For a weak sentence $B$, there exists a colored immaculate tableau $T$ of shape $J$ and type $B$ that standardizes to $U$ if and only if $\tilde{B} \preceq co_A(U)$.
\end{prop}

\begin{proof}
$(\Rightarrow)$ Let $T$ be a colored immaculate tableau of shape $J$ and type $B$ such that $std(T)=U$.  Both $B$ and $co_A(U)$ are defined by the order that boxes appear in the type of $T$, thus they have the same maximum words $w(B)=w(co_A(U))$. Note that this also means the $i^{\text{th}}$ letter in $\tilde{B}$ and the $i^{\text{th}}$ letter in $co_A(U)$ correspond to the same box in $J$. Recall that $co_A(U)$ splits only after descents, and suppose that $co_A(U)$ splits after the $i^{\text{th}}$ letter. Then the $(i+1)^{\text{th}}$ letter is on a strictly lower row. Given that these entries correspond exactly to the $i^{\text{th}}$ and $(i+1)^{\text{th}}$ letter in $\tilde{B}$, this tells us that $\tilde{B}$ must also split since the following entry is on a lower row.  Thus $\tilde{B}$ also splits after every descent which implies that $\tilde{B} \preceq co_A(U)$.

$(\Leftarrow)$ Let $\tilde{B} = (v_1, \ldots , v_j) \preceq co_A(U)$ and let $v_i$ be the $n_i^{\text{th}}$ word in $B$. We create a colored immaculate tableau $T$ of shape $J$ and type $B$ that standardizes to $U$ by filling the boxes of $T$ in the order they are numbered in $U$.  The first $|v_1|$ boxes are labeled with $n_1$'s, the next $|v_2|$ boxes are labeled with $n_2$'s, and continue this process until the last $|v_j|$ boxes are labeled with $n_j$'s.  Since $\tilde{B} \preceq co_A(U)$, each time there is a descent in $U$ the number being filled in must increase. This maintains the order of the boxes in the type from $U$, meaning $T$ standardizes to $U$. This filling also maintains the strictly increasing condition on the first column and the weakly increasing condition on each row by construction.  Therefore, $T$ is a colored immaculate tableau of shape $J$ and type $B$ with $std(T)=U$. 
\end{proof}

%%%%%%%%%%%%%%%%%%%%%%%%%%%%%%%%%%%%%%%%%%%%%%%%%%%%%%%%%%%%%%%%%
\subsection{Expansion into the colored monomial and colored fundamental bases}

The colored dual immaculate functions have positive expansions into the colored monomial and colored fundamental bases. Their coefficients are determined combinatorially using colored immaculate tableaux.

\subsubsection{Expansion into the colored monomial functions}

First, we establish the relationship between the colored monomial quasisymmetric functions and colored immaculate tableaux. Then, we define coefficients counting colored immaculate tableaux and prove our expansion.  Finally, the transition matrix of these coefficients leads to a proof that the colored dual immaculate functions are indeed a basis of $QSym_A$.

\begin{prop}\label{monomialtab} For a sentence $B$, consider a standard colored immaculate tableau $U$ such that $B \preceq co_A(U)$. Then, $$M_B = \sum_T x_T,$$  where the sum runs over all colored immaculate tableaux $T$ such that $std(T)=U$ and $\widetilde{type(T)}=B$. 
\end{prop}

\begin{proof}
Consider a standard colored immaculate tableau $U$ and a sentence $B = (v_1,\ldots,v_h)$ such that $B \preceq co_A(U)$. By definition,  $$M_B = \sum_{1 \leq j_1 < \ldots < j_h} x_{v_1,j_1}\ldots x_{v_h,j_h}.$$ Each monomial $x_{v_1,j_1}\ldots x_{v_h,j_h}$ is equal to $x_T$ where $T$ is the unique (by Proposition \ref{samestandardization}) colored immaculate tableau such that $std(T)=U$ and its type $C = (u_1, \ldots, u_g)$ is the sentence where word $u_{j_i}$ is equal to $v_i$ for $1 \leq i \leq h$ and all other words are empty. This includes a tableau $T$ for every sentence $C$ such that $\tilde{C}=B$.  Thus, the above sum is equivalent to summing $x_T$ over all CIT $T$ with type $C$ such that $std(T)=U$ and $\tilde{C}=B$.
\end{proof}

\begin{ex}\label{idk}
The colored immaculate tableaux of shape $J=(ab,cb)$ and type $B=(a,cb,b)$ are

$$\scalebox{.75}{ 
$T_1 =$ \begin{ytableau}
    a, 1 & b, 3 \\
    c, 2 & b, 2
\end{ytableau}
\quad \quad
$T_2 =$ \begin{ytableau}
    a, 1 & b, 2 \\
    c, 2 & b, 3
\end{ytableau}
}$$ \vspace{0mm}

The tableaux $T_1$ and $T_2$ have the same shape and type, but different standardizations (see Example \ref{color_std_ex}). Now, consider all tableaux with types that flatten to $B$ and standardizations equal to $std(T_1)$:

$$
\scalebox{.75}{
\begin{ytableau}
    a, 1 & b, 3 \\
    c, 2 & b, 2
\end{ytableau}
\quad 
\begin{ytableau}
    a, 1 & b, 4 \\
    c, 2 & b, 2
\end{ytableau}
\quad $\cdots$ \quad
\begin{ytableau}
    a, 1 & b, 4 \\
    c, 3 & b, 3
\end{ytableau}
\quad $\cdots$ \quad
\begin{ytableau}
    a, 2 & b, 5 \\
    c, 4 & b, 4
\end{ytableau}
\quad $\cdots$ \quad
\begin{ytableau}
    a, 5 & b, 9 \\
    c,7 & b, 7
\end{ytableau}
}
$$\vspace{0mm}

A single monomial function $M_{(a,cb,b)}$ can be associated with this set of tableaux.  The tableaux of flat type $B$ that standardize to $T_2$ are also represented by a function $M_{(a,cb,b)}$. Thus, when finding the overall monomial expansion for $\mathfrak{S}^*_{(ab,cb)}$, the tableaux of flat type $(a,cb,b)$ contribute to the sum as the term $2M_{(a,cb,b)}$. 
\end{ex}

\begin{defn} For a sentence $J$ and weak sentence $B$, define the \emph{colored immaculate Kostka coefficient} $K_{J,B}$ as the number of colored immaculate tableaux of shape $J$ and type $B$. 
\end{defn}

\begin{prop}\label{typenumber}
Let $J=(w_1, \ldots , w_k)$ and $C = (u_1,\ldots ,u_g)$ be sentences.  Then, $K_{J,C} = K_{J,\tilde{C}}$.
\end{prop}

\begin{proof}
Suppose $\tilde{C}=(v_1,\ldots,v_h)$ where $u_{i_1} = v_1$, \ldots  , $u_{i_h}=v_h$ for some $i_1 < \ldots  < i_h$, with all other $u_j = \emptyset$. We define a map from the colored immaculate tableaux of shape $J$ and type $\tilde{C}$ to the colored immaculate tableaux of shape $J$ and type $C$. Given a colored immaculate tableau $T$ of shape $J$ and type $\tilde{C}$, replace each 1 with $i_1$, 2 with $i_2$, \ldots, and $h$ with $i_h$.  This produces a tableau $T'$ of shape $J$ and type $C$. The inverse map takes a tableau $T'$ of shape $J$ and type $C$ and changes each $i_1$ to 1, $i_2$ to 2, $\ldots$, and $i_h$ to $h$, which yields the initial tableau $T$ of shape $J$ and type $\tilde{C}$. This is a bijection, meaning $K_{J,C} = K_{J,\tilde{C}}$. 
\end{proof} 

\begin{ex}\label{equaltypenumber}
Let $J=(ab,cb)$ and $B=(\emptyset,a,\emptyset,cb,b)$.  Then, $K_{J,B} = 2$ because the colored immaculate tableaux of shape $J$ and type $B$ are:
$$\bigtableau{2,a&5,b\\4,c&4,b}\ \ \ \bigtableau{2,a&4,b\\4,c&5,b}$$\\
Notice that $K_{J, \tilde{B}} =2$ as well, since the colored immaculate tableaux of shape $J$ and type $\tilde{B}$ are:
$$\bigtableau{1,a&3,b\\2,c&2,b}\ \ \ \bigtableau{1,a&2,b\\2,c&3,b}$$
\end{ex}

\begin{thm}\label{monomialexpansion}
For a sentence $J$, the colored dual immaculate function $\mathfrak{S}^*_J$ expands positively into the colored monomial basis as
$$\mathfrak{S}^*_J = \sum_{B}K_{J,B}M_{B},$$ where the sum is taken over all sentences $B$ such that $|B|=|J|$. 
\end{thm} 

\begin{proof}
Let $B_1,\ldots ,B_j$ be all possible flat types of colored immaculate tableaux of shape $J$. Then arrange the sum  $\mathfrak{S}^*_J = \sum_T x_T$ into parts based on the flat types of the tableaux $T$ as 
$$\mathfrak{S}^*_J = \sum_{\ \widetilde{type(T)}=B_1}x_T + \ldots  + \sum_{ \widetilde{type(T)}=B_j}x_T.$$  
Consider the sum  of $x_T$ over $T$ such that $\widetilde{type(T)}=B_i$.   By Proposition \ref{typenumber}, for any $C$ such that $\tilde{C}=B$ we have $K_{J,B_i} = K_{J,C}$. By definition, for any flat sentence B, $$M_{B}=\sum_{\tilde{C}=B}x_C.$$  Thus, we can write $$\sum_{\widetilde{type(T)}=B_i}x_T = 
\sum_{\tilde{C}=B_i} K_{J,C}x_C = K_{J,B_i}\left(\sum_{\tilde{C}=B_i}x_C\right) = K_{J,B_i}M_{B_i}.$$ 
Therefore the overall sum becomes
$$\mathfrak{S}^*_J = K_{J,B_1}M_{B_1} + \ldots  + K_{J,B_j}M_{B_j} = \sum_B K_{J,B}M_B,$$ where the sum runs over all flat types  $B$ of the colored immaculate tableaux of shape $J$. For all other $B$ such that $|B|=|J|$, we have $K_{J,B}=0$ and we can extend this sum to be over all sentences $B$ such that $|B|=|J|$.
\end{proof}

\begin{thm}
The set of colored dual immaculate functions forms a basis for $QSym_A$.
\end{thm}

\begin{proof}
Let $A$ be an alphabet with a total ordering, and consider the transition matrix $\mathcal{K}$ from $\{\mathfrak{S}^*_I\}_I$ to $\{M_I\}_I$. By Theorem \ref{monomialexpansion}, the entry of $\mathcal{K}$ in row $J$ and column $C$ is $K_{J,C}$. We want to prove that $\mathcal{K}$ is upper unitriangular and thus invertible when the rows and columns are ordered first by the reverse lexicographic order of compositions applied to word lengths, then by lexicographic order on words. 
%For example, row $(a_1a_2a_3,a_4a_5)$ would come before row $(a_1a_2,a_3,a_4a_5)$ because $(3,2)\preceq_{r \ell} (2,1,2)$, and, given $a_1 < a_2 < \ldots <a_5$, row $(a_1a_2a_2,a_1a_3)$ would come before row $(a_1a_2a_3,a_4a_5)$ because $(3,2)=(3,2)$ and $(a_1a_2a_2) \preceq_{\ell} (a_1a_2a_3)$,

Let $J=(w_1, \ldots, w_k)$ and $C=(v_1,\ldots,v_h)$ be sentences with $|J|=|C|$.  We claim that if $wl(J)\succeq_{rl} wl(C)$ and $K_{J,C}\not=0$ then $J=C$. Assume there exists a tableau $T$ of shape $J$ and type $C$ with $wl(J)\succeq_{rl} wl(C)$ and $wl(J) \not= wl(C)$.  Then $|w_1|\leq|v_1|$.  Observe that the first row of the tableau has $|w_1|$ boxes and so if $|w_1|<|v_1|$, there would have to be a 1 placed in a box somewhere below row 1.  This is impossible by the conditions on colored immaculate tableaux so $|w_1|=|v_1|$ and every box in row 1 is filled with 1's.  Next, $|w_2| \leq |v_2|$ and so the second row must start with a 2 for any 2's to exist in $T$. This implies that the first entry in each subsequent row is greater than 2 meaning that no other row can contain 2's.  If every 2 is in the second row then and the number of 2's is at least $w_2$, then $|w_2|=|v_2|$.  Continuing this reasoning, $|w_i|=|v_i|$ for $1 \leq i \leq k$.  Thus, $wl(J)=wl(C)$.
Further, by this method, we have filled the first row with 1's, the second row with 2's, the $i^{\text{th}}$ row with $i$'s, etc. to construct a colored immaculate tableau such that $w_i=v_i$ for all $i$.  Therefore, $J=C$.  By construction, this is the only tableau of shape $J$ and type $J$ so $K_{J,J}=1$. To summarize, we have shown that $K_{J,C} = 0$ when $wl(J)\succeq_{l} wl(C)$ unless $J=C$, in which case the entry of the matrix lies on the diagonal and $K_{J,J}=1$.  Thus, we have proved $\mathcal{K}$ is upper unitriangular.  
\end{proof} 

\subsubsection{Expansion into the colored fundamental functions}
To expand the colored dual immaculate functions into the colored fundamental basis we first define coefficients counting SCIT. Relating these to our earlier coefficients counting colored immaculate tableaux, we reformulate our expansion in Theorem \ref{monomialexpansion} to an expansion in terms of the colored fundamental basis.

\begin{defn} \label{colorL}
For sentences $J$ and $C$, define $L_{J,C}$ as the number of standard colored immaculate tableaux of shape $J$ that have colored descent composition $C$.
\end{defn}

\begin{ex}\label{exL2}
Let $J=(ab,cb,b)$ and $C=(a,cb,bb)$.  The standard colored immaculate tableaux of shape $J$ with colored descent composition $C$ are 

$$
\scalebox{.75}{
$U_1 =$ 
\begin{ytableau}
    a, 1 & b, 3 \\
    c, 2 & b, 5 \\
    b, 4
\end{ytableau}
\quad \quad \quad
$U_2 =$ 
\begin{ytableau}
    a, 1 & b, 5 \\
    c, 2 & b, 3 \\
    b, 4
\end{ytableau}}
$$
Thus, $L_{(ab,cb,b),(a,cb,bb)}=2$.
\end{ex}

\begin{prop}\label{coeffsum}
For sentences $J$ and $B$, $$K_{J,B}=\sum_{C \succeq B}L_{J,C}.$$
\end{prop}

\begin{proof}
Recall that $K_{J,B}$ is the number of colored immaculate tableaux of shape $J$ and type $B$.  We want to show that $K_{J,B}$ is equal to the sum of $L_{J,C}$, the number of standard colored immaculate tableaux of shape $J$ and descent composition $C$, overall $C \succeq B$. For this proof, let $\mathcal{T}$ be the set of all colored immaculate tableaux of shape $J$ and type $B$, and let $\mathcal{U}$ be the set of standard colored immaculate tableaux $U$ of shape $J$ and descent composition $C$ with $C \succeq B$. We need to show that the map $std: \mathcal{T} \rightarrow \mathcal{U}$, where $std$ is the standardization map from Definition \ref{colorstd}, is a bijection on these sets. By Proposition \ref{samestandardization}, colored immaculate tableaux with the same shape and type must have different standardizations or they would be the same tableau, thus our map is injective.  By Proposition \ref{standardrefinement}, the map is surjective.  This makes our map a bijection and so $\mathcal{T}$ and $\mathcal{U}$ have the same size.  Thus, we have shown that $K_{J,B}=\sum_{C \succeq B}L_{J,C}$.
\end{proof}

\begin{thm} \label{fund_exp}
For a sentence $J$, the colored dual immaculate function $\mathfrak{S}^*_J$ expands positively into the fundamental basis as $$ \mathfrak{S}^*_J = \sum_{C} L_{J,C} F_{C},$$
where the sum runs over sentences $C$ such that $|C|=|J|$. 
\end{thm}

\begin{proof}
Let $J$ be a sentence. First, observe that applying the Möbius inversion to Proposition \ref{coeffsum} yields $$L_{J,C} = \sum_{C \preceq B} (-1)^{\ell(C)-\ell(B)}K_{J,B}.$$

Then, by Theorem \ref{monomialexpansion} and Equation \eqref{m_f_defs}, $$\mathfrak{S}^*_J = \sum_B K_{J,B}M_B = \sum_B K_{J,B} \left(\sum_{C \preceq B} (-1)^{\ell(C) - \ell(B)} F_C\right) = \sum_C \left( \sum_{C \preceq B} (-1)^{\ell(C)-\ell(B)}K_{J,B} \right) F_C = \sum_C L_{J,C} F_C.$$
\end{proof}

This expansion can be written as a sum over all standard colored immaculate tableaux of a certain shape instead of using coefficients to count tableaux based on their colored descent compositions.

\begin{cor}
For a sentence $J$, 
$$\mathfrak{S}^*_J = \sum_U F_{co_A(U)},$$ where the sum runs over all standard colored immaculate tableaux $U$ of shape $J$.
\end{cor}

\subsection{The colored immaculate descent graph} \label{cidg}

We define the colored immaculate descent graph to directly determine the expansion of the colored fundamental functions into the colored dual immaculate basis. Additionally, our result specializes to a new combinatorial expansion of the fundamental quasisymmetric functions into the dual immaculate functions. 

\begin{defn} \label{color_imm_desc_graph}
Define the \emph{colored immaculate descent graph}, denoted $\mathfrak{D}_{A}^n$, as an edge-weighted directed simple graph such that the vertex set is the set of sentences in $A$ of size $n$, and there is a directed edge from $I$ to $J$, denoted $I \rightarrow J$ or $J \leftarrow I$, if there exists a standard colored immaculate tableau of shape $I$ with colored descent composition $J$. The edge from $I$ to $J$ is weighted with the coefficient $L_{I,J}$ from Definition \ref{colorL}. For a path $\mathcal{P}$ in $\mathfrak{D}_{A}^n$, let $prod(\mathcal{P})$ denote the product of the edge-weights in $\mathcal{P}$ and let $prod(\mathcal{\emptyset})=1$.
\end{defn}

\begin{ex} In Figure 1 we illustrate the subgraph of $\mathfrak{D}^5_{\{a,b,c\}}$ with top vertex $(ab,cbb)$. In this subgraph, all edges are weighted $1$ because $L_{I,J} = 1$ for each $I$ and $J$ (and thus $prod(\mathcal{P})=1$ for all paths) but, for example, the edge from $(ab,cb,b)$ to $(a,cb,bb)$ would be $2$ since $L_{(ab,cb,b)(a,cb,bb)}=2$ as in Example \ref{exL2}.

\begin{figure}[h!]
\tikzset{every picture/.style={line width=0.75pt}} %set default line width to 0.75pt        
\scalebox{.75}{
\begin{tikzpicture}[x=0.75pt,y=0.75pt,yscale=-1,xscale=1]
%uncomment if require: \path (0,615); %set diagram left start at 0, and has height of 615

%Shape: Rectangle [id:dp9746015143292242] 
\draw   (291,30) -- (311,30) -- (311,50) -- (291,50) -- cycle ;
%Shape: Rectangle [id:dp2975849923553231] 
\draw   (311,30) -- (331,30) -- (331,50) -- (311,50) -- cycle ;
%Shape: Rectangle [id:dp30314133301193014] 
\draw   (311,50) -- (331,50) -- (331,70) -- (311,70) -- cycle ;
%Shape: Rectangle [id:dp6581535212995715] 
\draw   (291,50) -- (311,50) -- (311,70) -- (291,70) -- cycle ;
%Shape: Rectangle [id:dp9037568112022336] 
\draw   (331,50) -- (351,50) -- (351,70) -- (331,70) -- cycle ;
%Shape: Rectangle [id:dp5762206208429126] 
\draw   (189.5,111) -- (209.5,111) -- (209.5,131) -- (189.5,131) -- cycle ;
%Shape: Rectangle [id:dp2838379314935835] 
\draw   (189.5,131) -- (209.5,131) -- (209.5,151) -- (189.5,151) -- cycle ;
%Shape: Rectangle [id:dp8950192129270562] 
\draw   (189.5,151) -- (209.5,151) -- (209.5,171) -- (189.5,171) -- cycle ;
%Shape: Rectangle [id:dp34259616509318813] 
\draw   (209.5,131) -- (229.5,131) -- (229.5,151) -- (209.5,151) -- cycle ;
%Shape: Rectangle [id:dp6756401510475543] 
\draw   (229.5,131) -- (249.5,131) -- (249.5,151) -- (229.5,151) -- cycle ;
%Shape: Rectangle [id:dp871587154072083] 
\draw   (391,111) -- (411,111) -- (411,131) -- (391,131) -- cycle ;
%Shape: Rectangle [id:dp10913225543138982] 
\draw   (391,131) -- (411,131) -- (411,151) -- (391,151) -- cycle ;
%Shape: Rectangle [id:dp8510006848240161] 
\draw   (411,131) -- (431,131) -- (431,151) -- (411,151) -- cycle ;
%Shape: Rectangle [id:dp5833780773901733] 
\draw   (431,131) -- (451,131) -- (451,151) -- (431,151) -- cycle ;
%Shape: Rectangle [id:dp12208768237454248] 
\draw   (300,211) -- (320,211) -- (320,231) -- (300,231) -- cycle ;
%Shape: Rectangle [id:dp15612630554147677] 
\draw   (300,231) -- (320,231) -- (320,251) -- (300,251) -- cycle ;
%Shape: Rectangle [id:dp2856197042323023] 
\draw   (300,251) -- (320,251) -- (320,271) -- (300,271) -- cycle ;
%Shape: Rectangle [id:dp35279735058702877] 
\draw   (320,231) -- (340,231) -- (340,251) -- (320,251) -- cycle ;
%Shape: Rectangle [id:dp8124373113166679] 
\draw   (320,251) -- (340,251) -- (340,271) -- (320,271) -- cycle ;
%Shape: Rectangle [id:dp6309153038105233] 
\draw   (390,295) -- (410,295) -- (410,315) -- (390,315) -- cycle ;
%Shape: Rectangle [id:dp17665055751411307] 
\draw   (390,315) -- (410,315) -- (410,335) -- (390,335) -- cycle ;
%Shape: Rectangle [id:dp694169643293634] 
\draw   (390,335) -- (410,335) -- (410,355) -- (390,355) -- cycle ;
%Shape: Rectangle [id:dp7217961052404935] 
\draw   (410,335) -- (430,335) -- (430,355) -- (410,355) -- cycle ;
%Shape: Rectangle [id:dp543747572291549] 
\draw   (390,355) -- (410,355) -- (410,375) -- (390,375) -- cycle ;
%Shape: Rectangle [id:dp941217117640911] 
\draw   (301,430) -- (321,430) -- (321,450) -- (301,450) -- cycle ;
%Shape: Rectangle [id:dp3485056535868607] 
\draw   (301,450) -- (321,450) -- (321,470) -- (301,470) -- cycle ;
%Shape: Rectangle [id:dp4189205073232498] 
\draw   (301,470) -- (321,470) -- (321,490) -- (301,490) -- cycle ;
%Shape: Rectangle [id:dp021175860106261712] 
\draw   (301,490) -- (321,490) -- (321,510) -- (301,510) -- cycle ;
%Shape: Rectangle [id:dp9217228684034253] 
\draw   (321,490) -- (341,490) -- (341,510) -- (321,510) -- cycle ;
%Shape: Rectangle [id:dp885227711667715] 
\draw   (189,310) -- (209,310) -- (209,330) -- (189,330) -- cycle ;
%Shape: Rectangle [id:dp11637030279490501] 
\draw   (189,330) -- (209,330) -- (209,350) -- (189,350) -- cycle ;
%Shape: Rectangle [id:dp4494899590666239] 
\draw   (189,350) -- (209,350) -- (209,370) -- (189,370) -- cycle ;
%Shape: Rectangle [id:dp6997256894266561] 
\draw   (209,350) -- (229,350) -- (229,370) -- (209,370) -- cycle ;
%Shape: Rectangle [id:dp5570685599395608] 
\draw   (229,350) -- (249,350) -- (249,370) -- (229,370) -- cycle ;

\draw    (286,76.5) -- (246.91,115.59) ;
\draw [shift={(245.5,117)}, rotate = 315] [color={rgb, 255:red, 0; green, 0; blue, 0 }  ][line width=0.75]    (10.93,-4.9) .. controls (6.95,-2.3) and (3.31,-0.67) .. (0,0) .. controls (3.31,0.67) and (6.95,2.3) .. (10.93,4.9)   ;
%Straight Lines [id:da7263684485202386] 
\draw    (360,79.5) -- (379.08,98.34) ;
\draw [shift={(380.5,99.75)}, rotate = 224.65] [color={rgb, 255:red, 0; green, 0; blue, 0 }  ][line width=0.75]    (10.93,-4.9) .. controls (6.95,-2.3) and (3.31,-0.67) .. (0,0) .. controls (3.31,0.67) and (6.95,2.3) .. (10.93,4.9)   ;
%Straight Lines [id:da2042970315806194] 
\draw    (320,80.5) -- (320,197.5) ;
\draw [shift={(320,199.5)}, rotate = 270] [color={rgb, 255:red, 0; green, 0; blue, 0 }  ][line width=0.75]    (10.93,-4.9) .. controls (6.95,-2.3) and (3.31,-0.67) .. (0,0) .. controls (3.31,0.67) and (6.95,2.3) .. (10.93,4.9)   ;
%Straight Lines [id:da7174749420352136] 
\draw    (250.5,160.25) -- (289.09,198.84) ;
\draw [shift={(290.5,200.25)}, rotate = 225] [color={rgb, 255:red, 0; green, 0; blue, 0 }  ][line width=0.75]    (10.93,-4.9) .. controls (6.95,-2.3) and (3.31,-0.67) .. (0,0) .. controls (3.31,0.67) and (6.95,2.3) .. (10.93,4.9)   ;
%Straight Lines [id:da3453043428608944] 
\draw    (349,279.5) -- (377.62,309.55) ;
\draw [shift={(379,311)}, rotate = 226.4] [color={rgb, 255:red, 0; green, 0; blue, 0 }  ][line width=0.75]    (10.93,-4.9) .. controls (6.95,-2.3) and (3.31,-0.67) .. (0,0) .. controls (3.31,0.67) and (6.95,2.3) .. (10.93,4.9)   ;
%Straight Lines [id:da5957106468272875] 
\draw    (290.5,280.25) -- (232.28,329.21) ;
\draw [shift={(230.75,330.5)}, rotate = 319.94] [color={rgb, 255:red, 0; green, 0; blue, 0 }  ][line width=0.75]    (10.93,-4.9) .. controls (6.95,-2.3) and (3.31,-0.67) .. (0,0) .. controls (3.31,0.67) and (6.95,2.3) .. (10.93,4.9)   ;
%Straight Lines [id:da12419212275295899] 
\draw    (380,380.75) -- (332.07,418.27) ;
\draw [shift={(330.5,419.5)}, rotate = 321.95] [color={rgb, 255:red, 0; green, 0; blue, 0 }  ][line width=0.75]    (10.93,-4.9) .. controls (6.95,-2.3) and (3.31,-0.67) .. (0,0) .. controls (3.31,0.67) and (6.95,2.3) .. (10.93,4.9)   ;
%Shape: Rectangle [id:dp6960160441531655] 
\draw   (451,131) -- (471,131) -- (471,151) -- (451,151) -- cycle ;

%Straight Lines [id:da45691290369904936] 
\draw    (210.02,284) -- (211,189) ;
\draw [shift={(210,286)}, rotate = 270.59] [color={rgb, 255:red, 0; green, 0; blue, 0 }  ][line width=0.75]    (10.93,-4.9) .. controls (6.95,-2.3) and (3.31,-0.67) .. (0,0) .. controls (3.31,0.67) and (6.95,2.3) .. (10.93,4.9)   ;
%\draw   (451,131) -- (471,131) -- (471,151) -- (451,151) -- cycle ;

% Text Node
\draw (194.25,315) node [anchor=north west][inner sep=0.75pt]   [align=left] {a};
% Text Node
\draw (296.25,34.5) node [anchor=north west][inner sep=0.75pt]   [align=left] {a};
% Text Node
\draw (395.25,299.5) node [anchor=north west][inner sep=0.75pt]   [align=left] {a};
% Text Node
\draw (305.25,215.5) node [anchor=north west][inner sep=0.75pt]   [align=left] {a};
% Text Node
\draw (306.25,434.5) node [anchor=north west][inner sep=0.75pt]   [align=left] {a};
% Text Node
\draw (396,116) node [anchor=north west][inner sep=0.75pt]   [align=left] {a};
% Text Node
\draw (195,116) node [anchor=north west][inner sep=0.75pt]   [align=left] {a};
% Text Node
\draw (395.5,320.5) node [anchor=north west][inner sep=0.75pt]   [align=left] {c};
% Text Node
\draw (306.5,455.5) node [anchor=north west][inner sep=0.75pt]   [align=left] {c};
% Text Node
\draw (195,136.5) node [anchor=north west][inner sep=0.75pt]   [align=left] {c};
% Text Node
\draw (305,237) node [anchor=north west][inner sep=0.75pt]   [align=left] {c};
% Text Node
\draw (194,336) node [anchor=north west][inner sep=0.75pt]   [align=left] {c};
% Text Node
\draw (296.5,55.5) node [anchor=north west][inner sep=0.75pt]   [align=left] {c};
% Text Node
\draw (396,137.5) node [anchor=north west][inner sep=0.75pt]   [align=left] {c};
% Text Node
\draw (316.5,32.5) node [anchor=north west][inner sep=0.75pt]   [align=left] {b};
% Text Node
\draw (235.5,133.5) node [anchor=north west][inner sep=0.75pt]   [align=left] {b};
% Text Node
\draw (316.67,52.33) node [anchor=north west][inner sep=0.75pt]   [align=left] {b};
% Text Node
\draw (336.5,52.5) node [anchor=north west][inner sep=0.75pt]   [align=left] {b};
% Text Node
\draw (215.5,133.5) node [anchor=north west][inner sep=0.75pt]   [align=left] {b};
% Text Node
\draw (416.5,133.5) node [anchor=north west][inner sep=0.75pt]   [align=left] {b};
% Text Node
\draw (195.5,153) node [anchor=north west][inner sep=0.75pt]   [align=left] {b};
% Text Node
\draw (436.5,133.5) node [anchor=north west][inner sep=0.75pt]   [align=left] {b};
% Text Node
\draw (456.5,133.5) node [anchor=north west][inner sep=0.75pt]   [align=left] {b};
% Text Node
\draw (326,234) node [anchor=north west][inner sep=0.75pt]   [align=left] {b};
% Text Node
\draw (306,253.5) node [anchor=north west][inner sep=0.75pt]   [align=left] {b};
% Text Node
\draw (326,253.5) node [anchor=north west][inner sep=0.75pt]   [align=left] {b};
% Text Node
\draw (194.5,352.5) node [anchor=north west][inner sep=0.75pt]   [align=left] {b};
% Text Node
\draw (214,353) node [anchor=north west][inner sep=0.75pt]   [align=left] {b};
% Text Node
\draw (234.5,352.5) node [anchor=north west][inner sep=0.75pt]   [align=left] {b};
% Text Node
\draw (395.5,337.5) node [anchor=north west][inner sep=0.75pt]   [align=left] {b};
% Text Node
\draw (395.5,357.5) node [anchor=north west][inner sep=0.75pt]   [align=left] {b};
% Text Node
\draw (415.5,337.5) node [anchor=north west][inner sep=0.75pt]   [align=left] {b};
% Text Node
\draw (306.5,472.5) node [anchor=north west][inner sep=0.75pt]   [align=left] {b};
% Text Node
\draw (306.5,492.5) node [anchor=north west][inner sep=0.75pt]   [align=left] {b};
% Text Node
\draw (326.5,492.5) node [anchor=north west][inner sep=0.75pt]   [align=left] {b};
\draw (252,73) node [anchor=north west][inner sep=0.75pt]   [align=left] {1};
% Text Node
\draw (374,69) node [anchor=north west][inner sep=0.75pt]   [align=left] {1};
% Text Node
\draw (328,127) node [anchor=north west][inner sep=0.75pt]   [align=left] {1};
% Text Node
\draw (274,155) node [anchor=north west][inner sep=0.75pt]   [align=left] {1};
% Text Node
\draw (189,224) node [anchor=north west][inner sep=0.75pt]   [align=left] {1};
% Text Node
\draw (250,280) node [anchor=north west][inner sep=0.75pt]   [align=left] {1};
% Text Node
\draw (366,274) node [anchor=north west][inner sep=0.75pt]   [align=left] {1};
% Text Node
\draw (340,380) node [anchor=north west][inner sep=0.75pt]   [align=left] {1};
\end{tikzpicture}}
\caption{A subgraph of $\mathfrak{D}^5_{\{a,b,c\}}$.}
\end{figure}
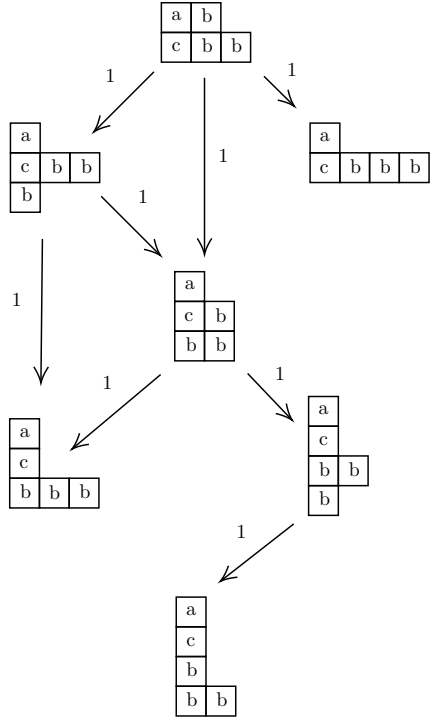

The element $(ab,cbb) \in \mathfrak{D}^5_{\{a,b,c\}}$ has edges going down to elements $(a,cbb,b)$, $(a,cb,bb)$, and $(a,cbbb)$ because these sentences represent possible descent compositions (with the exception of $(ab,cbb)$ itself) of colored standard immaculate tableaux of shape $(ab,cbb)$ as shown below.

$$
\scalebox{.75}{ 
\begin{ytableau}
    a, 1 & b, 4 \\
    c, 2 & b, 3 & b, 5 \\
\end{ytableau}
\quad \quad \quad
\begin{ytableau}
    a, 1 & b, 3 \\
    c, 2 & b, 4 & b, 5 \\
\end{ytableau}
\quad \quad \quad
\begin{ytableau}
    a, 1 & b, 5 \\
    c, 2 & b, 3 & b, 4 \\
\end{ytableau}
\quad \quad \quad
\begin{ytableau}
    a, 1 & b, 2 \\
    c, 3 & b, 4 & b, 5 \\
\end{ytableau}
}
$$\vspace{0mm}    
$$ (a,cbb,b) \text{ \qquad \qquad} (a,cb,bb) \text{ \quad \qquad} (a,cbbb) \text{ \qquad \qquad} (ab,cbb) \text{\quad}$$
\end{ex}
\vspace{1mm}

We say a sentence $K$ is \emph{reachable} from a sentence $I$ if there is a directed path from $I$ to $K$. This includes the empty path, meaning that $I$ is reachable from itself.

\begin{thm} \label{color_fund_to_imm}
    For a sentence $I$ of size $n$, the colored fundamental functions expand into the colored dual immaculate basis as $$F_I = \sum_{K}L^{-1}_{I,K} \mathfrak{S}^*_K \text{\qquad with coefficients \qquad} L^{-1}_{I,K} = \sum_{\mathcal{P}} (-1)^{\ell(\mathcal{P})}prod(\mathcal{P}),$$ where the sums run over all sentences $K$ reachable from $I$ in 
    $\mathfrak{D}^n_{A}$ and directed paths $\mathcal{P}$ from $I$ to $K$ in $\mathfrak{D}^n_A$.
\end{thm}

\begin{proof}
We proceed by induction on the length of the longest path starting at $I$ in $\mathfrak{D}^n_{A}$, denoted here with $k$. If $k=0$, there are no elements reachable from $I$ so $F_I = \mathfrak{S}^*_I$ which agrees with Theorem \ref{fund_exp}. Now for some positive integer $k$, assume the statement is true for any path of length $\leq k$. Consider a sentence $I$ where the length of the longest path starting at $I$ is $k+1$. By Theorem \ref{fund_exp}, $$F_I = \mathfrak{S}^*_I - \sum_J L_{I,J} F_{J},$$ where the sum runs over all sentences $J \not= I$ such that $|J|=|I|$. We only need to consider, however, sentences $J$ that are descent compositions of a SCIT of shape $I$ because otherwise $L_{I,J}=0$. Since there is an edge from $I$ to each of these $J$'s, the length of the longest path starting at any $J$ is at most $k$.  Thus, by induction,
$$F_I = \mathfrak{S}^*_I - \sum_J L_{I,J} \sum_K L^{-1}_{J,K} \mathfrak{S}^*_K,$$  for all sentences $K$ reachable from $J$ and $L^{-1}_{J,K} = \sum_{\mathcal{P}} (-1)^{\ell(\mathcal{P})}L_{K_1,K_2} \cdots L_{K_{j-1}, K_{j}}$ for paths $\mathcal{P} = \{K = K_j \leftarrow K_{j-1} \leftarrow \ldots \leftarrow K_1=J\}$ from $K$ to $J$. Note that $$ - \sum_J L_{I,J} \sum_K L^{-1}_{J,K} = \sum_{\mathcal{P}} (-1)^{\ell(\mathcal{P})}L_{I,J}L_{K_1,K_2}L_{K_2,K_3} \cdots L_{K_{j-1},K_{j}} = L^{-1}_{I,K},$$ where the sum runs over all paths $\mathcal{P} = \{K = K_j \leftarrow \ldots \leftarrow K_1 = J \leftarrow I\}$ from $K$ to $I$. Then, $$F_I = \sum_K L^{-1}_{I,K}\mathfrak{S}^*_K,$$ summing over all sentences $K $ reachable from $I$.
\end{proof}

\begin{ex} The subgraph in Figure 1 yields the following expansion of $F_{(ab,cbb)}$:
$$F_{(ab,cbb)} = \mathfrak{S}^*_{(ab,cbb)} - \mathfrak{S}^*_{(a,cbb,b)} + \mathfrak{S}^*_{(a,c,bbb)} - \mathfrak{S}^*_{(a,cbbb)}.$$
\end{ex}

Similarly, the (non-colored) immaculate descent graph $\mathfrak{D}^n$ can be defined as the graph with a vertex set of compositions of size $n$ where there is an edge from $\alpha$ to $\beta$ if there exists an s standard immaculate tableau of shape $\alpha$ with descent composition $\beta$. The edge from $\alpha$ to $\beta$ will be weighted with coefficient $L_{\alpha,\beta}$.  This leads to an analogous result that follows from the proof above.

\begin{cor} \label{noncol_fund_to_imm}
For a composition $\alpha \models n$, the fundamental quasisymmetric functions expand into the dual immaculate functions as $$F_{\alpha} = \sum_{\beta} L^{-1}_{\alpha, \beta} \mathfrak{S}^*_{\beta} \text{\qquad with coefficients \qquad} L^{-1}_{\alpha, \beta} = \sum_{\mathcal{P}} (-1)^{\ell(\mathcal{P})}prod(\mathcal{P}),$$ where the sums runs over all $\beta$ reachable from $\alpha$ in $\mathfrak{D}^n$ and over paths $\mathcal{P}$ going from $\alpha$ to $\beta$ in $\mathfrak{D}^n$.
\end{cor}

%%%%%%%%%%%%%%%%%%%%%%%%%%%%%%%%%%%%%%%%%%%%%%%%%%%%%%%%%%%%%%%%%%%%%%%%%%%%%%%%%%%%%%%%%
\section{A colored generalization of the immaculate functions in $NSym_A$}\label{colorimmsection} 

A colored generalization of the immaculate basis can be defined by first introducing a colored version of noncommutative Bernstein creation operators. Various properties of these operators and extensions of our earlier results via duality lead to results on the colored immaculate functions.  These notably include a right Pieri rule and an expansion of the colored immaculate functions into the colored ribbon functions. As a corollary, we provide a new combinatorial model for the expansion of an immaculate function into the ribbon basis. It remains an open problem to find a cancellation-free expansion of the immaculate functions into the ribbon functions, but our formula does provide a straightforward and explicit way to compute the entries in the transition matrix between the immaculate and ribbon bases. Applying the forgetful map to our expression also yields a new expansion of Schur functions into the ribbon Schur functions.

The process for constructing our generalization of the noncommutative Bernstein operators mirrors that done in \cite{berg18} with some adjustments to account for the use of sentences in place of compositions. 

\begin{defn}\label{mperp1}
For $M \in QSym_A$, define the action of the linear \emph{perp operator} $M^{\perp}$ on $H\in NSym_A$ by $\langle M^{\perp} H, G \rangle = \langle H, MG \rangle$ for all $G \in QSym_A$. We define the action of the linear  \emph{right perp operator} $M^{\rperp}$ on $H \in NSym_A$ as $\langle M^{\rperp}H,G \rangle= \langle H, GM \rangle $ for all $G \in QSym_A$.  Thus, for dual bases $\{A_I\}_I$ of $QSym_A$ and $\{B_I\}_I$ of $NSym_A$, we have $$ M^{\perp}(H) = \sum_I \langle H, MA_I \rangle B_I \text{\qquad and \qquad} M^{\rperp}(H)= \sum_I \langle H, A_{I}M\rangle B_I.$$
\end{defn}

These operators are dual to the left and right multiplication by $M$ in $QSym_A$. Note that the analogues to these operators in $QSym$ are equivalent due to commutativity.

\begin{prop}\label{Mperp2} For sentences $I = (w_1,\ldots,w_k)$ and $J = (v_1,\ldots,v_h)$,  
$$M^{\rperp}_I(H_J) = \sum_{K} H_{\widetilde{J/_R K}},$$
where the sum runs over all sentences $K$ such that $\tilde{K} = I$ and $K \subseteq_R J$. Moreover, each $\widetilde{J/_R K}$ appearing in this sum is equivalent to the shape of a colored composition diagram originally of shape $J$ with boxes corresponding to each word in $I$ uniquely removed from its righthand side such that each word $w_j$ is removed from a single row strictly lower than the row from which $w_{j+1}$ is removed. 
\end{prop}

\begin{proof}
Let $I = (w_1, \ldots, w_k)$ and $J = (v_1, \ldots, v_h)$.  We have that $$M^{\rperp}_I(H_J) = \sum_L \langle H_J, M_LM_I \rangle H_L= \sum_L \langle H_J, \sum_R M_R \rangle H_L = \sum_L \sum_R \langle H_J, M_R \rangle H_L,$$
where the sums run over all sentences $L$ of size $|J|-|I|$ and each summand $R$ in $L \Qshuffle I$, respectively. Note that each sentence $R$ may occur multiple times in $L \Qshuffle I$ and we account for the multiplicity in the summations. The sum $\sum_R \langle H_J, M_R\rangle$ is equal to the number of times that $J$ appears as a summand in $L \Qshuffle I$.  Recall that in $L \Qshuffle I$, each summand is a sentence made up of words from $L$, words from $I$, and concatenated pairs of words from $L$ and $I$ (in that order) where all words from $L$ and all words from $I$ are present and in the same relative order, respectively.  For each time $J$ is a summand in $L \Qshuffle I$ there exists a unique weak sentence $K'$ such that $\tilde{K'}=I$ and $\widetilde{J/_R K'}=L$. Further, the set of all $K'$ obtained for $J$ in $L \Qshuffle I$ considered across every possible $L$ is simply the set of weak sentences $K$ such that $\tilde{K}=I$ and $K \subseteq_R J$, and so we can rewrite $$M^{\rperp}_I(H_J) = \sum_L \sum_{K'} H_L = \sum_{K} H_{\widetilde{J/_RK}},$$ where the sums run over all sentences $L$ of size $|J|-|I|$, all weak sentences $K'$ such that $\tilde{K'}=I$ and $\widetilde{J/_RK'}=L$, and all weak sentences $K$ such that $\tilde{K}= I$ and $K \subseteq_R J$, respectively.

Visualizing sentences as colored composition diagrams, we see that each weak sentence $K$ can be viewed as a unique set of boxes being removed from the right-hand side of the colored composition diagram of $J$ where the first word in $K$ (including empty words) is removed from the first row of $J$ and so on. Thus, the set of indices $\widetilde{J/_R K}$ of $H$ in the sum can also be viewed as the set of colored composition diagrams resulting from all possible ways of removing boxes corresponding to $I$ from a colored composition diagram of shape $J$ then moving rows up to fill empty rows, where each $w_j$ in $I$ is removed from a single row strictly lower than the single row from which $w_{j+1}$ in $I$ is removed.
\end{proof}

\begin{ex} In this example we show the action of $M^{\rperp}_{c,ab}$ on colored diagrams:
$$M^{\rperp}_{(c,ab)}(H_{(ac,bc,ab,cab)})=H_{(a,bc,cab)} + H_{(a,bc,ab,c)} + H_{(ac,b,cab)} + H_{(ac,b,ab,c)}.$$
\ytableausetup{boxsize=1.5em}
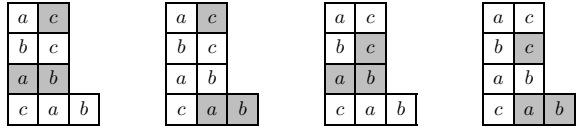
$$
\scalebox{.75}{
\begin{ytableau}
*(white) a & *(lightgray) c \\
*(white) b & *(white) c \\
*(lightgray) a & *(lightgray) b\\
*(white) c & *(white) a & *(white) b \\
\end{ytableau}
\quad \quad \quad
\begin{ytableau}
*(white) a & *(lightgray) c \\
*(white) b & *(white) c \\
*(white) a & *(white) b\\
*(white) c & *(lightgray) a & *(lightgray) b \\
\end{ytableau}
\quad \quad \quad
\begin{ytableau}
*(white) a & *(white) c \\
*(white) b & *(lightgray) c \\
*(lightgray) a & *(lightgray) b\\
*(white) c & *(white) a & *(white) b \\
\end{ytableau}
\quad \quad \quad
\begin{ytableau}
*(white) a & *(white) c \\
*(white) b & *(lightgray) c \\
*(white) a & *(white) b\\
*(white) c & *(lightgray) a & *(lightgray) b \\
\end{ytableau}}
$$
\end{ex}

Next, we prove various properties of the $M^{\rperp}$ operator that will be key in constructing creation operators for the colored immaculate basis.

\begin{lem}\label{tensorperp}
Let $J, K$ be sentences, $A_I \in QSym_A$ and $f, H \in NSym_A$. Then, $$\langle f \otimes H, \Delta(A_I)(M_J \otimes M_K) \rangle = \langle M^{\rperp}_J(f) \otimes M^{\rperp}_K(H), \Delta(A_I) \rangle.$$
\end{lem}

\begin{proof}
Let $a,b \in NSym_A$ and $c,d \in QSym_A$. The inner product on $NSym_A \times QSym_A$ extends to $NSym_A \otimes NSym_A \times QSym_A \otimes QSym_A$ as $$\langle \cdot , \cdot \rangle: NSym_A \otimes NSym_A \times QSym_A \otimes QSym_A \rightarrow \mathbb{Q} \text{\quad where \quad} \langle a \otimes b, c \otimes d \rangle \rightarrow \langle a, c \rangle \langle b, d \rangle$$
In Sweedler notation, $\Delta(A_I) = \sum_i A^{(i)}\otimes A_{(i)}$.  Thus, we write
    \begin{align*}
    \langle f \otimes H, \Delta(A_I)(M_J \otimes M_K) \rangle &= \left\langle f \otimes H, \sum_i A^{(i)}M_J \otimes A_{(i)}M_K \right\rangle = \sum_i \langle f \otimes H, A^{(i)}M_J \otimes A_{(i)}M_K \rangle\\
    &= \sum_i \langle f, A^{(i)}M_J \rangle \langle H, A_{(i)}M_K \rangle\\
    &= \sum_i \langle M^{\rperp}_J(f), A^{(i)} \rangle \langle M^{\rperp}_K(H), A_{(i)} \rangle \text{\quad by Definition \ref{mperp1}}\\
    &= \sum_i \langle M^{\rperp}_J(f) \otimes M^{\rperp}_K(H), A^{(i)} \otimes A_{(i)} \rangle = \left\langle M^{\rperp}_J(f) \otimes M^{\rperp}_K(H), \sum_i A^{(i)} \otimes A_{(i)} \right\rangle\\
    &= \langle M^{\rperp}_J(f) \otimes M^{\rperp}_K(H), \Delta(A_I) \rangle. \qedhere
    \end{align*}
\end{proof}

\begin{prop}\label{perpmult}
For a sentence $Q = (q_1,\ldots,q_i)$ and $f, H \in NSym_A$,
$$M_{Q}^{\rperp}(fH) = \sum_{0 \leq j \leq i} M^{\rperp}_{(q_1,\ldots,q_j)}(f)M^{\rperp}_{(q_{j+1},\ldots,q_i)}(H).$$  In particular, for a word $w$,

$$M_{Q}^{\rperp}(fH_w) = M^{\rperp}_{Q}(f)H_w + M^{\rperp}_{(q_1,\ldots,q_{i-1})}(f)M^{\rperp}_{q_i}(H_w).$$
\end{prop}

\begin{proof}
Let $\{A_I\}$ and $\{B_I\}$ be dual bases of $QSym_A$ and $NSym_A$ respectively, and let $Q = (q_1,\ldots,q_i)$. Then, 
\begin{align*}
    M_Q^{\rperp}(fH) &= \sum_I \langle fH, A_IM_Q \rangle B_I \text{\quad by Definition \ref{mperp1}} \\ 
    &= \sum_I \langle f \otimes H, \Delta(A_IM_Q)\rangle B_I
    = \sum_I \langle f \otimes H, \Delta(A_I)\Delta(M_Q)\rangle B_I \text{\quad by Definition \ref{hopf_dual_mult}}\\ &= \sum_I \sum_{Q=J \cdot K}\langle f \otimes H , \Delta(A_I)(M_J \otimes M_K)\rangle B_I \text{\quad by Equation \eqref{comult_m}}\\ 
    &= \sum_I \sum_{Q=J \cdot K} \langle M_J^{\rperp}(f)\otimes M_K^{\rperp}(H), \Delta(A_I) \rangle B_I \text{\quad by Lemma \ref{tensorperp}} \\ &= \sum_I \sum_{Q=J \cdot K} \langle M_J^{\rperp}(f)M_K^{\rperp}(H), A_I \rangle B_I \text{\quad by Definition \ref{hopf_dual_mult}}\\
     &= \sum_I  \langle \sum_{Q=J \cdot K} M_J^{\rperp}(f)M_K^{\rperp}(H), A_I \rangle B_I = \sum_{Q = J \cdot K} M_J^{\rperp}(f)M_K^{\rperp}(H) \text{\quad by Definition \ref{mperp1}}  \\ &= \sum_{j=0}^i M^{\rperp}_{(q_1,\ldots,q_j)}(f)M^{\rperp}_{(q_{j+1},\ldots,q_{i})}(H).
\end{align*}

In the case of $H=H_w$, the term $M^{\rperp}_{(q_{j+1}, \ldots, q_i)}(H_w)$ is 0 whenever $i-(j+1) > 0$ because boxes corresponding to $(q_{j}, \ldots, q_{i})$ must each be removed from separate rows but $w$ has only one row. Thus, the equation simplifies as

$$ M^{\rperp}_Q(fH_w) = M^{\rperp}_Q(f)H_w + M^{\rperp}_{(q_1, \ldots, q_{i-1})}(f)M^{\rperp}_{q_i}(H_w). \eqno \qedhere $$  \end{proof}

\begin{defn} \label{colorbern} For a word $v$, the \textit{colored noncommutative Bernstein operator} $\mathbb{B}_v$ is defined to be
$$\mathbb{B}_v = \sum_{u} \sum_{w(Q^r)=u}(-1)^i H_{v \cdot u} \left( \sum_{Q \preceq S}M^{ \rperp}_{S}\right),$$ where the sums run over all words $u$, all sentences $Q = (q_1,\ldots,q_i)$ such that $q_i \cdot \ldots \cdot q_1 = u$, and all sentences $S$ that are coarsenings of $Q$. 
\end{defn}

Notice that, by the definition of $M^{\rperp}$, the only values of $u$ that could yield a nonzero summand in $\mathbb{B}_v(H_I)$ for a sentence $I$ are those for which there is some permutation of the letters in $u$ that yields a subword of $w(I)$. Thus, this sum always has a finite number of terms.

\begin{defn}\label{color_imm_def} For a sentence $J = (v_1,\ldots,v_h)$, we define the \textit{colored immaculate function} $\mathfrak{S}_J$ as $$\mathfrak{S}_J = \mathbb{B}_{v_1}\mathbb{B}_{v_2}\ldots\mathbb{B}_{v_h}(1).$$
\end{defn}

\begin{ex} The colored immaculate functions $\mathfrak{S}_{(def)}$ and $\mathfrak{S}_{(abc,def)}$ are obtained using creation operators as follows: \vspace{-2mm}
\begin{align*}
\mathfrak{S}_{(def)} &= \mathbb{B}_{def}(1) = \sum_u \sum_{w(Q^r)=u} (-1)^{i}H_{(def \cdot u)}\left(\sum_{Q \preceq S}M^{\rperp}_S(1)\right)= (-1)^0 H_{(def)}M^{\rperp}_{\emptyset}(1)= H_{(def)}.\\ 
\mathfrak{S}_{(abc,def)} &= \mathbb{B}_{abc}(\mathfrak{S}_{(def)}) =  \mathbb{B}_{abc}(H_{(def)}) = \sum_u \sum_{w(Q^r)=u} (-1)^i H_{(abc \cdot u)} \left(\sum_{Q \preceq S} M^{\rperp}_S (H_{(def)})\right)\\ 
&= (-1)^0 H_{(abc)}M^{\rperp}_{\emptyset}(H_{(def)}) + (-1)^1 H_{(abcf)}M^{\rperp}_{(f)}(H_{(def)}) + (-1)^1 H_{(abcef)}M^{\rperp}_{(ef)}(H_{(def)})\\ & \quad + (-1)^2 H_{(abcfe)}M^{\rperp}_{(ef)}(H_{(def)}) 
 + (-1)^1 H_{(abcdef)}M^{\rperp}_{(def)}(H_{(def)}) + (-1)^2H_{(abcefd)}M^{\rperp}_{(def)}(H_{(def)})\\ & \quad +(-1)^2H_{(abcfde)}M^{\rperp}_{(def)}(H_{(def)}) + (-1)^3H_{(abcfed)}M^{\rperp}_{(def)}(H_{(def)})\\[1mm]
&= H_{(abc,def)} - H_{(abcf,de)}  - H_{(abcef,d)} + H_{(abcfe,d)} - H_{(abcdef)} + H_{(abcefd)} + H_{(abcfde)} - H_{(abcfed)}.
\end{align*}\vspace{-4mm}

To get the term $H_{(abcfe,d)}$, for example, we look at $u=fe$.  The possible values of $Q$ for this $u$ are $Q = (fe)$ and $Q = (e,f)$, meaning the possible $S$ values are $S= (fe)$, $S=(e,f)$, and $S=(ef)$.  Observe that $M^{\rperp}_{(fe)}(H_{(def)})$ and $M^{\rperp}_{(e,f)}(H_{(def)})$ are both zero because $S$ is not right-contained in $def$. Thus, the only remaining term for these values is $S=(ef)$ for which $M^{\rperp}_{(ef)}(H_{(def)})=H_{(d)}$.  Thus the term of the sum given by $u=fe$, $Q=(e,f)$, and $S=ef$ is $(-1)^2H_{(abcfe,d)}$, which is also the only term for $u=fe$.  Many values of $u$ will yield entirely zero terms.
\end{ex}

Before proving that this basis is indeed analogous to the immaculate functions in $NSym$, we must prove that it is dual to the colored dual immaculate basis.  The following property of the colored noncommutative Bernstein operators leads to a right Pieri rule which illuminates the structure of the colored immaculate functions to this end.

\begin{prop}\label{bernop} Let $w = a_1\ldots a_k$ and $f,H \in NSym_A$, then
$$\mathbb{B}_v(f)H_w = \sum_{0\leq j \leq k}\mathbb{B}_{(v \cdot (a_{j+1}\cdots a_{k}))}(f H_{(a_1\cdots a_j)}).$$
\end{prop}
\begin{proof}
Given a sentence $Q=(q_1, \ldots, q_i),$ we write $Q' = (q_1, \ldots, q_{i-1})$. Let $f \in NSym_A$ and let $v$ and $w=a_1 \ldots a_k$ be words. Then,
\begin{align*}
\mathbb{B}_{v}(fH_{w}) &= \sum_u \sum_{w(Q^r)=u} (-1)^i H_{v \cdot u}\left(\sum_{Q \preceq S}M^{\rperp}_{S}(fH_w)\right) \text{ by Definition \ref{colorbern}} \\
&= \sum_u \sum_{w(Q^r)=u} (-1)^i H_{v \cdot u}\left(\sum_{Q \preceq S} \left[ M^{\rperp}_{S}(f)H_w + M^{\rperp}_{S'}(f)M^{\rperp}_{s_t}(H_w)\right] \right) \text{by Proposition \ref{perpmult}}\\
&= \mathbb{B}_v(f)H_w + \sum_u \sum_{w(Q^r)=u} (-1)^i H_{v \cdot u}\left(\sum_{Q \preceq S} M^{\rperp}_{S'}(f)M^{\rperp}_{s_t}(H_w) \right) \text{by factoring.} \\ 
\intertext{ We want to consider the cases in which $M^{\rperp}_{s_t}(H_w)$ is non-zero.  This only happens whenever $s_t \subseteq_R w$ because in our combinatorial interpretation, we visualize $M^{\rperp}_{s_t}$ as removing $s_t$ from the righthand side of $w = a_1 \cdots a_k$ to get $H_{(a_1 \ldots a_h)}$ for some $h \leq k$. Note that because $Q \preceq S$ and $q_i$ and $s_t$ are the final words in $Q$ and $S$ respectively, $q_i \subseteq_R s_t$.  It follows that $q_i \subseteq_R w$ and thus $q_i = a_{j+1}\cdots a_k$ for a non-negative integer $j < k$. Recalling that $u = q_{i} \cdots q_1$, let $u' = q_{i-1} \cdots q_1$ so that we can write $u = a_{j+1} \cdots a_k \cdot u'$.  Rewriting the last equation in terms of $u'$ and $Q'$ yields}
\mathbb{B}_{v}(fH_{w})&= \mathbb{B}_v(f)H_w + \sum_{0 \leq j < k}\sum_{u'} \sum_{Q'} (-1)^i H_{(v \cdot (a_{j+1}\cdots a_k) \cdot u')}\left(\sum_{(Q' \cdot (a_{j+1}\cdots a_k)) \preceq S} M^{\rperp}_{S'}(f)M^{\rperp}_{s_t}(H_{(a_1\cdots a_j)}) \right).\\ 
\intertext{Next, the sum can be split into two parts by separating out the cases where $q_i = s_t$ and those where $q_i \not= s_t$. 
If $q_i=s_t$ for $q_i = a_{j+1} \cdots a_k$ then $M^{\rperp}_{s_t}(H_w) = M^{\rperp}_{(a_{j+1}\cdots a_{k})}(H_w) = H_{(a_1 \cdots a_j)}$.  Otherwise, there must exist a non-negative integer $\iota < i-1$ such that $s_t = q_{\iota +1}  \cdots q_{i-1}q_i$. We can rearrange the part of the sum by substituting $s_t$ with $q_{\iota+1}\cdots q_i$ and summing over the possible $\iota$. Then,} 
\mathbb{B}_{v}(fH_{w}) &= \mathbb{B}_v(f)H_w - \sum_{0 \leq j < k}\sum_{u'} \sum_{Q'} \left( (-1)^{i-1}H_{(v(a_{j+1}\cdots a_k) u')}  \Bigg[\sum_{Q' \preceq S'} M^{\rperp}_{S'}(f)H_{(a_1\cdots a_j)}\Bigg]\right. \\
& \left. \quad \quad + \Bigg[\sum_{0 \leq \iota < i-1} \sum _{(q_1,\ldots,q_{\iota}) \preceq S'} M^{\rperp}_{S'}(f)M^{\rperp}_{(q_{\iota+1}\cdot\ \cdots\ \cdot q_{i-1} \cdot (a_{j+1}\cdots a_k))}(H_{w})\Bigg] \right). \\ \intertext{Again thinking of the combinatorial visualization for the $M^{\rperp}$ operator, observe that 
$M^{\rperp}_{(q_{\iota +1} \cdots q_{i-1} (a_{j+1} \cdots a_k))} (H_w) = M^{\rperp}_{(q_{\iota +1} \cdots q_{i-1})} (H_{(a_1 \cdots a_j)})$ and so}
\mathbb{B}_{v }(fH_{w})&= \mathbb{B}_v(f)H_w - \sum_{0 \leq j < k}\sum_{u'} \sum_{Q'} \left( (-1)^{i-1}H_{(v (a_{j+1}\ldots a_k) u')}  \Bigg[\sum_{Q' \preceq S'} M^{\rperp}_{S'}(f)H_{(a_1\ldots a_j)}\Bigg] \right.\\
& \left. \quad \quad + \Bigg[\sum_{0 \leq \iota < i-1} \sum _{(q_1,\ldots,q_{\iota}) \preceq S'} M^{\rperp}_{S'}(f)M^{\rperp}_{(q_{\iota+1}\cdots q_{i-1})}(H_{(a_1\cdots a_j)})\Bigg]\right). \\ 
\intertext{Next, rename every $S'$ to $R = (r_1,\ldots ,r_{\tau})$ in the first section of the sum.  In the second section, rename $S'$ to $R' =(r_1, \ldots , r_{\tau - 1})$ and let $q_{\iota +1}\cdots q_{i-1} = r_{\tau}$.}
\mathbb{B}_{v }(fH_{w})&= \mathbb{B}_v(f)H_w - \sum_{0 \leq j < k}\sum_{u'} \sum_{Q'} \left( (-1)^{i-1}H_{(v (a_{j+1}\ldots a_k) u')}  \Bigg[\sum_{Q' \preceq R} M^{\rperp}_{R}(f)H_{(a_1\ldots a_j)}\Bigg]\right.\\
& \left. \quad \quad + \Bigg[\sum_{0 \leq \iota < i-1} \sum _{(q_1,\ldots,q_{\iota}) \preceq R'} M^{\rperp}_{R'}(f)M^{\rperp}_{r_{\tau}}(H_{(a_1\cdots a_j)})\Bigg] \right).\\ 
\intertext{ In the second part of the sum, notice that considering $R' \cdot r_{\tau}$ where $R' = (q_1, \ldots, q_{\iota})$ and $r_{\tau} = q_{\iota + 1} \cdots q_{i-1}$ for $1 \leq \iota \leq i-1$ is equivalent to considering $R' \cdot r_{\tau} = R \succeq (q_1, \ldots , q_{i-1})=Q'$. Then,}
\mathbb{B}_{v }(fH_{w})&= \mathbb{B}_v(f)H_w - \sum_{0 \leq j < k}\sum_{u'} \sum_{Q'} \left( (-1)^{i-1}H_{(v (a_{j+1}\ldots a_k) u')}  \Bigg[\sum_{Q' \preceq R} M^{\rperp}_{R}(f)H_{(a_1\ldots a_j)}\Bigg] \right. \\
& \left. \quad \quad + \Bigg[\sum _{Q' \preceq R} M^{\rperp}_{R'}(f)M^{\rperp}_{r_{\tau}}(H_{(a_1\cdots a_j)})\Bigg] \right). \\ \intertext{Now in both parts of the sum, we are looking at sentences $R$ such that $Q' \preceq R$, and combining them we get}
\mathbb{B}_{v}(fH_{w})&= \mathbb{B}_v(f)H_w - \sum_{0 \leq j < k}\sum_{u'} \sum_{Q'} (-1)^{i-1}H_{(v (a_{j+1}\ldots a_k) u')} \Bigg( \sum_{Q' \preceq R} \Bigg[ M^{\rperp}_{R}(f)H_{(a_1 \cdots a_j)} + M^{\rperp}_{R'}(f)M^{\rperp}_{r_{\tau}}(H_{(a_1 \cdots a_j)})\Bigg] \Bigg) \\
&= \mathbb{B}_v(f)H_w - \sum_{0 \leq j < k}\sum_{u'} \sum_{Q'} (-1)^{i-1}H_{(v(a_{j+1}\ldots a_k) u')}\left( \sum_{Q' \preceq R} M^{\rperp}_{R}(fH_{(a_1\ldots a_j)}) \right) \text{by Proposition \ref{perpmult}}\\
&= \mathbb{B}_v(f)H_w - \sum_{0 \leq j < k}\mathbb{B}_{(v \cdot (a_{j+1}\cdots a_k)}(fH_{(a_1\cdots a_j)}) \text{ by Definition \ref{colorbern}. } \qedhere 
\end{align*}
\end{proof}

\begin{thm}[Right Pieri Rule]\label{rightpieriI} For the sentence $J = (v_1,\ldots,v_h)$ and the word $w =a_1\ldots a_i$, $$\mathfrak{S}_J H_w = \sum_{J \subset_{w} K}\mathfrak{S}_K,$$ where $J \subset_{w} K = (u_1, \ldots, u_g)$ if $u_j = v_j \cdot q_j$  for $1 \leq j \leq g$ such that $q_g \cdot q_{g-1} \cdot \ldots\cdot q_2 \cdot q_1 = w$ and $g \leq h+1$ where $v_{h+1}= \emptyset$.
\end{thm}

For a sentence $I$ and word $w$, the product $\mathfrak{S}_IH_w$ given by the Pieri rule can be visualized in terms of the indices of colored immaculate functions in the resulting sum. The indices correspond to all diagrams obtained by adding colored boxes below or to the right of the diagram of $I$, such that when reading the colors of boxes left to right from bottom to top they correspond exactly to $w$.

\ytableausetup{boxsize=1.5em} 
\begin{ex} The product below can be visualized using the following tableaux:

$$\scalebox{.75}{
\begin{ytableau}
    *(white) a & *(white) b \\
    *(white) b & *(white) c \\
    *(gray!20) c & *(gray!20) a
\end{ytableau}
\quad \quad
\begin{ytableau}
    *(white) a & *(white) b \\
    *(white) b & *(white) c & *(gray!20) a \\
    *(gray!20) c 
\end{ytableau}
\quad \quad
\begin{ytableau}
    *(white) a & *(white) b & *(gray!20) a\\
    *(white) b & *(white) c \\
    *(gray!20) c
\end{ytableau}
\quad \quad
\begin{ytableau}
    *(white) a & *(white) b \\
    *(white) b & *(white) c & *(gray!20) c & *(gray!20) a
\end{ytableau}
\quad \quad
\begin{ytableau}
    *(white) a & *(white) b & *(gray!20) a\\
    *(white) b & *(white) c & *(gray!20) c
\end{ytableau}
\quad \quad
\begin{ytableau}
    *(white) a & *(white) b & *(gray!20)c & *(gray!20) a\\
    *(white) b & *(white) c
\end{ytableau}}
$$\vspace{0mm}
    $$\mathfrak{S}_{(ab,bc)}H_{(ca)} = \mathfrak{S}_{(ab,bc,ca)} +\mathfrak{S}_{(ab,bca,c)} + \mathfrak{S}_{(aba,bc,c)} + \mathfrak{S}_{(ab,bcca)} +\mathfrak{S}_{(aba,bcc)} + \mathfrak{S}_{(abca,bc)}.$$
\end{ex}

\begin{proof}[Proof of Theorem \ref{rightpieriI}]
We proceed by induction on $|w| + \ell{(J)}$. There are two base cases where $|w|+\ell(J) =1$. 
\begin{enumerate}
    \item If $|w|=1$ and $\ell(J)=0$, then $\mathfrak{S}_{\emptyset}H_{w}= \sum_{1 \subset_w K} \mathfrak{S}_K = \mathfrak{S}_{w} = H_w$.
    \item If $\ell(J)=1$ and $|w|=0$ then $\mathfrak{S}_{J}H_{\emptyset} = \sum_{J \subset_{\emptyset} K}\mathfrak{S}_K = \mathfrak{S}_J$. 
\end{enumerate}
Next, assume the statement is true when $|w| + \ell(J) \leq k$ and let $\bar{J}=(v_2,\ldots,v_h)$. Let $|w| + \ell(J)=k+1$. 
\begin{align*}
\mathfrak{S}_JH_w &= \mathbb{B}_{v_1}(\mathfrak{S}_{\bar{J}})H_w = \sum_{0\leq j < i} \mathbb{B}_{(v_1 \cdot (a_{j+1} \cdots a_i))}\left(\mathfrak{S}_{\bar{J}}H_{a_1\ldots a_j}\right) \text{ by Definition \ref{color_imm_def} and Proposition \ref{bernop},}\\
&= \sum_{0\leq j < i} \mathbb{B}_{(v_1 \cdot (a_{j+1} \cdots a_{i}))}\left(\sum_{\bar{J}\subset_{a_1\ldots a_j}G}\mathfrak{S}_G\right)\text{\quad \quad by induction,}\\
&= \sum_{0\leq j < i} \left( \sum_{\bar{J}\subset_{a_1\ldots a_j}G} \mathfrak{S}_{((v_1 (a_{j+1}\cdots a_{i})) \cdot G)} \right) = \sum_{J\subset_w K} \mathfrak{S}_K. \qedhere
\end{align*}\vspace{-2mm}
\end{proof}

The expansion of the colored complete homogeneous functions into the colored immaculate functions follows from repeated application of the right Pieri rule.

\begin{thm} \label{H_to_S} For a sentence $C$, the colored complete homogeneous function expands positively into the colored immaculate basis as 
$$H_C = \sum_J K_{J,C}\mathfrak{S}_J,$$ where the sum runs over sentences $J$ such that $|J|=|C|$.
\end{thm}

\begin{proof}
Let $C= (t_1,\ldots,t_k)$ and $C'=(t_1,\ldots,t_{k-1})$. First we claim that $K_{J,C} = \sum_{G \subset_{t_k}J}K_{G,C'}$ where the sum runs over sentences $G$ such that $G \subset_{t_k} J$.  For any colored immaculate tableau of shape $J$ and type $C$, we can remove the boxes of $T$ filled with the number k, all of which will be on the right-hand side of $T$, to obtain a colored immaculate tableau of shape $G$ with type $C'$. Thus the sum of $K_{G,C'}$ for all the $G \subset_{t_k} J$ gives $K_{J,C}$. With this fact, we proceed by induction on the length of $C$.
\begin{align*}
H_C &= H_{C'}H_{t_k} = \left( \sum_G K_{G,C'}\mathfrak{S}_G \right)H_{t_k} \text{\quad by induction}\\
&= \sum_G K_{G,C'}\mathfrak{S}_GH_{t_k} = \sum_G K_{G,C'}\sum_{G \subset_{t_k}J}\mathfrak{S}_J \text{ \quad by Theorem \ref{rightpieriI}}\\
&= \sum_J \left( \sum_{G \subset_{t_k}J}K_{G,C'}\right) \mathfrak{S}_J \text{\quad by rearranging the sums} \\
&= \sum_J K_{J,C}\mathfrak{S}_J,
\end{align*}
where the final two sums run over all sentences $J$ such that there exists a colored immaculate tableau of shape $J$ and type $C$. If there is no such CIT of shape $J$ and type $C$ then $K_{J,C}=0$, and it is equivalent to taking this sum over all sentences $J$ such that $|J|=|C|$.
\end{proof}

Note that this unique expansion satisfies Proposition \ref{hopf_dual} and in fact verifies the duality of the colored immaculate and colored dual immaculate bases.

\begin{cor}\label{color_dual} The colored immaculate basis is dual to the colored dual immaculate basis.  
\end{cor} 

% \begin{proof} Recall that any $G\in QSym_A$ can be expanded into the colored monomial basis  as $G = \sum_I \langle H_I, G \rangle M_I$. In particular, for $G = \mathfrak{S}^*_K$, 
% \begin{align*}
%     \mathfrak{S}^*_K &= \sum_I \langle H_I, \mathfrak{S}^*_K \rangle M_{I} = \sum_I \langle \sum_J K_{J,I} \mathfrak{S}_J, \mathfrak{S}^*_K \rangle M_I \text{\quad by Theorem \ref{H_to_S}.}\\
%     &= \sum_I \sum_J K_{J,I} \langle \mathfrak{S}_J, \mathfrak{S}^*_K \rangle M_I = \sum_J \langle \mathfrak{S}_J, \mathfrak{S}^*_K \rangle \sum_I K_{J,I} M_I\\
%     &= \sum_J \langle \mathfrak{S}_J, \mathfrak{S}^*_K \rangle \mathfrak{S}^*_J. \text{\quad by Theorem \ref{monomialexpansion}.}
% \end{align*}
% Therefore, $\langle \mathfrak{S}_J, \mathfrak{S}^*_K \rangle = \delta_{J,K}$.
% \end{proof}

With this duality verified, we can prove that the colored immaculate functions are analogous to the original noncommutative Bernstein operators because they are isomorphic under $\upsilon $ in the case of a unary alphabet $A$.  

\begin{prop} \label{uncoloring_dual} Let $G \in NSym_A$ and $F \in QSym_A$. If $A = \{a\}$, then $\langle G, F\rangle = \langle \upsilon(G), \upsilon(F) \rangle.$
\end{prop}

\begin{proof} Let $A = \{a\}$, and let $G = \sum_J c_J H_J$ and $F = \sum_I b_I M_I$ where the sums run over all sentences $I, J$, respectively.  Then, $$\langle G, F \rangle =  \left\langle \sum_J c_J H_J,  \sum_I b_I M_I \right\rangle = \sum_{I,J}c_Jb_I \left\langle H_J, M_I \right\rangle = \sum_I c_Ib_I.$$
Next, for $\upsilon (G) \in NSym$ and $\upsilon (F) \in QSym$, we have that $$\langle \upsilon (G), \upsilon (F) \rangle = \left\langle \sum_J c_J \upsilon (H_J), \sum_I b_I \upsilon (M_I) \right\rangle = \left\langle \sum_J c_J H_{w\ell(J)}, \sum_I b_I M_{w\ell(I)} \right\rangle = \sum_{I,J} c_Jb_I \langle H_{w\ell(J)}, M_{w\ell(I)}\rangle.$$
The inner product $\langle H_{w\ell(J)}, M_{w\ell(I)}\rangle$ is zero unless $w\ell(I)=w\ell(J)$ which happens exactly when $I = J$ because the alphabet $A$ is made up of only one color.  In other words, there is exactly one sentence $I$ such that $w\ell(I)=\alpha$ for each composition $\alpha$ in this case.  Thus,  $$\langle \upsilon (G), \upsilon (F) \rangle = \sum_I c_Ib_I = \langle G, F \rangle. \eqno \qedhere $$ \end{proof}

\begin{prop}\label{color_imm_analog} Let $A = \{a\}$, and let $I$ be a sentence. Then, $\upsilon (\mathfrak{S}_I) = \mathfrak{S}_{w\ell(I)}.$ Moreover, $\{\mathfrak{S}_I\}_I$ in $NSym_A$ is analogous to 
$\{\mathfrak{S}_{\alpha}\}_{\alpha}$ in $NSym$. \end{prop}

\begin{proof} Let $A = \{a\}$ and let $I$ and $J$ be sentences.  By Proposition \ref{uncoloring_dual},  $$\langle \mathfrak{S}_I, \mathfrak{S}^*_J \rangle = \langle \upsilon (\mathfrak{S}_I), \upsilon (\mathfrak{S}^*_J) \rangle = \langle \upsilon (\mathfrak{S}_I), \mathfrak{S}^*_{w\ell(J)} \rangle.$$ Because $A$ is unary, $I= J$ if and only if $w\ell(I)=w\ell(J)$ and thus $\delta_{I,J} = \delta_{w\ell(I), w\ell(J)}$. As a result, $$\langle \mathfrak{S}_I, \mathfrak{S}^*_J \rangle  = \langle \mathfrak{S}_{w\ell(I)}, \mathfrak{S}^*_{w\ell(J)} \rangle = \langle \upsilon (\mathfrak{S}_I), \mathfrak{S}^*_{w\ell(J)} \rangle,$$ for all sentences $I$ and $J$. Therefore, $\upsilon (\mathfrak{S}_I) = \mathfrak{S}_{w\ell(I)}$.
\end{proof}

The expansion of the colored ribbon functions into the colored immaculate functions now follows from the application of Proposition \ref{hopf_dual} to Theorem \ref{fund_exp}.

\begin{cor} \label{color_r_into_imm}
For a sentence $C$, the colored ribbon noncommutative symmetric functions expand positively into the colored immaculate functions as
$$R_C = \sum_J L_{J,C}\mathfrak{S}_J,$$ where the sum runs over all sentences $J$ such that $|J|=|C|$.
\end{cor}

This corollary allows us to define the expansion of the colored immaculate function indexed by a sentence of the form $(a_1, \ldots, a_k)$ in terms of the $\{H_I \}$ basis.

\begin{prop} For a sentence $(a_1, \ldots a_k)$ where $a_1, \ldots, a_k \in A$ are colors, $$\mathfrak{S}_{(a_1, \ldots, a_k)} = \sum_{(a_1, \ldots, a_k) \preceq J} (-1)^{k-\ell(J)}H_J.$$  
\end{prop}

\begin{proof} 
Let $C = (a_1, \ldots, a_k)$, and notice that $L_{J, (a_1, \ldots, a_k)}=0$ unless $J=(a_1, \ldots, a_k)$ in which case $L_{(a_1, \ldots, a_k), (a_1, \ldots, a_k)} = 1$. Then by Corollary \ref{color_r_into_imm}, we have $\mathfrak{S}_{(a_1, \ldots, a_k)} = R_{(a_1, \ldots, a_k)}$. Then, expanding $R_{(a_1, \ldots, a_k)}$ into the $\{ H_I \}_I$ basis yields the desired formula.
\end{proof}

Applying Proposition \ref{hopf_dual} to Theorem \ref{color_fund_to_imm} also yields an expansion of the colored immaculate functions into the colored ribbon basis using the colored immaculate descent graph of Definition \ref{color_imm_desc_graph}.

\begin{cor} \label{color_imm_to_R} For a sentence $J$, the colored immaculate functions expand into colored ribbon functions as $$\mathfrak{S}_J = \sum_I L^{-1}_{I,J}R_I \text{\qquad with coefficients \qquad} L^{-1}_{I,J} = \sum_{\mathcal{P}}(-1)^{\ell(\mathcal{P})}prod(\mathcal{P}),$$ where the  first sum runs over sentences $I$ and the second runs over directed paths $\mathcal{P}$ from $I$ to $J$ in $\mathfrak{D}^n_A$.
\end{cor}

\begin{ex} The colored immaculate function $\mathfrak{S}_{(a,cb,b)}$ expands in to the colored ribbon functions as $$\mathfrak{S}_{(a,cb,b)} = R_{(a,cb,b)}-R_{(ab,cb)}+R_{(abb,c)}-R_{(ab,c,b)}.$$

The term $R_{(abb,c)}$, for example,  has a coefficient of $1$ because the only path from $(abb,c)$ to $(a,cb,b)$ is $$(abb,c) \xrightarrow{1} (ab,cb) \xrightarrow{1} (a,cb,b).$$
\end{ex}

Proposition \ref{hopf_dual} can be applied to Corollary \ref{noncol_fund_to_imm} to get a result in $NSym$ analogous to Corollary \ref{color_imm_to_R}. It is actually an open question to find a cancellation-free combinatorial way of expanding immaculate functions into the ribbon basis. Campbell defines formulas for a few special cases in \cite{Camp}. In \cite{AllMas}, Allen and Mason give a complete combinatorial description of the expansion of immaculate functions into the complete homogeneous basis in terms of tunnel hooks, which generalize the special rim hooks of E\u{g}ecio\u{g}lu and Remmel \cite{EgeRem}. This becomes a somewhat complicated expansion of any immaculate function into the ribbon basis, but for certain immaculate functions, the expression simplifies to a Jacobi-Trudi-Like formula. While our formula is not cancellation-free, it does provide a concise way to compute the coefficients in the expansion for every case. Additionally, it is relatively easy to compute a single coefficient without calculating the entire expression or the entire transition matrix.

\begin{cor} \label{imm_to_R} For a composition $\beta \models n$, the immaculate functions expand into the ribbon functions as $$\mathfrak{S}_{\beta} = 
\sum_{\alpha \models n} L^{-1}_{\alpha, \beta} R_{\alpha} \text{\qquad with coefficients \qquad} L^{-1}_{\alpha, \beta} = \sum_{\mathcal{P}}(-1)^{\ell(\mathcal{P})}prod(\mathcal{P}),$$ where the first sum runs over compositions $\alpha$ and the second runs over directed paths $\mathcal{P}$ from $\alpha$ to $\beta$ in $\mathfrak{D}^n$.
\end{cor}

Applying the forgetful map $\chi$ to Corollary \ref{imm_to_R} produces a new expansion of the Schur functions into ribbon Schur functions. The question of expressing Schur functions in terms of ribbon Schur functions was notably studied by Lascoux and Pragacz in \cite{LasPra} as well as Hamel and Goulden in \cite{HamGou}. One advantage of this expression compared to the matrix determinant expressions of Lascoux and Pragacz or Hamel and Goulden is that we can compute single coefficients without computing the entire expansion.
\begin{cor} For a partition $\lambda \vdash n$, a Schur function can be decomposed into ribbon Schur functions as $$s_{\lambda} = \sum_{\alpha \models n} L^{-1}_{\alpha, \lambda} r_{\alpha} \text{\quad with \quad} L^{-1}_{\alpha, \lambda} = \sum_{\mathcal{P}}(-1)^{\ell(\mathcal{P})}prod(\mathcal{P}), $$ where the first sum runs over compositions $\alpha$ and the second runs over directed paths $\mathcal{P}$ from $\alpha$ to $\lambda$ in $\mathfrak{D}^n$. 
\end{cor}

While the colored immaculate functions mirror many of the properties of the immaculate functions, the Jacobi-Trudi formula does not generalize naturally. This is in part due to the challenges of a deletion operation on words which would be needed to generalize integer subtraction.  Future work may investigate such a formula.

\section{The colored immaculate poset and 
 skew colored immaculate tableaux}  \label{posetsection}

Colored composition diagrams admit a natural partial ordering similar to that of Young's lattice and the immaculate poset. The elements of this poset can be thought of as sentences or colored composition diagrams, which gives a more visual representation. This poset has a combinatorial relationship with standard colored immaculate tableaux and leads to a natural definition of skew colored immaculate tableaux which in turn leads to the skew colored dual immaculate functions. Additionally, the right Pieri rule on colored immaculate functions connects this poset and these skew functions to the structure constants of the colored immaculate functions as it does in the non-colored case.

\begin{defn}
The \emph{colored immaculate poset $\mathfrak{P}_A$} is the set of all sentences on $A$ with the partial order defined by the cover relation that $I$ covers $J$ if $J \subset_{a} I$ for some $a \in A$. 
\end{defn}

This cover relation means that $I$ covers $J$ if $I$ differs from $J$ by the addition of a box colored with $a$ placed on the right side of, or below, $J$. In this case, arrows from $J$ to $I$ in the Hasse diagram of $\mathfrak{P}_A$ are labeled with $(m,a)$ where $m$ is the number of the row to which $a$ is added in $J$. 

\tikzset{every picture/.style={line width=0.75pt}} %set default line width to 0.75pt      

\begin{figure}[h!]
\centering
\scalebox{0.65}{
\begin{tikzpicture}[x=0.75pt,y=0.75pt,yscale=-1,xscale=1]

%uncomment if require: \path (0,556); %set diagram left start at 0, and has height of 556

%Shape: Rectangle [id:dp575965922614204] 
\draw  [line width=.5]  (164,242) -- (194,242) -- (194,272) -- (164,272) -- cycle ;
%Shape: Rectangle [id:dp8044898863012624] 
\draw  [line width=.5]  (277,403) -- (307,403) -- (307,433) -- (277,433) -- cycle ;
%Shape: Rectangle [id:dp7905547104716522] 
\draw  [line width=.5]  (397,404) -- (427,404) -- (427,434) -- (397,434) -- cycle ;
%Shape: Rectangle [id:dp8647261706177158] 
\draw  [line width=.5]  (77,225) -- (107,225) -- (107,255) -- (77,255) -- cycle ;
%Shape: Rectangle [id:dp8596816641566041] 
\draw  [line width=.5]  (77,255) -- (107,255) -- (107,285) -- (77,285) -- cycle ;
%Shape: Rectangle [id:dp20927213114675958] 
\draw  [line width=.5]  (220,227) -- (250,227) -- (250,257) -- (220,257) -- cycle ;
%Shape: Rectangle [id:dp6586471682862116] 
\draw  [line width=.5]  (220,257) -- (250,257) -- (250,287) -- (220,287) -- cycle ;
%Shape: Rectangle [id:dp4688322013783064] 
\draw  [line width=.5]  (366,228) -- (396,228) -- (396,258) -- (366,258) -- cycle ;
%Shape: Rectangle [id:dp9684081264459412] 
\draw  [line width=.5]  (366,258) -- (396,258) -- (396,288) -- (366,288) -- cycle ;
%Shape: Rectangle [id:dp21890790077669253] 
\draw  [line width=.5]  (418,228) -- (448,228) -- (448,258) -- (418,258) -- cycle ;
%Shape: Rectangle [id:dp2759477139217428] 
\draw  [line width=.5]  (418,258) -- (448,258) -- (448,288) -- (418,288) -- cycle ;
%Shape: Rectangle [id:dp4666793998158161] 
\draw  [line width=.5]  (134,242) -- (164,242) -- (164,272) -- (134,272) -- cycle ;
%Shape: Rectangle [id:dp9144394614950948] 
\draw  [line width=.5]  (306,242) -- (336,242) -- (336,272) -- (306,272) -- cycle ;
%Shape: Rectangle [id:dp179758949332528] 
\draw  [line width=.5]  (276,242) -- (306,242) -- (306,272) -- (276,272) -- cycle ;
%Shape: Rectangle [id:dp993119704417242] 
\draw  [line width=.5]  (29,34) -- (59,34) -- (59,64) -- (29,64) -- cycle ;
%Shape: Rectangle [id:dp20725453643577607] 
\draw  [line width=.5]  (29,64) -- (59,64) -- (59,94) -- (29,94) -- cycle ;
%Shape: Rectangle [id:dp8782089146601351] 
\draw  [line width=.5]  (87,51) -- (117,51) -- (117,81) -- (87,81) -- cycle ;
%Shape: Rectangle [id:dp42348593295052894] 
\draw  [line width=.5]  (87,81) -- (117,81) -- (117,111) -- (87,111) -- cycle ;
%Shape: Rectangle [id:dp29503917266779056] 
\draw  [line width=.5]  (506,66) -- (536,66) -- (536,96) -- (506,96) -- cycle ;
%Shape: Rectangle [id:dp8425995873889967] 
\draw  [line width=.5]  (476,66) -- (506,66) -- (506,96) -- (476,96) -- cycle ;
%Shape: Rectangle [id:dp2167212554018969] 
\draw  [line width=.5]  (29,94) -- (59,94) -- (59,124) -- (29,124) -- cycle ;
%Shape: Rectangle [id:dp4870845648312154] 
\draw  [line width=.5]  (117,81) -- (147,81) -- (147,111) -- (117,111) -- cycle ;
%Shape: Rectangle [id:dp8720272286366499] 
\draw  [line width=.5]  (536,66) -- (566,66) -- (566,96) -- (536,96) -- cycle ;
%Shape: Rectangle [id:dp8427842952010649] 
\draw  [line width=.5]  (177,34) -- (207,34) -- (207,64) -- (177,64) -- cycle ;
%Shape: Rectangle [id:dp6209357871208898] 
\draw  [line width=.5]  (177,64) -- (207,64) -- (207,94) -- (177,94) -- cycle ;
%Shape: Rectangle [id:dp627007474058193] 
\draw  [line width=.5]  (237,52) -- (267,52) -- (267,82) -- (237,82) -- cycle ;
%Shape: Rectangle [id:dp052225851993351835] 
\draw  [line width=.5]  (237,82) -- (267,82) -- (267,112) -- (237,112) -- cycle ;
%Shape: Rectangle [id:dp1772894266753322] 
\draw  [line width=.5]  (625,65) -- (655,65) -- (655,95) -- (625,95) -- cycle ;
%Shape: Rectangle [id:dp2187131753926035] 
\draw  [line width=.5]  (595,65) -- (625,65) -- (625,95) -- (595,95) -- cycle ;
%Shape: Rectangle [id:dp34667173126984285] 
\draw  [line width=.5]  (177,94) -- (207,94) -- (207,124) -- (177,124) -- cycle ;
%Shape: Rectangle [id:dp379658379008994] 
\draw  [line width=.5]  (267,52) -- (297,52) -- (297,82) -- (267,82) -- cycle ;
%Shape: Rectangle [id:dp07343214760304262] 
\draw  [line width=.5]  (655,65) -- (685,65) -- (685,95) -- (655,95) -- cycle ;
%Shape: Rectangle [id:dp39327869651417147] 
\draw  [line width=.5]  (328,33) -- (358,33) -- (358,63) -- (328,63) -- cycle ;
%Shape: Rectangle [id:dp9442314809207009] 
\draw  [line width=.5]  (328,63) -- (358,63) -- (358,93) -- (328,93) -- cycle ;
%Shape: Rectangle [id:dp12361211082599866] 
\draw  [line width=.5]  (389,52) -- (419,52) -- (419,82) -- (389,82) -- cycle ;
%Shape: Rectangle [id:dp168024631740892] 
\draw  [line width=.5]  (389,82) -- (419,82) -- (419,112) -- (389,112) -- cycle ;
%Shape: Rectangle [id:dp73913653649088] 
\draw  [line width=.5]  (328,93) -- (358,93) -- (358,123) -- (328,123) -- cycle ;
%Shape: Rectangle [id:dp31165221069914906] 
\draw  [line width=.5]  (419,52) -- (449,52) -- (449,82) -- (419,82) -- cycle ;
%Straight Lines [id:da013697030856572479] 
\draw [color={rgb, 255:red, 0; green, 0; blue, 0 }  ,draw opacity=1 ][line width=.5]    (343,487) -- (311.7,441.47) ;
\draw [shift={(310,439)}, rotate = 55.49] [color={rgb, 255:red, 0; green, 0; blue, 0 }  ,draw opacity=1 ][line width=.5]    (14.21,-4.28) .. controls (9.04,-1.82) and (4.3,-0.39) .. (0,0) .. controls (4.3,0.39) and (9.04,1.82) .. (14.21,4.28)   ;
%Straight Lines [id:da566827956828055] 
\draw [color={rgb, 255:red, 0; green, 0; blue, 0 }  ,draw opacity=1 ][line width=.5]    (360,487) -- (391.25,443.44) ;
\draw [shift={(393,441)}, rotate = 125.66] [color={rgb, 255:red, 0; green, 0; blue, 0 }  ,draw opacity=1 ][line width=.5]    (14.21,-4.28) .. controls (9.04,-1.82) and (4.3,-0.39) .. (0,0) .. controls (4.3,0.39) and (9.04,1.82) .. (14.21,4.28)   ;
%Straight Lines [id:da9970360480736631] 
\draw [color={rgb, 255:red, 0; green, 0; blue, 0 }  ,draw opacity=1 ][line width=.5]    (287,395) -- (149.43,295.76) ;
\draw [shift={(147,294)}, rotate = 35.81] [color={rgb, 255:red, 0; green, 0; blue, 0 }  ,draw opacity=1 ][line width=.5]    (14.21,-4.28) .. controls (9.04,-1.82) and (4.3,-0.39) .. (0,0) .. controls (4.3,0.39) and (9.04,1.82) .. (14.21,4.28)   ;
%Shape: Rectangle [id:dp6183191013664859] 
\draw  [line width=.5]  (505,239) -- (535,239) -- (535,269) -- (505,269) -- cycle ;
%Shape: Rectangle [id:dp9972582945775195] 
\draw  [line width=.5]  (475,239) -- (505,239) -- (505,269) -- (475,269) -- cycle ;
%Shape: Rectangle [id:dp6933384912425269] 
\draw  [line width=.5]  (588,239) -- (618,239) -- (618,269) -- (588,269) -- cycle ;
%Shape: Rectangle [id:dp30062854419368] 
\draw  [line width=.5]  (558,239) -- (588,239) -- (588,269) -- (558,269) -- cycle ;
%Straight Lines [id:da4236162517104267] 
\draw [color={rgb, 255:red, 0; green, 0; blue, 0 }  ,draw opacity=1 ][line width=.5]    (287,395) -- (94.66,295.38) ;
\draw [shift={(92,294)}, rotate = 27.38] [color={rgb, 255:red, 0; green, 0; blue, 0 }  ,draw opacity=1 ][line width=.5]    (14.21,-4.28) .. controls (9.04,-1.82) and (4.3,-0.39) .. (0,0) .. controls (4.3,0.39) and (9.04,1.82) .. (14.21,4.28)   ;
%Straight Lines [id:da22090322030915743] 
\draw [color={rgb, 255:red, 0; green, 0; blue, 0 }  ,draw opacity=1 ][line width=.5]    (287,395) -- (224.62,297.53) ;
\draw [shift={(223,295)}, rotate = 57.38] [color={rgb, 255:red, 0; green, 0; blue, 0 }  ,draw opacity=1 ][line width=.5]    (14.21,-4.28) .. controls (9.04,-1.82) and (4.3,-0.39) .. (0,0) .. controls (4.3,0.39) and (9.04,1.82) .. (14.21,4.28)   ;
%Straight Lines [id:da7785593062942968] 
\draw [color={rgb, 255:red, 0; green, 0; blue, 0 }  ,draw opacity=1 ][line width=.5]    (287,395) -- (305.93,297.94) ;
\draw [shift={(306.5,295)}, rotate = 101.03] [color={rgb, 255:red, 0; green, 0; blue, 0 }  ,draw opacity=1 ][line width=.5]    (14.21,-4.28) .. controls (9.04,-1.82) and (4.3,-0.39) .. (0,0) .. controls (4.3,0.39) and (9.04,1.82) .. (14.21,4.28)   ;
%Straight Lines [id:da4160419512010869] 
\draw [color={rgb, 255:red, 0; green, 0; blue, 0 }  ,draw opacity=1 ][line width=.5]    (412,395) -- (379.96,299.84) ;
\draw [shift={(379,297)}, rotate = 71.39] [color={rgb, 255:red, 0; green, 0; blue, 0 }  ,draw opacity=1 ][line width=.5]    (14.21,-4.28) .. controls (9.04,-1.82) and (4.3,-0.39) .. (0,0) .. controls (4.3,0.39) and (9.04,1.82) .. (14.21,4.28)   ;
%Straight Lines [id:da9430452435986705] 
\draw [color={rgb, 255:red, 0; green, 0; blue, 0 }  ,draw opacity=1 ][line width=.5]    (412,395) -- (431.41,297.94) ;
\draw [shift={(432,295)}, rotate = 101.31] [color={rgb, 255:red, 0; green, 0; blue, 0 }  ,draw opacity=1 ][line width=.5]    (14.21,-4.28) .. controls (9.04,-1.82) and (4.3,-0.39) .. (0,0) .. controls (4.3,0.39) and (9.04,1.82) .. (14.21,4.28)   ;
%Straight Lines [id:da8629009471688103] 
\draw [color={rgb, 255:red, 0; green, 0; blue, 0 }  ,draw opacity=1 ][line width=.5]    (412,395) -- (499.98,298.22) ;
\draw [shift={(502,296)}, rotate = 132.27] [color={rgb, 255:red, 0; green, 0; blue, 0 }  ,draw opacity=1 ][line width=.5]    (14.21,-4.28) .. controls (9.04,-1.82) and (4.3,-0.39) .. (0,0) .. controls (4.3,0.39) and (9.04,1.82) .. (14.21,4.28)   ;
%Straight Lines [id:da9243886494638835] 
\draw [color={rgb, 255:red, 0; green, 0; blue, 0 }  ,draw opacity=1 ][line width=.5]    (412,395) -- (586.39,295.49) ;
\draw [shift={(589,294)}, rotate = 150.29] [color={rgb, 255:red, 0; green, 0; blue, 0 }  ,draw opacity=1 ][line width=.5]    (14.21,-4.28) .. controls (9.04,-1.82) and (4.3,-0.39) .. (0,0) .. controls (4.3,0.39) and (9.04,1.82) .. (14.21,4.28)   ;
%Straight Lines [id:da47057991331036386] 
\draw [color={rgb, 255:red, 0; green, 0; blue, 0 }  ,draw opacity=1 ][line width=.5]    (91,213.5) -- (46.52,137.59) ;
\draw [shift={(45,135)}, rotate = 59.63] [color={rgb, 255:red, 0; green, 0; blue, 0 }  ,draw opacity=1 ][line width=.5]    (14.21,-4.28) .. controls (9.04,-1.82) and (4.3,-0.39) .. (0,0) .. controls (4.3,0.39) and (9.04,1.82) .. (14.21,4.28)   ;
%Straight Lines [id:da9904323319764552] 
\draw [color={rgb, 255:red, 0; green, 0; blue, 0 }  ,draw opacity=1 ][line width=.5]    (588,229) -- (628.02,112.84) ;
\draw [shift={(629,110)}, rotate = 109.01] [color={rgb, 255:red, 0; green, 0; blue, 0 }  ,draw opacity=1 ][line width=.5]    (14.21,-4.28) .. controls (9.04,-1.82) and (4.3,-0.39) .. (0,0) .. controls (4.3,0.39) and (9.04,1.82) .. (14.21,4.28)   ;
%Straight Lines [id:da20408619878087064] 
\draw [color={rgb, 255:red, 0; green, 0; blue, 0 }  ,draw opacity=1 ][line width=.5]    (588,229) -- (523.46,112.62) ;
\draw [shift={(522,110)}, rotate = 60.99] [color={rgb, 255:red, 0; green, 0; blue, 0 }  ,draw opacity=1 ][line width=.5]    (14.21,-4.28) .. controls (9.04,-1.82) and (4.3,-0.39) .. (0,0) .. controls (4.3,0.39) and (9.04,1.82) .. (14.21,4.28)   ;
%Straight Lines [id:da8921746008252889] 
\draw [color={rgb, 255:red, 0; green, 0; blue, 0 }  ,draw opacity=1 ][line width=.5]    (503,225) -- (418.94,126.28) ;
\draw [shift={(417,124)}, rotate = 49.59] [color={rgb, 255:red, 0; green, 0; blue, 0 }  ,draw opacity=1 ][line width=.5]    (14.21,-4.28) .. controls (9.04,-1.82) and (4.3,-0.39) .. (0,0) .. controls (4.3,0.39) and (9.04,1.82) .. (14.21,4.28)   ;
%Straight Lines [id:da021342442309055798] 
\draw [color={rgb, 255:red, 0; green, 0; blue, 0 }  ,draw opacity=1 ][line width=.5]    (226.5,214) -- (197.09,138.79) ;
\draw [shift={(196,136)}, rotate = 68.64] [color={rgb, 255:red, 0; green, 0; blue, 0 }  ,draw opacity=1 ][line width=.5]    (14.21,-4.28) .. controls (9.04,-1.82) and (4.3,-0.39) .. (0,0) .. controls (4.3,0.39) and (9.04,1.82) .. (14.21,4.28)   ;
%Straight Lines [id:da5486117794094123] 
\draw [color={rgb, 255:red, 0; green, 0; blue, 0 }  ,draw opacity=1 ][line width=.5]    (226.5,214) -- (264.79,126.75) ;
\draw [shift={(266,124)}, rotate = 113.7] [color={rgb, 255:red, 0; green, 0; blue, 0 }  ,draw opacity=1 ][line width=.5]    (14.21,-4.28) .. controls (9.04,-1.82) and (4.3,-0.39) .. (0,0) .. controls (4.3,0.39) and (9.04,1.82) .. (14.21,4.28)   ;
%Straight Lines [id:da8957834845348327] 
%\draw [color={rgb, 255:red, 155; green, 155; blue, 155 }  ,draw opacity=1 ][line width=.5.5]    (164,214) -- (109.6,127.54) ;
%\draw [shift={(108,125)}, rotate = 57.82] [color={rgb, 255:red, 155; green, 155; blue, 155 }  ,draw opacity=1 ][line width=.5.5]    (14.21,-4.28) .. controls (9.04,-1.82) and (4.3,-0.39) .. (0,0) .. controls (4.3,0.39) and (9.04,1.82) .. (14.21,4.28)   ;
%Straight Lines [id:da8482129987269036] 
\draw [color={rgb, 255:red, 0; green, 0; blue, 0 }  ,draw opacity=1 ][line width=.5]    (91,213.5) -- (96.8,127.99) ;
\draw [shift={(97,125)}, rotate = 93.88] [color={rgb, 255:red, 0; green, 0; blue, 0 }  ,draw opacity=1 ][line width=.5]    (14.21,-4.28) .. controls (9.04,-1.82) and (4.3,-0.39) .. (0,0) .. controls (4.3,0.39) and (9.04,1.82) .. (14.21,4.28)   ;
%Straight Lines [id:da8511787378572135] 
\draw [color={rgb, 255:red, 0; green, 0; blue, 0 }  ,draw opacity=1 ][line width=.5]    (306,225) -- (277.83,126.88) ;
\draw [shift={(277,124)}, rotate = 73.98] [color={rgb, 255:red, 0; green, 0; blue, 0 }  ,draw opacity=1 ][line width=.5]    (14.21,-4.28) .. controls (9.04,-1.82) and (4.3,-0.39) .. (0,0) .. controls (4.3,0.39) and (9.04,1.82) .. (14.21,4.28)   ;
%Straight Lines [id:da030808171366762505] 
\draw [color={rgb, 255:red, 0; green, 0; blue, 0 }  ,draw opacity=1 ][line width=.5]    (377,214) -- (357.73,136.91) ;
\draw [shift={(357,134)}, rotate = 75.96] [color={rgb, 255:red, 0; green, 0; blue, 0 }  ,draw opacity=1 ][line width=.5]    (14.21,-4.28) .. controls (9.04,-1.82) and (4.3,-0.39) .. (0,0) .. controls (4.3,0.39) and (9.04,1.82) .. (14.21,4.28)   ;
%Straight Lines [id:da12960979974430398] 
\draw [color={rgb, 255:red, 0; green, 0; blue, 0 }  ,draw opacity=1 ][line width=.5]    (428,215) -- (408.65,127.93) ;
\draw [shift={(408,125)}, rotate = 77.47] [color={rgb, 255:red, 0; green, 0; blue, 0 }  ,draw opacity=1 ][line width=.5]    (14.21,-4.28) .. controls (9.04,-1.82) and (4.3,-0.39) .. (0,0) .. controls (4.3,0.39) and (9.04,1.82) .. (14.21,4.28)   ;

% Text Node
\draw (347,494) node [anchor=north west][inner sep=0.75pt]  [font=\large] [align=left] {$\emptyset$};
% Text Node
\draw (286,420) node [anchor=west][inner sep=0.75pt]  [font= \large] [align=left] {a};
% Text Node
\draw (86,242) node [anchor= west][inner sep=0.75pt]  [font=\large] [align=left] {a};
% Text Node
\draw (229,244) node [anchor= west][inner sep=0.75pt]  [font=\large] [align=left] {a};
% Text Node
\draw (375,274) node [anchor= west][inner sep=0.75pt]  [font=\large] [align=left] {a};
% Text Node
\draw (86,273) node [anchor= west][inner sep=0.75pt]  [font=\large] [align=left] {a};
% Text Node
\draw (142,258) node [anchor= west][inner sep=0.75pt]  [font=\large] [align=left] {a};
% Text Node
\draw (284,259) node [anchor= west][inner sep=0.75pt]  [font=\large] [align=left] {a};
% Text Node
\draw (513,255) node [anchor= west][inner sep=0.75pt]  [font=\large] [align=left] {a};
% Text Node
\draw (406,419) node [anchor= west][inner sep=0.75pt]  [font=\large] [align=left] {b};
% Text Node
\draw (229,273) node [anchor= west][inner sep=0.75pt]  [font=\large] [align=left] {b};
% Text Node %%%%%%%
\draw (315,257) node [anchor= west][inner sep=0.75pt]  [font=\large] [align=left] {b};
% Text Node
\draw (375,243) node [anchor= west][inner sep=0.75pt]  [font=\large] [align=left] {b};
% Text Node
\draw (427,242) node [anchor= west][inner sep=0.75pt]  [font=\large] [align=left] {b};
% Text Node
\draw (484,254) node [anchor= west][inner sep=0.75pt]  [font=\large] [align=left] {b};
% Text Node
\draw (566,254) node [anchor= west][inner sep=0.75pt]  [font=\large] [align=left] {b};
% Text Node
\draw (597,254) node [anchor= west][inner sep=0.75pt]  [font=\large] [align=left] {b};
% Text Node
\draw (426,273) node [anchor= west][inner sep=0.75pt]  [font=\large] [align=left] {b};
% Text Node
\draw (173,258) node [anchor= west][inner sep=0.75pt]  [font=\large] [align=left] {a};
% Text Node
\draw (38,52) node [anchor= west][inner sep=0.75pt]  [font=\large] [align=left] {a};
% Text Node
\draw (37,81) node [anchor= west][inner sep=0.75pt]  [font=\large] [align=left] {a};
% Text Node
\draw (185,50) node [anchor= west][inner sep=0.75pt]  [font=\large] [align=left] {a};
% Text Node
\draw (185,111) node [anchor= west][inner sep=0.75pt]  [font=\large] [align=left] {a};
% Text Node
\draw (96,68) node [anchor= west][inner sep=0.75pt]  [font=\large] [align=left] {a};
% Text Node
\draw (96,98) node [anchor= west][inner sep=0.75pt]  [font=\large] [align=left] {a};
% Text Node
\draw (125,98) node [anchor= west][inner sep=0.75pt]  [font=\large] [align=left] {a};
% Text Node %%%%%%
\draw (246,70) node [anchor= west][inner sep=0.75pt]  [font=\large] [align=left] {a};
% Text Node
\draw (337,80) node [anchor= west][inner sep=0.75pt]  [font=\large] [align=left] {a};
% Text Node
\draw (428,68) node [anchor= west][inner sep=0.75pt]  [font=\large] [align=left] {a};
% Text Node
\draw (664,81) node [anchor= west][inner sep=0.75pt]  [font=\large] [align=left] {a};
% Text Node
\draw (604,80) node [anchor= west][inner sep=0.75pt]  [font=\large] [align=left] {b};
% Text Node
\draw (635,80) node [anchor= west][inner sep=0.75pt]  [font=\large] [align=left] {b};
% Text Node
\draw (545,81) node [anchor= west][inner sep=0.75pt]  [font=\large] [align=left] {b};
% Text Node
\draw (515,81) node [anchor= west][inner sep=0.75pt]  [font=\large] [align=left] {b};
% Text Node
\draw (485,80) node [anchor= west][inner sep=0.75pt]  [font=\large] [align=left] {b};
% Text Node
\draw (397,67) node [anchor= west][inner sep=0.75pt]  [font=\large] [align=left] {b};
% Text Node
\draw (398,97) node [anchor= west][inner sep=0.75pt]  [font=\large] [align=left] {b};
% Text Node
\draw (337,48) node [anchor= west][inner sep=0.75pt]  [font=\large] [align=left] {b};
% Text Node
\draw (338,108) node [anchor= west][inner sep=0.75pt]  [font=\large] [align=left] {b};
% Text Node
\draw (275,67) node [anchor= west][inner sep=0.75pt]  [font=\large] [align=left] {b};
% Text Node
\draw (246,97) node [anchor= west][inner sep=0.75pt]  [font=\large] [align=left] {b};
% Text Node
\draw (185,79) node [anchor= west][inner sep=0.75pt]  [font=\large] [align=left] {b};
% Text Node
\draw (38,109) node [anchor= west][inner sep=0.75pt]  [font=\large] [align=left] {b};
% Text Node
\draw (278,471) node [anchor= west][inner sep=0.75pt]   [align=left] {(1, a)};
% Text Node
\draw (392,471) node [anchor= west][inner sep=0.75pt]   [align=left] {(1, b)};
% Text Node
\draw (128,345.5) node [anchor= west][inner sep=0.75pt]   [align=left] {(2, a)};
% Text Node
\draw (256,326.5) node [anchor= west][inner sep=0.75pt]   [align=left] {(2, b)};
% Text Node
\draw (305,346.5) node [anchor= west][inner sep=0.75pt]   [align=left] {(1, b)};
% Text Node
\draw (190,315.5) node [anchor= west][inner sep=0.75pt]   [align=left] {(1, a)};
% Text Node
\draw (358,370) node [anchor= west][inner sep=0.75pt]   [align=left] {(2, a)};
% Text Node
\draw (432,317.5) node [anchor= west][inner sep=0.75pt]   [align=left] {(2, b)};
% Text Node
\draw (490,320) node [anchor= west][inner sep=0.75pt]   [align=left] {(1, a)};
% Text Node
\draw (522,346.5) node [anchor= west][inner sep=0.75pt]   [align=left] {(1, b)};
% Text Node
\draw (32,183.5) node [anchor= west][inner sep=0.75pt]   [align=left] {(3, b)};
% Text Node
\draw (100,187) node [anchor= west][inner sep=0.75pt]   [align=left] {(2, a)};
% Text Node
\draw (175,186.5) node [anchor= west][inner sep=0.75pt]   [align=left] {(3, a)};
% Text Node
\draw (295,156.5) node [anchor= west][inner sep=0.75pt]   [align=left] {(2, b)};
% Text Node
%\draw (138,156.5) node [anchor= west][inner sep=0.75pt]   [align=left] {2a};
% Text Node
\draw (330,186.5) node [anchor= west][inner sep=0.75pt]   [align=left] {(3, b)};
% Text Node
\draw (246,180) node [anchor= west][inner sep=0.75pt]   [align=left] {(1, b)};
% Text Node
\draw (378,165) node [anchor= west][inner sep=0.75pt]   [align=left] {(1, a)};
% Text Node
\draw (483,185.5) node [anchor= west][inner sep=0.75pt]   [align=left] {(2, b)};
% Text Node
\draw (555,145.5) node [anchor= west][inner sep=0.75pt]   [align=left] {(1, b)};
% Text Node
\draw (614,186.5) node [anchor= west][inner sep=0.75pt]   [align=left] {(1, a)};

\end{tikzpicture}}

\caption{A portion of the colored immaculate poset $\mathfrak{P}_{A}$ on the alphabet $A = \{a,b\}$.}
\end{figure}
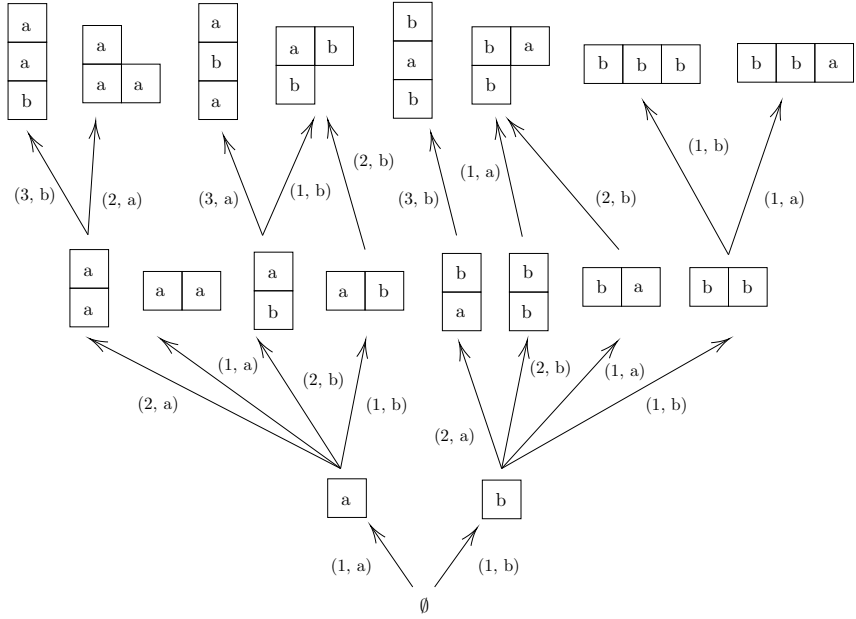

The maximal chains on $\mathfrak{P}_A$ from $\emptyset$ to $I$ are equivalent to the standard colored immaculate tableaux of shape $I$. The maximal chain $C = \{\emptyset = J_0 \xrightarrow{(m_1,a_1)} J_1 \xrightarrow{(m_2, a_2)} \cdots \xrightarrow{(m_k,a_k)} J_k = I  \}$ is associated with the standard colored immaculate tableau of shape $I$ whose boxes are filled with the integers $1$ through $n$ in the order they appear in the path. That is, the box added from $J_{j} \xrightarrow{(m_{j+1}, a_{j+1})} J_{j+1}$, which is added to row $m_{j+1}$ and colored with $a_{j+1}$, is filled with the integer $j+1$.

\begin{ex} The maximal chain
$C = \{\emptyset \xrightarrow{(1,a)}[a] \xrightarrow{(2,d)} [a,d] \xrightarrow{(2,e)} [a,de] \xrightarrow{(1,b)} [ab,de] \xrightarrow{(2,f)} [ab,def] \xrightarrow{(1,c)} [abc,def] \}$ is associated with the following tableaux: \vspace{1mm} 
\ytableausetup{boxsize=8mm}
$$\scalebox{.75}{
\begin{ytableau}
    \small a,1 & \small b,4 & \small c,6\\
    \small d,2 & \small e,3 & \small f,5
\end{ytableau}

}\vspace{2mm}$$
\end{ex}

Maximal chains starting from a non-empty sentence $J$ going to $I$ lead to a natural definition of \textit{skew standard colored immaculate tableaux}. 

\begin{defn}
    For sentences $I$ and $J=(v_1, \ldots, v_h)$ with $J \subseteq_L I$, the \emph{colored skew shape} $I/J$ is the colored composition diagram of $I$ where, for $1 \leq i \leq h$, the first $|v_i|$ boxes of the $i^{\text{th}}$ row are inactive. The inactive boxes are shaded gray to indicate that they have in a sense been ``removed'', however the colors filling them are still relevant.
\end{defn}

\begin{defn}
For sentences $I$ and $J$ with $J \subseteq_L I$, a \emph{skew colored immaculate tableau} of shape $I/J$ is a colored skew shape $I/J$ filled with integers such that the sequence of integer entries in the first column is strictly increasing from top to bottom and the sequence of integer entries in each row is weakly increasing from left to right. Here the inactive boxes of $I/J$ are not filled, and we consider the first column of a colored skew shape $I/J$ to be the column corresponding to the first column of $I$.
\end{defn}

The maximal chain $C = \{J = J_0 \xrightarrow{(m_1,a_1)} J_1 \xrightarrow{(m_2, a_2)} \cdots \xrightarrow{(m_k,a_k)} J_k = I  \}$ is associated with the skew standard colored immaculate tableau of shape $I/J$ whose boxes are filled with the integers $1, 
\ldots, k$ in the order they appear in the path. 

\begin{ex} The maximal chain 
$ C = \{[a,de] \xrightarrow{(1,b)} [ab,de] \xrightarrow{(2,f)} [ab,def] \xrightarrow{(1,c)} [abc,def] \}$ is associated with the following skew colored immaculate tableau:\vspace{1mm}
\ytableausetup{boxsize=8mm}
$$
\scalebox{.75}{

\begin{ytableau}
    *(gray!80) \small a & \small b,1 & \small c,3\\
    *(gray!80) \small d & *(gray!80) \small e & \small f,2
\end{ytableau}}$$\vspace{-1mm}
\end{ex}

\begin{defn}\label{dualskewdef}
For sentences $I,J$ such that $J \subseteq_L I$, define the \emph{skew colored dual immaculate function} as $$\mathfrak{S}^*_{I/J} = 
\sum_K \langle \mathfrak{S}_JH_K, \mathfrak{S}^*_I \rangle M_K,$$ where the sum runs over all sentences $K \in \mathfrak{P}_A$ such that $|I|-|J| = |K|$.
\end{defn}

\begin{prop}\label{coeff_skew_number}
The coefficient $\langle \mathfrak{S}_JH_K, \mathfrak{S}^*_I \rangle$ is equal to the number of skew colored immaculate tableaux of shape $I/J$ with type $K$. 
\end{prop}
\begin{proof} 
Let $K = (u_1, \ldots, u_g)$ be a sentence.  Notice that $\mathfrak{S}_JH_K= (((\mathfrak{S}_JH_{u_1})H_{u_2})\cdots H_{u_g})$ and by Theorem \ref{rightpieriI}, we have $$\mathfrak{S}_JH_K = \sum_{J \subseteq_{u_1} J_1 \subseteq_{u_2} \ldots J_{g-1} \subseteq_{u_g} L}  \mathfrak{S}_L,$$ for some sentences $J_1, \ldots J_{g-1}$. Thus, $$\langle \mathfrak{S}_JH_K, \mathfrak{S}^*_I \rangle = \left\langle \sum_{J \subseteq_{u_1} J_1 \subseteq_{u_2} \ldots J_{g-1} \subseteq_{u_g} L} \mathfrak{S}_L, \mathfrak{S}^*_I \right\rangle = \sum_{J \subseteq_{u_1} J_1 \subseteq_{u_2} \ldots J_{g-1} \subseteq_{u_g} L} \langle \mathfrak{S}_L, \mathfrak{S}^*_I \rangle$$ for some sentences $J_1, \ldots, J_{g-1}$. Therefore, for $J_1, \ldots, J_{g-1}$, this inner product is equivalent to the number of times that the sentence $I$ appears when summing over all sentences $L$ such that $J \subseteq_{u_1} J_1 \subseteq_{u_2} \ldots J_{g-1} \subseteq_{u_g} L$.  Each occurrence of $I$ can be associated with a unique sequence of sentences $(J, J_1, \ldots, J_{g-1})$ that appear in the sum, and each sequence can be associated with a unique skew colored immaculate tableau of shape $I/J$ and type $K$.  Starting with the colored skew shape $I/J$, first fill the boxes corresponding to those in $J_1 / J$ with $1$'s.  Then fill the boxes corresponding to $J_2 / J_1$ with 2's and continue repeating this process until the remaining boxes in $I/J_{g-1}$ are filled with $(g-1)$'s.  Note that because $J \subseteq_{u_1} J_1 \subseteq_{u_2} \ldots J_{g-1} \subseteq_{u_g} I$, the colors of the boxes filled with each number $j$, read from left to right and bottom to top, correspond exactly to the word $u_j$.  Through this construction, each sequence $J, J_1, \ldots, J_{g-1}$ corresponds to a unique skew colored immaculate tableau of shape $I/J$ and type $K$.  Additionally, each skew CIT $T$ of shape $I/J$ and type $K$ can be associated with a unique sequence $J, J_1, \ldots, J_{g-1}$ such that $J \subseteq_{u_1} J_1 \subseteq_{u_2} \ldots J_{g-1} \subseteq_{u_g} I$ by taking $T$ and removing all boxes filled with integers greater than $j$, for each $1 \leq j <g$, to get a colored tableau of shape $J_j$. Therefore, $\langle \mathfrak{S}_JH_K, \mathfrak{S}^*_I \rangle$ counts the number of skew CIT with shape $I/J$ and type $K$. 
\end{proof}

The use of linear functionals and properties of duality allows for the expansions of the skew colored dual immaculate functions into the colored fundamental basis and the colored dual immaculate basis with inner product coefficients.
\begin{prop} \label{skewcoeff} For an interval $[J,I]$ in $\mathfrak{P}_A,$

$$\mathfrak{S}^*_{I/J} = \sum_K \langle \mathfrak{S}_JR_K, \mathfrak{S}^*_I \rangle F_K = \sum_K \langle \mathfrak{S}_J\mathfrak{S}_K, \mathfrak{S}^*_I \rangle \mathfrak{S}^*_K,$$ where the sums run over all sentences $K$ such that $|I|-|J|=|K|$. The coefficients $\langle \mathfrak{S}_J\mathfrak{S}_K, \mathfrak{S}^*_I \rangle$ are equal to the structure coefficients $c^I_{J,K}$ for colored immaculate multiplication, $$\mathfrak{S}_J\mathfrak{S}_K = \sum_I c^{I}_{J,K}\mathfrak{S}_I = \sum_I \langle \mathfrak{S}_J\mathfrak{S}_K, \mathfrak{S}^*_I \rangle \mathfrak{S}_I,$$ where the sums run over all sentences $I$.
\end{prop}

\begin{proof}
Observe that by Definition $\ref{mperp1}$, $$\mathfrak{S}^*_{I/J} = \sum_K \langle \mathfrak{S}_J H_K, \mathfrak{S}^*_I \rangle M_K = \mathfrak{S}^{\perp}_J(\mathfrak{S}^*_I) = \sum_K \langle \mathfrak{S}_JR_K, \mathfrak{S}^*_I \rangle F_K = \sum_K \langle \mathfrak{S}_J\mathfrak{S}_K, \mathfrak{S}^*_I \rangle \mathfrak{S}^*_K. \qedhere$$ 
\end{proof}

The skew colored dual immaculate functions can also be defined explicitly in terms of skew colored immaculate tableaux following Definition \ref{dualskewdef}.

\begin{prop}
Let $I=(w_1,\ldots ,w_k)$ and $J=(v_1,\ldots ,v_h)$ be sentences such that $J \subseteq_L I$. Then
$$\mathfrak{S}^*_{I/J} = \sum_T x_T,$$ where the sum is taken over all skew colored immaculate tableaux of shape $I/J$.
\end{prop}

\begin{proof} By Definition \ref{dualskewdef}, $\mathfrak{S}^*_{I/J} = \sum_K \langle \mathfrak{S}_JH_K, \mathfrak{S}^*_I \rangle M_K,$ where the sum runs over all sentences $K \in \mathfrak{P}_A$.  By Proposition \ref{coeff_skew_number}, $\langle \mathfrak{S}_JH_K, \mathfrak{S}^*_I \rangle$ is equal to the number of skew colored immaculate tableaux of shape $I/J$ and type $K$. Thus, following Proposition \ref{typenumber}, $\langle \mathfrak{S}_JH_K, \mathfrak{S}^*_I \rangle M_K = \sum_{T'} x_{T'}$ where the sum runs over all skew CIT $T'$ of shape $I/J$ and flat type $K$. Therefore, $$\mathfrak{S}^*_{I/J} = \sum_K \sum_{T'}x_{T'}= \sum_{T}x_T$$ where the sums run over sentences $K$ such that $|I|-|J|=|K|$, skew CIT $T'$ of shape $I/J$ and flat type $K$, and all skew CIT $T$ of shape $I/J$ and type $T$.
\end{proof}

Additionally, comultiplication on the colored dual immaculate basis can be defined in terms of skew functions following Propositions \ref{hopf_dual_coproduct} and \ref{skewcoeff}.

\begin{prop} \label{color_comult} For a sentence $I$,
$$\Delta (\mathfrak{S}^*_I) = \sum_J \mathfrak{S}^*_J \otimes \mathfrak{S}^*_{I/J},$$ where the sum runs over all sentences $J$ such that $J \subseteq_L I$. 

\end{prop}
\begin{proof} Let $J$ and $K$ be sentences, and observe that $\mathfrak{S}_J\mathfrak{S}_K = \sum_I \langle \mathfrak{S}_J\mathfrak{S}_K, \mathfrak{S}^*_I \rangle \mathfrak{S}_I$. By Proposition \ref{hopf_dual_coproduct}, this implies 
\begin{align*}
    \Delta(\mathfrak{S}^*_{I}) &= \sum_{J,K} \langle \mathfrak{S}_J\mathfrak{S}_K, \mathfrak{S}^*_I \rangle \mathfrak{S}^*_J \otimes \mathfrak{S}^*_K = \sum_{J} \left(  \mathfrak{S}^*_J \otimes \sum_K \langle \mathfrak{S}_J\mathfrak{S}_K, \mathfrak{S}^*_I \rangle \mathfrak{S}^*_K \right)\\
    &= \sum_J \mathfrak{S}^*_J \otimes \mathfrak{S}^*_{I/J} \text{\quad by Proposition \ref{skewcoeff}.} \qedhere
\end{align*} 
\end{proof}

As in the non-colored case, finding general combinatorial formulas for multiplication and the antipode of the colored dual immaculate functions remains an open problem. As shown in the example below, the product of two colored dual immaculate functions does not have exclusively positive structure constants, and their combinatorial description is not yet evident. $$ \mathfrak{S}^*_{(ab)} \mathfrak{S}^*_{(c)} =  \mathfrak{S}^*_{(abc)} + \mathfrak{S}^*_{(c,ab)} + \mathfrak{S}^*_{(ac,b)} - \mathfrak{S}^*_{(a,bc)}$$

%%%%%%%%%%%%%%%
\section{A partially commutative generalization of the row-strict dual immaculate functions} \label{rowstrictsection}

Niese, Sundaram, Van Willigenburg, Vega, and Wang define a pair of dual bases in $QSym$ and $NSym$  in \cite{rowstrict} by applying an involution $\psi$ to the immaculate and dual immaculate bases.  The row-strict dual immaculate basis has extensive representation theoretic applications, specifically to 0-Hecke algebras \cite{row0hecke}. The combinatorics of this basis involve a variation of immaculate tableaux with different conditions on the rows and columns. Note that the original paper uses French notation for diagrams (the bottom row is row 1) so the definitions here have been adapted to English notation. We first review the theory of these two bases, then define their colored generalizations, and finally extend our earlier results using a lift of the original $\psi$.

\subsection{Row-strict immaculate and dual immaculate functions}

We begin by recalling several definitions and results from \cite{rowstrict}.

\begin{defn} Given a composition $\alpha$, a \emph{row-strict immaculate tableau} $T$ is a filling of the diagram of $\alpha$ such that the entries in the leftmost column weakly increase from top to bottom and the entries in each row strictly increase from left to right. A row-strict immaculate tableau with $n$ boxes is \emph{standard} if each integer $1$ through $n$ appears exactly once. The \emph{type} of a row-strict immaculate tableau is defined the same way as the type of an immaculate tableau. 
\end{defn}

Monomials are associated with row-strict immaculate tableaux according to their type in the same fashion as immaculate tableaux. 

\begin{defn}
For a composition $\alpha$, the \emph{row-strict dual immaculate function} is defined as $$\mathfrak{RS}^*_{\alpha} = \sum_T x^T,$$ where the sum runs over all row-strict immaculate tableaux $T$ of shape $\alpha$.
\end{defn}

 To standardize a row-strict immaculate tableau $T$, replace the $1$'s in $T$ with $1,2,\ldots$ moving left to right and top to bottom, then continue with the $2$'s, etc. Note also that the set of standard row-strict immaculate tableaux is the same as the set of standard immaculate tableaux.

\begin{defn}  A positive integer $i$ is a \emph{row-strict descent} of a standard row-strict immaculate tableau $U$ if $U$ contains the entry $i+1$ in a weakly higher row than entry $i$. The \textit{row-strict descent set} of a standard row-strict immaculate tableau $U$ is $$Des^{rs}(U) = \{i : i+1 \text{ is weakly above } i \text{ in } U\}.$$ The \emph{row-strict descent composition} of a standard row-strict immaculate tableaux $U$ is defined as $$co^{rs}(U) = comp(Des^{rs}(U)).$$
\end{defn}

Following these definitions, the row-strict dual immaculate function for a composition $\alpha$ can also be defined as
$$\mathfrak{RS}^*_{\alpha} = \sum_S F_{co^{rs}(S)},$$ where the sum is taken over all standard row-strict immaculate tableaux of shape $\alpha$.

\begin{defn}
The dual involutions $\psi: QSym \rightarrow QSym$ and $\psi: NSym \rightarrow NSym$ are defined

$$\psi(F_{\alpha})=F_{\alpha^c} \text{\quad and \quad} \psi(R_{\alpha})=R_{\alpha^c}.$$
\end{defn}

Note that there are two separate $\psi$ involutions, although they are often referred to together as if they are one map.

\begin{thm}\cite{rowstrict}
 Let $\alpha$ be a composition. Then,
$\psi(\mathfrak{S}^*_{\alpha})=\mathfrak{RS}^*_{\alpha^c}.$
\end{thm}

 Since $\psi$ is an involution and $\mathfrak{S}^*_{\alpha}$ is a basis, $\{\mathfrak{RS}^*_{\alpha} \}_{\alpha}$ is a basis for $QSym$.

 \begin{defn} For a composition $\alpha$ and weak composition $\beta$, let 
 $K^{rs}_{\alpha,\beta}$ be the number of row-strict immaculate tableaux of shape $\alpha$ and type $\beta$, and $L^{rs}_{\alpha,\beta}$ be the number of standard row-strict immaculate tableaux of shape $\alpha$ with row-strict descent composition $\beta$. 
 \end{defn}

\begin{thm} \cite{rowstrict}
For a composition $\alpha$, the row-strict dual immaculate function expands as $$\mathfrak{RS}^*_{\alpha} = \sum_{\beta}K^{rs}_{\alpha,\beta} M_{\beta} = \sum_{\gamma} L^{rs}_{\alpha, \gamma}F_{\gamma}. $$  where the sums run over compositions $\beta$ and $\gamma$ such that $|\beta|=|\alpha|$ and $|\gamma|=|\alpha|$.
\end{thm}

The row-strict dual immaculate functions have a dual basis that can be constructed similarly to the immaculate basis.  

\begin{defn} For $m \in \mathbb{Z}$, the \emph{noncommutative row-strict Bernstein operator} $\mathbb{B}^{rs}_m$  is defined by $$\mathbb{B}^{rs}_m = \sum_{i \geq 0} (-1)^iE_{m+i}F^{\perp}_{(i)},\text{\qquad and \qquad}  \mathbb{B}^{rs}_{\alpha} = \mathbb{B}^{rs}_{\alpha_1}\ldots\mathbb{B}^{rs}_{\alpha_k} \text{\qquad for $\alpha \in \mathbb{Z}^k$}.$$
For a composition $\alpha$, the \emph{row-strict immaculate function} $\mathfrak{RS}_{\alpha}$ is defined as $$\mathfrak{RS}_{\alpha} = \mathbb{B}^{rs}_{\alpha}(1).$$ These functions are dual to the row-strict dual immaculate basis, $\langle \mathfrak{RS}_{\alpha}, \mathfrak{RS}^*_{\beta} \rangle = \delta_{\alpha, \beta}$, and they are the image of the immaculate basis under $\psi$.
\end{defn}

Applying $\psi$ to various results from \cite{berg18} yields similar results for the row-strict immaculate and row-strict dual immaculate bases, which are summarized in the following result.

\begin{thm}\label{rowstrictresults} \cite{rowstrict}
For compositions $\alpha, \beta \models n$, $s\in \mathbb{Z}_{\geq 0}$, $m \in \mathbb{Z}$, and $f \in NSym$, \vspace{2mm}
\begin{itemize} 
    \item[(1)] \qquad \qquad \quad $\mathbb{B}_m(f)H_s = \mathbb{B}_{m+1}(f)H_{s-1} + \mathbb{B}_m(fH_s) \xLongleftrightarrow{\psi} \mathbb{B}_{m+1}^{rs}(f)E_{s-1} + \mathbb{B}^{rs}_m(fE_s).$ \vspace{2mm}
    \item[(2)] Multiplicity-free right Pieri rule: $$\mathfrak{S}_{\alpha}H_s = 
    \sum_{\alpha \subset_s \beta} \mathfrak{S}_{\beta} \xLongleftrightarrow{\psi} \mathfrak{RS}_{\alpha}E_s = \sum_{\alpha \subset_s \beta} \mathfrak{RG}_{\beta}.$$
    \item[(3)] Multiplicity-free right Pieri rule:
    $$\mathfrak{S}_{\alpha}\mathfrak{S}_{(1^s)}=\mathfrak{S}_{\alpha}E_s = \sum_{\beta}\mathfrak{S}_{\beta} \xLongleftrightarrow{\psi} \mathfrak{RS}_{\alpha}\mathfrak{RG}_{(1^s)} = \mathfrak{RS}_{\alpha}H_s = \sum_{\beta}\mathfrak{RS}_{\beta},$$ where the sum runs over compositions $\beta \models |\alpha|+s$ such that $\alpha_i \leq \beta_i \leq \alpha_i+1$ and $\alpha_i = 0$ for $i>\ell(\alpha)$. \vspace{0mm}
    \item[(4)] \qquad \quad $\mathfrak{S}_{(1^n)} = \sum_{\alpha \models n} (-1)^{n-\ell(\alpha)}H_{\alpha} = E_n \xLongleftrightarrow{\psi} \mathfrak{RG}_{(1^n)} = \sum_{\alpha \models n}(-1)^{n - \ell(\alpha)}E_{\alpha} = H_n.$ \vspace{2mm}
    \item[(5)] Complete homogeneous and elementary expansions: $$H_{\beta} = \sum_{\alpha \geq_{\ell} \beta}K_{\alpha,\beta}\mathfrak{S}_{\alpha} \xLongleftrightarrow{\psi} E_{\beta} = \sum_{\alpha \geq_{\ell} \beta}K_{\alpha,\beta}\mathfrak{RS}_{\alpha}, \text{\quad \quad} H_{\beta} = \sum_{\alpha \geq_{\ell} \beta} K^{rs}_{\alpha,\beta}\mathfrak{RS}_{\alpha} \xLongleftrightarrow{\psi} E_{\beta} = \sum_{\alpha \geq_{\ell} \beta}K^{rs}_{\alpha,\beta}\mathfrak{S}_{\alpha}.$$
    \item[(6)] Ribbon basis expansions: $$R_{\beta} = \sum_{\alpha \geq_{\ell} \beta}L_{\alpha,\beta} \mathfrak{S}_{\alpha} \xLongleftrightarrow{\psi} R_{\beta^c} = \sum_{\alpha \geq_{\ell} \beta}L_{\alpha, \beta}\mathfrak{RS}_{\alpha}.$$
\end{itemize}
\end{thm}

The immaculate poset also represents a poset of the standard row-strict immaculate tableaux as a result of the equivalence between standard immaculate tableaux and standard row-strict immaculate tableaux, thus results for the row-strict skew case follow closely to those of the dual immaculate functions. 

\begin{defn}Let $\alpha$ and $\beta$ be compositions with $\beta \subseteq \alpha$.  A \textit{skew row-strict immaculate tableau} is a skew shape $\alpha/\beta$ filled with positive integers such that the entries in the first column are weakly increasing from top to bottom and the entries in each row strictly increase from left to right.
\end{defn}

\begin{defn} For compositions $\alpha, \beta$ such that $\beta \subseteq \alpha$, the \textit{skew row-strict dual immaculate functions} are defined as$$\mathfrak{RS}^*_{\alpha/\beta} = \sum_{\gamma} \langle \mathfrak{RS}_{\beta}H_{\gamma}, \mathfrak{RS}^*_{\alpha}\rangle M_{\gamma},$$ where the sum runs over all $\gamma \in \mathfrak{P}$ such that $|\alpha|-|\beta| = |\gamma|$.
\end{defn}
As with the skew dual immaculate functions, these functions connect to the multiplication of the row-strict immaculate functions and the comultiplication of the row-strict dual immaculate functions.
\begin{thm} \cite{rowstrict}
Let $\alpha$ and $\beta$ be compositions with $\beta \subseteq \alpha$. Then, $$\mathfrak{RS}^*_{\alpha/\beta} = \sum_T x^T,$$ where the sum runs over all skew row-strict immaculate tableaux $T$ of shape $\alpha/\beta$. Moreover, $$\mathfrak{RS}^*_{\alpha/\beta} = \psi(\mathfrak{S}^*_{\alpha/\beta}) = \sum_{\gamma} \langle \mathfrak{RS}_{\beta}R_{\gamma}, \mathfrak{RS}^*_{\alpha} \rangle F_{\gamma} = \sum_{\gamma}\langle \mathfrak{RS}_{\beta}\mathfrak{RS}_{\gamma}, \mathfrak{RS}^*_{\alpha} \rangle \mathfrak{RS}^*_{\gamma},$$ where the sums run over all compositions $\gamma \in \mathfrak{P}$ such that $|\alpha|-|\beta| = |\gamma|$.
\end{thm}

Comultiplication on the row-strict dual immaculate functions is also defined in terms of skew shapes.

\begin{defn} For a composition $\alpha$,
$$\Delta (\mathfrak{RS}^*_{\alpha}) = \sum_{\beta} \mathfrak{RG}^*_{\beta} \otimes \mathfrak{RS}^*_{\alpha/\beta},$$ where the sum runs over all compositions $\beta$ such that $\beta \subseteq \alpha$.
\end{defn}

\subsection{Colored row-strict dual immaculate functions in $QSym_A$}

To generalize these definitions and results to the colored case, we first define a lift of the involution $\psi$ to $QSym_A$ and $NSym_A$. Note that we technically define two separate dual involutions $\psi$, one on $QSym_A$ and one on $NSym_A$, but we treat them as a single map that works on both spaces. 

\begin{defn} \label{color_psi}
    For a sentence $J$, define the linear maps $\psi: QSym_A \rightarrow QSym_A$ and $\psi: NSym_A \rightarrow NSym_A$ by $$\psi(F_J) = F_{J^c} \text{\quad and \quad} \psi(R_J)=R_{J^c}.$$
\end{defn}

\begin{prop}\label{rowstrictdual} The maps $\psi$ are involutions, and the duality between $QSym_A$ and $NSym_A$ is invariant under $\psi$, meaning that $$\langle G,F \rangle = \langle \psi(G), \psi(H) \rangle.$$ Furthermore, the map $\psi:NSym_A \rightarrow NSym_A$ is an isomorphism. 
\end{prop}

\begin{proof}
To see that $\psi$ is invariant under duality, it suffices to observe that $\langle R_I, F_J \rangle = \langle R_{I^c}, F_{J^c} \rangle = \langle \psi(R_I), \psi(F_J) \rangle$. The map $\psi$ is an involution because $\psi(\psi(F_I)) =F_{(I^c)^c}=F_I$ and $\psi(\psi(R_I)) = R_{(I^c)^c}=R_I$ and the map extends linearly. 
Next, we show that $\psi$ is an isomorphism on $NSym_A$.  For sentences $I$ and $J$, we have $R_I R_J = R_{I \cdot J} + R_{I \odot J}$ \cite{doliwa21} and thus $\psi(R_IR_J) = \psi(R_{I \cdot J}) + \psi(R_{I \odot J})$. Observe that $(I \cdot J)^c = I^c \odot J^c$ and $(I \odot J)^c = I^c \cdot J^c$. Therefore, $\psi(R_IR_J) = R_{I^c \odot J^c} + R_{I^c \cdot J^c} = R_{I^c}R_{J^c} = \psi(R_I)\psi(R_J).$
\end{proof}

Note that $\psi: QSym_A \rightarrow QSym_A$ is not an isomorphism because it fails to preserve multiplication. Now, we prove that $\psi$ maps the complete homogenous basis to the elementary basis in $NSym$ and vice versa, which will allow us to apply $\psi$ to both these bases.

\begin{prop}
For a sentence $J$, $\psi(E_J)=H_J$. 
\end{prop}
\begin{proof}
First, for a sentence $J$, we expand $E_J$ in terms of the colored ribbon basis as
\begin{align*}
    E_J &= \sum_{K \preceq J} (-1)^{|J| - \ell(K)}H_K = \sum_{K \preceq J} (-1)^{|J|-\ell(K)} \left[ \sum_{I \succeq K} R_I \right]  = \sum_{I} \left[ \sum_{K \preceq J, I} (-1)^{|J| - \ell(K)} \right] R_I.\\\intertext{ Next, we split the sum into two pieces according to $I$: one where $I \succ J^c$ and the other where $I \preceq J^c$,}
    E_J &= \sum_{I \succ J^c} \left[ \sum_{K \preceq J, I} (-1)^{|J| - \ell(K)} \right] R_I +  \sum_{I \preceq J^c}  \left[ \sum_{K \preceq J, I} (-1)^{|J| - \ell(K)} \right] R_I.\\
    \intertext{ In the first case, observe that $I \succ J^c$ implies that $J \succ I$. Thus, $K \preceq J, I$ becomes $K \preceq J$. Also notice that because $J$ is constant we can write $ (-1)^{|J|} = (-1)^{|J|-\ell(J)}(-1)^{\ell(J)}$ and factor the first term out of the sum.  In the second case, $I \preceq J^c$ means that $K \preceq I,J$ becomes $K \preceq J, J^c$. The only way for $K$ to be a refinement of a sentence and its complement is if $K$ is a sentence made up of only single letters. That is, $|K|=\ell(K)$.  Thus the inner sum has only one summand, which is $(-1)^{|J|-\ell(K)}=(-1)^{|J|-|K|}=1$.  As a result, the equation simplifies as } 
    E_J &= \sum_{I \succ J^c} (-1)^{|J|-\ell(J)} \left[ \sum_{J^c \preceq K \preceq I} (-1)^{\ell(J) - \ell(K)} \right] R_I + \sum_{I \preceq J^c} R_I. \\ \intertext{By properties of the Möbius function \cite{doliwa21}, the coefficient of first section is 0 for all $K$ and we are left with}
    E_J &= \sum_{I \preceq J^c} R_I.\\ \intertext{Therefore, applying $\psi$  to $E_J $ and noticing that $I \preceq J^c$ if and only if $J \preceq I^c$, yields}
    \psi(E_J)& = \sum_{I \preceq J^c} \psi(R_I) = \sum_{I \preceq J^c} R_{I^c} = \sum_{J \preceq I^c} R_{I^c} = H_J. \qedhere
\end{align*}
\end{proof}

We continue by defining and studying colored row-strict immaculate tableaux.  Their combinatorics in relation to those of the colored immaculate tableaux will allow us to define the colored row-strict dual immaculate basis and verify its relationship to the colored dual immaculate basis via $\psi$.

\begin{defn}
A \textit{colored row-strict immaculate tableau} (CRSIT) of shape $I$ is a colored composition diagram of shape $I$ in which the sequence of integer entries is strictly increasing from left to right in each row, and weakly increasing top to bottom in the leftmost column. The \emph{type} of a colored row-strict immaculate tableau $T$ is the sentence $C = (u_1, \ldots, u_g)$ such that for each $i \in [g]$ the word $u_i$ lists the colors of all boxes in $T$ filled with the integer $i$ in the order they appear when entries in $T$ are read from left to right and top to bottom. A \textit{standard colored row-strict immaculate tableau} is a colored row-strict immaculate tableau of size $n$ with the integer entries $1, \ldots, n$ each appearing exactly once. To \emph{standardize} a colored row-strict tableau, replace its integer entries with the numbers $1,2, \ldots$ based on the order they appear in the type, first replacing all entries equal to $1$, then $2$, etc. just as in the standardization of non-colored row-strict immaculate tableaux.
\end{defn}

We also use the same notion of \textit{row-strict descents} and the \textit{row-strict descent set} $Des^{rs}$ from row-strict immaculate tableaux, but define an additional concept of colored row-strict descent composition.

\begin{defn}
The \emph{colored row-strict descent composition} of a standard colored row-strict immaculate tableau $U$, denoted $co^{rs}_A(U)$, is the sentence obtained by reading the colors in each box in order of their number and splitting into a new word after each row-strict descent.
 \end{defn}

\begin{ex} A few CRSIT of shape $(ab,bca)$ along with their types and standardization, as well as the row-strict descent sets and colored row-strict descent compositions of these standardizations, are:\vspace{1mm}
$$
\scalebox{.75}{
$T_1 =$
\begin{ytableau}
    a, 1 & b, 2 \\
    b, 1 & c, 3 & a, 4 
\end{ytableau}
\quad \quad \quad 
$T_2 =$
\begin{ytableau}
    a, 1 & b, 3 \\
    b, 1 & c, 2 & a, 4 
\end{ytableau}

\quad \quad \quad 
$T_3 =$
\begin{ytableau}
    a, 1 & b, 4 \\
    b, 2 & c, 3 & a, 3 
\end{ytableau}
}$$
 \vspace{0mm}
$$ \quad \ \ (ab,b,c,a) \quad \quad \quad \quad \quad (ab,c,b,a) \quad \quad \quad \quad \quad \ (a,b,bc,a)$$ 

$$
\scalebox{.75}{
$U_2 = $
\begin{ytableau}
    a, 1 & b, 3 \\
    b, 2 & c, 4 & a, 5 
\end{ytableau}

\quad \quad \quad 
$U_3 = $
\begin{ytableau}
    a, 1 & b, 4 \\
    b, 2 & c, 3 & a, 5 
\end{ytableau}
\quad \quad \quad 
$U_4 = $
\begin{ytableau}
    a, 1 & b, 5 \\
    b, 2 & c, 3 & a, 4 
\end{ytableau}
}$$
\vspace{1mm}
$$
 \quad \quad \quad \substack{Des^{rs}(U_2) = \{2,4\}\\co^{rs}_A(U_2) = (ab,bc,a)} \quad \quad \quad \substack{Des^{rs}(U_3) = \{2,3\}\\co^{rs}_A(U_3) = (ab,c,ba)} \quad \quad \quad \substack{Des^{rs}(U_4) = \{2,3,4\}\\co^{rs}_A(U_4) = (ab,c,a,b)}$$
\end{ex}

\begin{defn} \label {crsimm_T}
 For a sentence J, the \emph{colored row-strict dual immaculate function} is defined as $$\mathfrak{RS}^*_J = \sum_{T}x_T,$$ where the sum is taken over all colored row-strict immaculate tableaux $T$ of shape $J$. 
 \end{defn}

 \begin{ex}
     For $J = (ab,bca)$, the colored row-strict dual immaculate function is
     $$\mathfrak{RS}^*_{ab,bca} = x_{ab,1}x_{b,2}x_{c,3}x_{a,4} + x_{ab,1}x_{c,2}x_{b,3}x_{a,4} + x_{a,1}x_{b,2}x_{bc,3}x_{a,4} + \ldots + 2x_{a,1}x_{b,2}x_{b,3}x_{c,4}x_{a,5} + \ldots.$$
 \end{ex}

\begin{prop}
For a sentence $J$, $$\mathfrak{RS}^*_J = \sum_{S} F_{co^{rs}_A(S)},$$ where the sum runs over all standard colored row-strict immaculate tableaux $S$ of shape $J$.
\end{prop}
\begin{proof}

Let $T$ be a colored row-strict immaculate tableau of shape $J$ that standardizes to the standard colored row-strict immaculate tableau $S$.  The flattening of the type of $T$ must be a refinement of the colored row-strict descent composition of $S$, which can be shown by applying the same reasoning used in the proof of Proposition \ref{standardrefinement}. In fact, each sentence $B$ that flattens to a refinement of $co^{rs}_A(S)$ corresponds to a unique colored row-strict immaculate tableau of type $B$ that standardizes to $S$. Therefore, $$F_{co^{rs}_A(S)} = \sum_{T_S} x_{T_S},$$ where the sum runs over all colored row-strict immaculate tableaux $T_S$ of shape $J$ that standardize to $S$. It follows that $$\mathfrak{RS}^*_J = \sum_T x_T = \sum_S \sum_{T_S} x_{T_S } = \sum_S F_{co^{rs}_A(S)},$$ where the sums run over all CRSIT $T$ of shape $J$, all standard CRSIT $S$ of shape $J$, and all CRSIT $T_S$ of shape $J$ that standardize to $S$.
\end{proof}

\begin{thm}\label{involutionbasis} Let J be a sentence. Then,
$$\psi(\mathfrak{S}^*_J) = \mathfrak{RS}^*_J.$$
\end{thm}

\begin{proof} For a sentence $J$, 
$$\psi(\mathfrak{S}^*_J) = \psi(\sum_{U} F_{co_A(U)}) = \sum_U F_{(co_A(U))^c}.$$ 
The complement of the colored descent composition of a standard colored immaculate tableau $U$ splits exactly where $U$ does not have a descent.  These are exactly the locations of the row-strict descents in $U$, thus $(co_A(U))^c = co^{rs}_A(U)$,  and
$$\psi(\mathfrak{S}^*_J) = \sum_U F_{(co^{rs}_A(U))} = \mathfrak{RS}^*_J. \eqno \qedhere $$ \end{proof}
Because $\{\mathfrak{S}^*_J\}_J$ is a basis and $\psi$ is an involution, Theorem \ref{involutionbasis} also implies the following.

\begin{cor}
    $\{\mathfrak{RS}^*_J \}_J$ is a basis for $QSym_A.$
\end{cor}

Using $\psi$, we extend each of our results on the colored dual immaculate functions to the colored row-strict dual immaculate functions. 

\begin{defn}
For sentences $J, C$ and weak sentence $B$, define $K^{rs}_{J,B}$ as the number of colored row-strict immaculate tableaux of shape J and type B, and $L^{rs}_{J,C}$ as the number of standard colored row-strict immaculate tableaux of shape J with row-strict descent composition C.
\end{defn}

\begin{prop} \label{rs_expansion}
    For a sentence $J$, $$\mathfrak{RS}^*_J = \sum_B K^{rs}_{J,B}M_B \text{\quad \quad and \quad \quad} \mathfrak{RS}^*_J = \sum_C L^{rs}_{J,C}F_C,$$ where the sums run over sentences $B$ and $C$ such that $|B|=|J|$ and $|C|=|J|$.
\end{prop}

The above proposition follows from Definition \ref{crsimm_T} in the manner of Theorem $\ref{monomialexpansion}$. The results of Section 4.3 also extend nicely to the row-strict case under the involution $\psi$.  

\begin{defn} The \emph{colored row-strict immaculate descent graph}, denoted $^{rs}\mathfrak{D}^{n}_A$ is the edge-weighted directed graph with the set of sentences on $A$ of size $n$ as its vertex set and an edge from each sentence $I$ to $J$ if there exists a standard colored row-strict immaculate tableau of shape $I$ with colored row-strict descent composition $J$. The edge from $I$ to $J$ is weighted with the coefficient $L^{rs}_{I,J}$.
\end{defn}

Due to the differing definitions of descents and descent compositions in row-strict tableaux, the neighbors of $I$ in $^{rs}\mathfrak{D}^{n}_A$ are exactly the (sentence) complements of $I$'s neighbors in $\mathfrak{D}^{n}_{A}$ and the complement of $I$ itself. Here, we say two vertices are \emph{neighbors} if they are adjacent by an edge in either direction.

\begin{ex} The standard colored row-strict immaculate tableaux of shape $(ab,cbb)$ have colored row-strict descent compositions $(a,bc,b,b)$, $(ac,bb,b)$, $(ac,b,bb)$, and $(ac,b,bb)$, so  $(ab,cbb)$ has outgoing edges to these sentences in $^{rs}\mathfrak{D}^5_{\{a,b,c\}}$. Notice that if we take the complement of each of these sentences we get $(ab,cbb)$, $(a,cb,bb)$, $(a,cbbb)$, and $(a,cbb,b)$ which are exactly $(ab,cbb)$ itself and the sentences to which it has outgoing edges to in $\mathfrak{D}^5_{\{a,b,c\}}$, as seen in Figure 1.
\end{ex}

\begin{prop}\label{rs_F_to_imm}
    For a sentence $I$, the colored fundamental functions expand into the colored row strict immaculate basis as $$F_I = 
    \sum_J L^{rs (-1)}_{I,J}\mathfrak{RS}^*_J \text{ \qquad with coefficients \qquad} L^{rs(-1)}_{I,J} = \sum_{\mathcal{P}}(-1)^{\ell(\mathcal{P})}prod(\mathcal{P}),$$ where the sums run over all sentences $J$ below $I$ in $^{rs}\mathfrak{D}^n_A$ and all directed paths $\mathcal{P}$ from $I$ to $J$ in $^{rs}\mathfrak{D}^n_A$.
\end{prop}

The proof follows that of Theorem \ref{color_fund_to_imm} using Proposition \ref{rs_expansion} in place of Theorem \ref{fund_exp}. Similarly, this proposition specializes to the non-colored case in the same manner as Corollary \ref{noncol_fund_to_imm}.

\subsection{Colored row-strict immaculate functions}

We define the colored row-strict immaculate functions as the image of the colored immaculate functions under $\psi$, and thus also as the basis dual to the colored row-strict dual immaculate functions.

\begin{defn} For a sentence $J$, the \emph{colored row-strict immaculate function} is defined as $$\mathfrak{RS}_J = \psi(\mathfrak{S}_J).$$
Equivalently, due to the invariance of $\psi$ under duality, we have $\langle \mathfrak{RS}_I, \mathfrak{RS}^*_J \rangle = \delta_{I,J}$.
\end{defn}

Applying $\psi$ to the colored immaculate functions yields row-strict versions of our earlier results and colored generalizations of the results in Theorem \ref{rowstrictresults}. Note that certain results from Theorem \ref{rowstrictresults} are not generalized here because we lack the corresponding result on the colored immaculate functions or due to the fact that $\psi$ is not an isomorphism on $QSym_A$. The non-colored analogues of $\psi$ are automorphisms on both $QSym$ and $NSym$.

\begin{thm} For words $w$ and $v$, sentences $J$ and $C$, and $f \in NSym_A$ \vspace{2mm}
\begin{itemize}
 \item[(1)] Right Pieri rule: $$\mathfrak{S}_{J}H_w = \sum_{J \subset_w K }\mathfrak{S}_{K} \xLongleftrightarrow{\psi} \mathfrak{RS}_{J}E_{w} = \sum_{J \subset_w K} \mathfrak{RS}_K.$$
 \item[(2)] Colored complete homogeneous and colored elementary expansions: $$H_C = \sum_J K_{J,C}\mathfrak{S}_J \xLongleftrightarrow{\psi} E_C = \sum_J K_{J,C} \mathfrak{RS}_J, \text{\quad \quad} H_C = \sum_J K^{rs}_{J,C}\mathfrak{RS}_J \xLongleftrightarrow{\psi} E_C = \sum_J K^{rs}_{J,C}\mathfrak{S}_J.$$
 \item[(3)] Colored ribbon expansions: $$R_C = \sum_J L_{J,C} \mathfrak{S}_J \xLongleftrightarrow{\psi} R_{C^c} = \sum_{J} L_{J,C}\mathfrak{RS}_J.$$
\end{itemize}
\end{thm}

The application of Proposition \ref{hopf_dual} to Proposition \ref{rs_F_to_imm} also yields the following result.  The analogous result is also true in $NSym$, as in Corollary \ref{imm_to_R}.

\begin{cor} For a sentence $J$, the colored row-strict immaculate functions expand into the colored ribbon basis as $$\mathfrak{RS}_J = \sum_I L^{rs(-1)}_{I,J} R_I \text{\qquad with coefficients \qquad} L^{rs(-1)}_{I,J} = \sum_{\mathcal{P}}(-1)^{\ell(\mathcal{P})}prod(\mathcal{P})$$ where the sums run over all $I$ above $J$ in $^{rs}\mathfrak{D}^n_{A}$ and all paths $\mathcal{P}$ from $I$ to $J$ in $^{rs}\mathfrak{D}^n_{A}$.
\end{cor}

\subsection{The colored immaculate poset and skew colored row-strict dual immaculate functions} The set of standard colored immaculate tableaux is equal to the set of standard colored row-strict immaculate tableaux, meaning that many of our results on the colored immaculate poset immediately extend to the row-strict setting.

\begin{defn}
For sentences $I$ and $J$ where $J \subseteq_L I$, a \emph{skew colored row-strict immaculate tableau} of shape $I/J$ is a colored skew shape $I/J$ filled with positive integers such that the sequences of entries in the first column is weakly increasing top to bottom and the sequence of integers in each row is strictly increasing left to right. 
\end{defn}

\begin{defn}
For sentences $I$ and $J$ where $J \subseteq_L I$, define the \textit{skew colored row-strict dual immaculate function} as $$\mathfrak{RS}^*_{I/J} = \sum_{K} \langle \mathfrak{RS}_JH_K, \mathfrak{RS}^*_I \rangle M_K,$$ where the sum runs over all sentences $K \in \mathfrak{P}_A$ such that $|I|-|J|=|K|$.
\end{defn}

Applying $\psi$ to the equations in Proposition \ref{skewcoeff} yields the following results.

\begin{thm}\label{rs_skew_coeff}
For sentences $I$ and $J$ with $J \subseteq_L I$, 
$$\mathfrak{RS}^*_{I/J} =  
\sum_K \langle \mathfrak{RS}_JR_K, \mathfrak{RS}^*_I \rangle F_K = \sum_K \langle \mathfrak{RS}_J\mathfrak{RS}_K, \mathfrak{RS}^*_{I} \rangle \mathfrak{RS}^*_K,$$ where the sums run over all sentences $K \in \mathfrak{P}_A$ such that $|I|-|J| = |K|$.
\end{thm}

\begin{prop} For sentences $I$ and $J$ such that $J \subseteq_L I$, $$\psi(\mathfrak{S}^*_{I/J})=\mathfrak{RS}^*_{I/J}.$$
\end{prop}
\begin{proof} Let $I$ and $J$ be sentences such that $J \subseteq_L I$. Then,
\begin{align*}
    \psi(\mathfrak{RS}^*_{I/J}) &= \sum_K \langle \mathfrak{RS}_JR_K, \mathfrak{RS}^*_I \rangle \psi(F_K) = \sum_K \langle \mathfrak{RS}_JR_K, \mathfrak{RS}^*_I \rangle F_{K^c}\\
    &= \sum_K \langle \psi(\mathfrak{S}_JR_{K^c}), \psi(\mathfrak{S}^*_I) \rangle F_{K^c} \text{\quad by Theorem \ref{involutionbasis}} \\
    &= \sum_K \langle \mathfrak{S}_JR_{K^c}, \mathfrak{S}^*_I \rangle F_{K^c} = \mathfrak{S}^*_{I/J}. \text{\quad By Proposition \ref{rowstrictdual}} \qedhere
\end{align*}
\end{proof}

Comultiplication on the colored row-strict immaculate basis can be defined in terms of skew colored row-strict immaculate functions.  The proof follows that of Proposition \ref{color_comult} using Theorem \ref{rs_skew_coeff}.
\begin{prop} Let $I$ be a sentence. Then,
    
$$\Delta(\mathfrak{RS}^*_I) = \sum_J \mathfrak{RS}^*_J \otimes \mathfrak{RS}^*_{I/J},$$ where the sum runs over all sentences $J$ such that $J \subseteq_L I$.
\end{prop}

Multiplication and antipode of the colored row-strict dual immaculate functions are closely related to the multiplication and antipode of colored dual immaculate functions, and thus also remain open.

\subsection{Future work} 
Our future work on this project will take three directions.  First, we hope to continue exploring properties of the colored immaculate and dual immaculate bases by looking at multiplicative structures, potential Jacobi-Trudi formulas, possible rim hook generalizations, and expansions to and from more bases.  Second, we will continue to generalize other Schur-like bases to $QSym_A$ and $NSym_A$, specifically the quasisymmetric shin functions and Young quasisymmetric Schur functions, as well as their duals. Finally, we are interested in defining and studying the colored generalization of the symmetric functions that would be a subset of $QSym_A$ and the image of $NSym_A$ under a forgetful map. 

%%%%%%%%%%%%%%%%%%%%%%%%%%%%%%%%%%%%%%%%%%%%%%%%%%%%%%%%%%%%%%%%%%%%%%%%%%%%%%%%%%%%%%%%%

\printbibliography
%%%%%%%%%%%%%%%%%%%%%%%%%%%%%%%%%%%%%%%%%%%%%%%%%%%%%%%%%%%%%%%%%%%%%%%%%%%%%%%%%%%%%%%%%

\section{Appendix}

Let $A$ be a finite alphabet with a total order $\leq$. Let $I = (w_1, \ldots, w_k)$ and $J=(v_1, \ldots, v_h)$ be sentences and $K=(u_1, \ldots, u_g)$ a weak sentence.  Let $w = a_1 \ldots a_n$ and $v = b_1 \ldots b_m$. Let $U$ be a standard colored immaculate tableau and thus also a standard colored row-strict immaculate tableau.

\begin{itemize}
\item[] \
\item[] $|w| = n$ \quad (size)
\item[] $w \cdot v = a_1 \ldots a_n b_1 \ldots b_m$ \quad (concatenation of words)
\item[] $w \preceq_{\ell} v$ \quad if $a_i < b_i$ for the first positive integer $i$ such that $a_i \not= b_i$ \quad (lexicographic order)
\item[] \
\item[] $|I| = \sum_{i=1}^{k} |w_i|$ \quad(size) 
\item[] $\ell(I) = k$ \quad(length)
\item[] $w(I)= w_1 \cdot w_2 \cdots w_k$ \quad (maximal word)
\item[] $w\ell(I) = (|w_1|, |w_2|, \ldots, |w_k|)$ \quad (word lengths)
\item[] \
\item[] $J \preceq I$ \quad for $w(I)=w(J)$, if $J$ splits at each location that $I$ splits \quad (refinement order)
\item[] $\mu(J,I) = (-1)^{\ell(J)-\ell(I)} \text{ for } J \preceq I$ \quad (Möbius function on the poset of sentences ordered by $\preceq$)
\item[] $J \subset_w I$ \quad if $w_i = v_iq_i$ for $1 \leq i \leq k$ such that $q_k \cdots q_1 = w$ and $k \leq h+1$ where $v_{h+1}= \emptyset$ 
\item[] \
\item[] $I \cdot J = (w_1, \ldots, w_n, v_1, \ldots, v_m)$\quad (concatenation of sentences)
\item[] $I \odot J = (w_1, \ldots, w_{n-1}, w_n \cdot v_1, v_2, \ldots, v_m)$ \quad (near-concatenation)
\item[] \
\item[] $I^r = (w_k, \ldots, w_2, w_1)$ \quad (reversal)
\item[] $I^c$ \quad the unique sentence with $w(I^c)=w(I)$ that splits exactly where $I$ does not \quad (complement)
\item[] $\tilde{K}$ \quad the sentence obtained by removing all empty words from K \quad (flattening)
\item[] \
\item[] $K\subseteq_L I$ \quad there exists a weak sentence $I/_L K$ as defined below \quad (left-containment)
\item[] $I /_L K = (q_1, \ldots, q_t)$ such that $w_i = u_iq_i$ for all $i \in [k]$
\item[] \
\item[] $K \subseteq_R I$ \quad there exists a weak sentence $I /_R K$ as defined below \quad (right-containment)
\item[] $I /_R K = (q_1, \ldots, q_t)$ such that $w_i = q_iu_i$ for all $i \in [k]$
\item[] \
\item[] $\{M_I\}_I = \sum_{1 \leq j_1 < j_2 < \ldots < j_k} x_{w_1, j_1}x_{w_2, j_2} \cdots x_{w_k, j_k}$ \quad (the monomial basis of $QSym_A$)
\item[] $\{F_I\}_I = \sum_{J \preceq I} M_J$ \quad (the fundamental basis of $QSym_A$)
\item[] \
\item[] $\{H_I\}_I$ \quad (the complete homogeneous basis of $NSym_A$)
\item[] $\{R_I\}_I = \sum_{I \preceq J}(-1)^{\ell(J) - \ell(I)}H_J$ \quad (the ribbon basis of $NSym_A$)
\item[] $\{E_I\}_I = \sum_{J \preceq I} (-1)^{|I|-\ell(J)} H_J$ \quad (the elementary basis of $NSym_A$)
\item[] \
\item[] $\langle \cdot, \cdot \rangle: NSym \otimes QSym \rightarrow \mathbb{Q}$ \quad $\langle H_I, M_J \rangle = \delta_{I,J}$ (inner product)
\item[] \
\item[] \emph{Colored immaculate tableau (CIT)} \quad a colored composition diagram filled with integers such that rows are weakly increasing left to right and the first column is strictly decreasing top to bottom
\item[] $shape(T) = (w_1, \ldots, w_k)$ \quad where the colors of row $i$ read left to right give $w_i$
\item[] $type(T) = (v_1, \ldots, v_h)$ \quad $v_i =$ colors of boxes filled with $i$ read from left to right, bottom to top
\item[] \
\item[] $Des(U) = \{ i : \text{$i+1$ is on a strictly lower row of $U$ than $i$} \}$ \quad (descent set)
\item[] $Co_A(U)$ \quad a sentence formed by reading the color in each box of $U$ in order of number, splitting into a new word when moving to a strictly lower row  \quad (colored descent composition)
\item[] $std(T)$ \quad (the standardization of a tableau $T$)
\item[] \
\item[] $\{\mathfrak{S}^*_I\}_I$ \quad (the colored dual immaculate basis of $QSym_A$)\
\item[] $\{\mathfrak{S}_I\}_I$ \quad (the colored immaculate basis of $NSym_A$)\
\item[] \
\item[] $\mathfrak{P}_A$ \quad (the colored immaculate poset)
\item[] $\mathfrak{D}_A^n$ \quad (the colored immaculate descent graph)
\item[] $^{rs}\mathfrak{D}_A^n$ \quad (the colored row-strict immaculate descent graph)
\item[] \
\item[] $Des^{rs}(U) = \{ i : \text{$i+1$ is on a weakly higher row of $U$ than $i$}\}$ \quad (row-strict descent set)
\item[] $Co^{rs}_A(U)$ \quad a sentence formed by reading the color in each box of $U$ in order of number, splitting into a new word when moving to a weakly higher row \quad (colored row-strict descent composition)
\item[] \
\item[] $\{\mathfrak{RS}^*_I\}_I$ \quad (the row-strict dual immaculate basis of $QSym_A$)
\item[] $\{\mathfrak{RS}_I\}_I$ \quad (the row-strict immaculate basis of $NSym_A$)
\end{itemize}

\end{document}